\documentclass{amsart}

\usepackage{amsmath}
\usepackage[nohug,heads=littlevee]{diagrams}
\usepackage{mathrsfs}
\usepackage{bm}
\usepackage{amssymb}
\usepackage[latin1]{inputenc}
\usepackage{yfonts}
\usepackage{xr-hyper}
\usepackage[pagebackref=true,plainpages=false,pdfpagelabels]{hyperref}
\usepackage[nobysame,alphabetic]{amsrefs}
\usepackage{mathtools}
\usepackage{enumitem}

\newcommand{\defnword}[1]{\textbf{#1}}
\newcommand{\comment}[1]{}
\newcommand{\on}[1]{\operatorname{#1}}
\newcommand{\bb}[1]{\mathbb{#1}}
\newcommand{\mb}[1]{\mathbf{#1}}
\newcommand{\mf}[1]{\mathfrak{#1}}
\newcommand{\ord}{\on{ord}}
\newcommand{\loc}{\on{loc}}

\numberwithin{equation}{subsection}

\newtheorem{prp}[subsection]{Proposition}
\newtheorem{thm}[subsection]{Theorem}
\newtheorem*{thm*}{Theorem}
\newtheorem{lem}[subsection]{Lemma}
\newtheorem*{lem*}{Lemma}
\newtheorem{corollary}[subsection]{Corollary}
\theoremstyle{definition}
\newtheorem{defn}[subsection]{Definition}

\newtheorem*{claim}{Claim}

\theoremstyle{remark}

\newtheorem{rem}[subsection]{Remark}
\newtheorem*{assump*}{Assumption}

\newarrow{Equals}{=}{=}{}{=}{}

\newcommand{\Real}{\mathbb{R}}
\newcommand{\Int}{\mathbb{Z}}

\newcommand{\Comp}{\mathbb{C}}
\newcommand{\Adele}{\mathbb{A}}
\newcommand{\Aff}{\mathbb{A}}
\newcommand{\Field}{\mathbb{F}}
\newcommand{\Rat}{\mathbb{Q}}
\newcommand{\Lie}{\on{Lie}}

\newcommand{\rk}{\on{rk}}
\newcommand{\texte}[1]{\text{\emph{#1}}}

\newcommand{\coker}{\on{coker}}
\newcommand{\im}{\on{im}}
\newcommand{\gr}{\on{gr}}
\newcommand{\into}{\hookrightarrow}

\newcommand{\Art}{\on{Art}}

\newcommand{\Def}{\on{Def}}

\newcommand{\pr}{\bm{\pi}}

\newcommand{\Rg}{\mathscr{O}}
\newcommand{\Reg}[1]{\Rg_{#1}}

\newcommand{\Hom}{\on{Hom}}
\newcommand{\End}{\on{End}}

\newcommand{\Aut}{\on{Aut}}

\newcommand{\SHom}{\underline{\on{Hom}}}

\newcommand{\disc}{\on{disc}}

\newcommand{\Gal}{\on{Gal}}

\newcommand{\Proj}{\on{Proj}}

\newcommand{\Gm}{\mathbb{G}_m}

\newcommand{\Gmh}[1]{\mathbb{G}_{m,#1}}

\newcommand{\cris}{\on{cris}}
\newcommand{\dR}{\on{dR}}

\newcommand{\Bcris}{B_{\cris}}
\newcommand{\Bdr}{B_{\dR}}

\newcommand{\pow}[1]{[\vert#1\vert]}

\newcommand{\Spec}{\on{Spec}}
\newcommand{\Spf}{\on{Spf}}

\newcommand{\et}{\on{\acute{e}t}}

\newcommand{\KS}{\mathrm{KS}}

\newcommand{\Fr}{\on{Fr}}

\newcommand{\mx}{\mathfrak{m}}

\newcommand{\rank}{\on{rank}}

\newcommand{\fdiff}[1]{\hat{\Omega}^1_{#1}}
\newcommand{\dual}[1]{{#1}^{\vee}}

\newcommand{\Sh}{\on{Sh}}
\newcommand{\Ss}{\mathscr{S}}

\newcommand{\Res}{\on{Res}}
\newcommand{\an}{\on{an}}

\newcommand{\tr}{\on{Tr}}
\newcommand{\trd}{\on{Trd}}
\newcommand{\Tr}{\on{Tr}}

\newcommand{\GSp}{\on{GSp}}

\newcommand{\GL}{\on{GL}}

\newcommand{\SO}{\on{SO}}

\newcommand{\GSpin}{\on{GSpin}}

\newcommand{\J}{\mathcal{J}}

\begin{document}
\title[Regular integral models]{Integral canonical models for Spin Shimura varieties}


\author{Keerthi Madapusi Pera}
\email{keerthi@math.harvard.edu}
\address{Department of Mathematics\\%
1 Oxford St\\%
Harvard University\\%
Cambridge, MA 02118\\%
USA}


\begin{abstract}
We construct regular integral canonical models for Shimura varieties attached to Spin and orthogonal groups at (possibly ramified) primes $p>2$ where the level is not divisible by $p$. We exhibit these models as schemes of 'relative PEL type' over integral canonical models of larger Spin Shimura varieties with good reduction at $p$. Work of Vasiu-Zink then shows that the classical Kuga-Satake construction extends over the integral model and that the integral models we construct are canonical in a very precise sense. Our results have applications to the Tate conjecture for K3 surfaces, as well as to Kudla's program of relating intersection numbers of special cycles on orthogonal Shimura varieties to Fourier coefficients of modular forms.
\end{abstract}

\maketitle

\section*{Introduction}\label{sec:intro}
The objects of study in this paper are certain Shimura varieties attached to GSpin and special orthogonal groups. More precisely, we will consider the $\Rat$-algebraic stack $\Sh_{K_{\on{max}}}\coloneqq\Sh_{K_{\on{max}}}(G,X)$, where $G=\GSpin(V,Q)$ is attached to a quadratic space $(V,Q)$ over $\Rat$ of signature $(n,2)$ and $X$ is the space of oriented negative definite planes in $V_{\Real}$. The level sub-group $K_{\on{max}}\subset G(\Adele_f)$ is attached to a lattice $L\subset V$: it is the intersection of $G(\Adele_f)$ with $C(L\otimes\widehat{\Int})^\times$, the unit group of the Clifford algebra of $L\otimes\widehat{\Int}$. The image of $K_{\on{max}}$ in $\SO(V)(\Adele_f)$ is the \defnword{discriminant kernel}: the largest sub-group of $\SO(L)(\widehat{\Int})$ that acts trivially on the discriminant $\dual{L}/L$. Here, $\dual{L}\subset V$ is the dual lattice for $L$.

Our results can be summarized by:

\begin{thm*}\label{intro:thm:main}
Assume that $L$ has square-free discriminant. Then, over $\Int\left[\frac{1}{2}\right]$, $\Sh_{K_{\on{max}}}$ admits a regular canonical model with a regular compactification.
\end{thm*}

The precise meaning of the word `canonical' is explained in \S~\ref{sec:self-dual}. In the body of the paper, we will isolate a prime $p>2$, and work with finite quasi-projective covers $\Sh_{K}\to\Sh_{K_{\on{max}}}$ where $K\subset K_{\on{max}}$ is of the form $K_pK^p$, with $K_p=G(\Rat_p)\cap C(L\otimes\Int_p)^\times$ and $K^p\subset G(\Adele_f^p)$ a small enough compact open. We will then build canonical models over $\Int_{(p)}$ for the tower of such covers.

The basic idea of the proof is quite simple, and involves exhibiting the model as the solution to a `relative PEL' problem\footnote{We thank an anonymous referee for pointing us to \cite{vasiu:orthoii}, where Vasiu has introduced the same terminology in a different context.} over the smooth integral model of a larger Shimura variety attached to a self-dual lattice. To explain this, we work over $\Int_{(p)}$ for some odd prime $p$. When the lattice $L$ is self-dual at a prime $p$, $K_p$ is hyperspecial, and the results of \cite{kis3} already give us a smooth integral canonical model over $\Int_{(p)}$ (compactifications are dealt with in \cite{mp:toroidal}). In general, we can exhibit $L$ as a sub-lattice of a bigger lattice $\widetilde{L}$ that is self-dual at $p$, and is such that the associated quadratic space $(\widetilde{L},\widetilde{Q})$ over $\Rat$ again has signature $(m,2)$, for some $m\in\Int_{\geq 0}$. If $\Sh_{\widetilde{K}}$ is the Shimura variety attached to a $\widetilde{L}$ and a compact open $\widetilde{K}\subset\widetilde{G}(\Adele_f^p)$ containing $K$, this allows us to exhibit $\Sh_K$ as an intersection of divisors in $\Sh_{\widetilde{K}}$. Let $\Ss_{\widetilde{K}}$ be the integral canonical model for $\Sh_{\widetilde{K}}$ over $\Int_{(p)}$. We show that the divisors have a moduli interpretation, and we use this interpretation to define models for them as schemes over $\Ss_{\widetilde{K}}$. We then construct our regular models as intersections (over $\Ss_{\widetilde{K}}$) of the integral models of these divisors.

The moduli interpretation of the divisors can be described as follows: The classical Kuga-Satake construction, combined with Kisin's construction, gives us a natural polarized abelian scheme $(\widetilde{A}^\KS,\widetilde{\lambda}^\KS)$ over $\Ss_{\widetilde{K}}$. Let $\widetilde{\bm{L}}_B$ be the $\Int_{(p)}$-local system over $\Ss_{\widetilde{K},\Comp}$ attached to $\widetilde{L}_{\Int_{(p)}}$. Let $\widetilde{\bm{H}}_B$ be the first Betti cohomology of $\widetilde{A}^\KS_{\Comp}$ over $\Ss_{\widetilde{K},\Comp}$ with coefficients in $\Int_{(p)}$. Then the construction allows us to view $\widetilde{\bm{L}}_B$ as a sub-local system of $\widetilde{\bm{H}}_B\otimes\dual{\widetilde{\bm{H}}}_B$. Given an $\Sh_{\widetilde{K}}$-scheme $T$, we say that an endomorphism $f$ of $\widetilde{A}^\KS_T$ is \defnword{special} if over $T_{\Comp}$ it induces a section of $\widetilde{\bm{L}}_B$. Then the divisors mentioned above can be viewed as the loci where $\widetilde{A}^\KS$ inherits certain special endomorphisms of fixed degree. So we obtain a moduli interpretation for $\Sh_K\to\Sh_{\widetilde{K}}$: It is the locus over which $\widetilde{A}^\KS$ inherits a certain family of special endomorphisms. For an analytic viewpoint of all this in the case of orthogonal Shimura varieties, cf.~\cite{kudla:algebraic}.

Let $\widetilde{\bm{H}}_{\cris}$ be the first crystalline cohomology of $\widetilde{A}^\KS_{\Field_p}$ over $\Ss_{\widetilde{K},\Field_p}$. Then Kisin's work provides us with a canonical sub-crystal $\widetilde{\bm{L}}_{\cris}\subset\widetilde{\bm{H}}_{\cris}\otimes\dual{\widetilde{\bm{H}}}_{\cris}$ attached to the quadratic space $\widetilde{L}$. This allows us to give a definition of specialness in characteristic $p$ as well: For any $\Ss_{\widetilde{K},\Field_p}$-scheme $T$, an endomorphism $f$ of $\widetilde{A}^\KS_T$ is \defnword{special} if its crystalline realization is a section of $\widetilde{\bm{L}}_{\cris}$ at every point of $T$. We can patch together the two notions of specialness to get the notion of a special endomorphism of $\widetilde{A}^\KS$ in general.

The moduli interpretation for $\Sh_K(G,X)$ as a scheme over $\Sh_{\widetilde{K}}(\widetilde{G},\widetilde{X})$ can now be extended over $\Int_{(p)}$ to obtain a natural integral model $\Ss_K(G,X)$. To study its local properties, we study the problem of deforming special endomorphisms of $\widetilde{A}^\KS$, using ideas of Deligne~\cite{deligne:k3liftings} and Ogus~\cite{ogus:ss}. These methods help us show that the quadric $\on{M}_G^{\loc}$ of isotropic lines in $L$ is an \'etale local model for $\Ss_K(G,X)$. In the case where the discriminant of $L$ is square-free, $\on{M}_G^{\loc}$ is regular, and this completes the construction.

The above construction actually works for general lattices $L$, but then it only gives us access to a certain open locus of the desired integral model, which we denote by $\Ss^{\on{pr}}_K(G,X)$: this is still modeled by $\on{M}_G^{\loc}$.  In general, the special fiber of this open locus will miss some very important parts of the expected `true' integral model. For instance, in the situation where $L$ is maximal, and $p$ is a prime such that $p^2$ divides the discriminant of $L$, $\Ss^{\on{pr}}_K(G,X)$ will entirely \emph{exclude} the part of the supersingular locus with maximal Artin invariant\footnote{We thank the referee for pointing out many subtleties arising in this situation.}; cf.~\eqref{regular:subsec:artin_invariant} for a discussion of this phenomenon. In particular, for $n=3$, where our Shimura variety is closely related to Siegel threefolds with parahoric level structure, our theory does \emph{not} recover the integral models of  Chai-Norman~\cite{chai_norman} and G\"ortz~\cite{gortz:siegel}: the simple local models that we describe are insufficient for this purpose. We expect to fix this gap in the theory in future work.

Along these lines, we should note that, for other low values of $n$, there are direct, moduli-theoretic ways to construct integral models. For the case of Shimura curves, cf.~\cite{kudrapyang}, and for the case of certain Hilbert-Blumenthal surfaces, cf.~\cite{deligne-pappas,kudrap:hilb}. Moreover, there is work in progress by Kisin and Pappas, which will generalize the methods of \cite{kis3} to construct integral models of Shimura varieties of abelian type with general parahoric level at $p$. However, the simple and direct nature of our construction seems to be quite useful in applications, which include the Tate conjecture for K3 surfaces in odd characteristic~\cite{mp:tate}, and also forthcoming work, in collaboration with F. Andreatta, E. Goren and B. Howard, on the arithmetic intersection theory over these Shimura varieties.

In the body of the paper, we also concurrently consider the case of orthogonal Shimura varieties. These are Shimura varieties of abelian type, but of a particularly simple nature: They are finite \'etale quotients of GSpin Shimura varieties, and we can easily deduce results for them from corresponding ones for their GSpin counterparts.

\subsection*{Notational conventions}
For any prime $\ell$, $\nu_{\ell}$ will be the $\ell$-adic valuation satisfying $\nu_{\ell}(\ell)=1$. $\Adele_f$ will denote the ring of finite ad\'eles over $\Rat$, and $\widehat{\Int}\subset\Adele_f$ will be the pro-finite completion of $\Int$. Given a rational prime $p$, $\Adele_f^p$ will denote the ring of prime-to-$p$ finite ad\'eles; that is, the restricted product $\prod'_{\ell\neq p}\Rat_{\ell}$. Moreover, $\widehat{\Int}^p\subset\Adele_f^p$ will be the closure of $\Int$. Unless otherwise specified, the letter $k$ will always represent a perfect field of characteristic $p$. Given such a $k$, $W(k)$ will denote its ring of Witt vectors, and $\sigma:W(k)\to W(k)$ will be the canonical lift of the Frobenius automorphism of $k$. For any group $G$, $\underline{G}$ will denote the locally constant \'etale sheaf (over a base that will be clear from context) with values in $G$.

\subsection*{Acknowledgments}
I thank Fabrizio Andreatta, Ben Howard, Mark Kisin, Ben Moonen, George Pappas and Junecue Suh for valuable comments, insights and corrections. Above all, I owe a deep debt to the referees for their careful reading and their numerous corrections and recommendations that immeasurably improved the paper, both in exposition and content. This work was partially supported by NSF Postdoctoral Research Fellowship DMS-1204165 and an AMS Simons Travel Grant.

\setcounter{tocdepth}{1}


\section{Clifford algebras and Spin groups}\label{sec:cliff}

We will quickly summarize some standard facts about Clifford algebras and the GSpin group. A good reference is~\cite{bourbaki_algebre_ix}*{\S 9}; cf. also~\cite{bass}.

\subsection{}
Let $R$ be a commutative ring in which $2$ is invertible, and let $(L,Q)$ be a quadratic space over $R$: by this we mean a projective $R$-module $L$ of finite rank equipped with a quadratic form $Q:L\to R$. We will denote by $[\cdot,\cdot]_Q:L\otimes L\rightarrow R$ the associated symmetric bi-linear form, its relation with $Q$ being given by the formula $[v,w]_Q=Q(v+w)-Q(v)-Q(w)$.

Let $C\coloneqq C(L)$ be the associated \defnword{Clifford algebra} over $R$. It is equipped with an embedding $L\into C$, which is universal for maps $f:L\to B$ into associative $R$-algebras $B$ satisfying $f(v)^2=Q(v)$. $C$ is equipped with a natural $\Int/2\Int$-grading $C=C^+\oplus C^-$, so that $C^+$ is a sub-algebra of $C$.

\subsection{}
Suppose now that $Q$ is \defnword{non-degenerate}: that is, it induces an isomorphism $L\xrightarrow{\simeq}\dual{L}$. In this situation, we will also call the module $L$ itself \defnword{self-dual}, especially when working over discrete valuation rings.

Then we can use the Clifford algebra to define a reductive group scheme $\GSpin(L,Q)$ over $R$. For any $R$-algebra $S$, we have:
\begin{align*}
  \GSpin(L,Q)(S)&=\bigl\{x\in (C_S^+)^\times:\;x(L_S)x^{-1}=L_S\bigr\}.
\end{align*}
It is a central extension of the special orthogonal group $\SO(L,Q)$ by $\Gm$, and there is a canonical character, the \defnword{spinor norm}, $\nu:\GSpin(L,Q)\to\Gm$. The spinor norm is defined as follows: There is a canonical anti-involution $*$ on $C$: it is the unique anti-automorphism that restricts to the identity on $L$; cf.~\cite{bourbaki_algebre_ix}*{\S 9}, where it is termed `\textit{l'antiautomorphisme principal}'. For any $x\in\GSpin(L,Q)(S)$, we set $\nu(x)=x^*x$. It follows from \cite{bass}*{3.2.1} that this is indeed an element of $S^\times\subset (C_S^+)^\times$.

$\GSpin(L,Q)$ acts naturally on $C$ via left multiplication, and we will denote the resulting representation by $H$. The right multiplication action of $C$ on itself endows $H$ with a $\GSpin(L,Q)$-equivariant right $C$-module structure, and the grading on $C$ endows it with a $\GSpin(L,Q)$-stable $\Int/2\Int$-grading $H=H^+\oplus H^-$. The action of $L$ on $C$ via left multiplication provides us with a $\GSpin(L,Q)$-equivariant embedding $L\into\End_C(H)$.

\subsection{}\label{cliff:subsec:tensors}
Let $\GL_R^+(H)\subset\GL_R(H)$ be the sub-group\footnote{For the rest of the section, the word `group' will mean `$R$-group scheme'} of automorphisms that preserve the grading on $H$: this is the group of units in the algebra $\End_R^+(H)$ of graded endomorphisms of $H$. Let $\on{U}(H)\subset\GL_R^+(H)$ be the centralizer of the right $C$-action. Then, for any $R$-algebra $S$, we have $\on{U}(H)(S)=\bigl(\End_{C}^+(H_S))^\times=(C^+_S)^\times$.
By definition, $\GSpin(L,Q)$ is a sub-group of $\on{U}(H)$.

The pairing $[\varphi_1,\varphi_2]=\frac{1}{2^n}\tr(\varphi_1\circ\varphi_2)$ on $\End_R(H)$ is symmetric, non-degenerate and restricts to the pairing $[\cdot,\cdot]_Q$ on $L\subset C$. Choose an $R$-basis $e_1, \ldots, e_m$ of $L$  and let $A\in \GL_m(R)$
be the matrix whose inverse is $(  [ e_i , e_j ]_Q )$. We will use the basis $\{e_i\}$ to identify $L$ with $R^m$ and hence to view $A$ as an automorphism of $L$. Consider the endomorphism $\pr:\End_{R}(H)\rightarrow\End_{R}(H)$ given, for $\varphi\in\End_{R}(H)$ by
\[
\pr(\varphi) =   \sum_i [\varphi,e_i]\cdot Ae_i.
\]
\begin{lem}\label{cliff:lem:pr}
\mbox{}
\begin{enumerate}
 \item\label{prequiv}$\pr$ is an idempotent endomorphism of $\End_R(H)$ with image $L\subset\End_R(H)$.
 \item\label{prind}$\pr$ is the unique projector onto $L$ satisfying:
 \[
 \ker\pr=\{\varphi\in\End_R(H):\;[\varphi,v]=0\text{, for all $v\in L$}\}.
 \]
 In particular, it is independent of the choice of basis.
 \item\label{gspintensors}$\GSpin(L,Q)\subset\on{U}(H)$ is the stabilizer of $\pr:\End_R(H)\rightarrow\End_R(H)$.
\end{enumerate}
\end{lem}
\begin{proof}
  For (\ref{prequiv}), we note:
  \begin{align*}
    \pr(\pr(\varphi))&=\sum_{i,k}[\varphi,e_i]\cdot [Ae_i,e_k]_Q\cdot Ae_k=\sum_{i,k}[\varphi,e_i]\cdot\delta_{i,k}\cdot Ae_k=\pr(\varphi).
  \end{align*}
  An easy computation also shows that, for $1\leq k\leq m$, $\pr(e_k)=e_k$, so that the image of $\pr$ is precisely $L$.

  (\ref{prind}) is clear.

  The stabilizer in $\on{U}(H)$ of $\pr$ must preserve $L$ under the conjugation action on $\End_R(H)$; it must therefore be contained in $\GSpin(L,Q)$. On the other hand, for any $g\in\GSpin(L,Q)$, since $g$ must preserve $L$, $g\pr g^{-1}$ is again a projector onto $L$ whose kernel agrees with that of $\pr$. Hence, we must have $g\pr g^{-1}=\pr$. This shows (\ref{gspintensors}).
\end{proof}

\begin{rem}\label{cliff:rem:nondeg}
In the above situation, suppose that $R=\Int_{(p)}$ and only that $(L_{\Rat},Q)$ is non-degenerate. If we choose a basis $\{e_1,\ldots,e_m\}$ for $L$, the corresponding matrix $A$ belongs to $\GL(L_{\Rat})$ and $A\cdot L\subset L_{\Rat}$ is precisely the dual lattice $\dual{L}$. Therefore, the image of $\End_{\Int_{(p)}}(H)$ under $\pr$ is exactly $\dual{L}\subset L_{\Rat}$.
\end{rem}

\subsection{}\label{cliff:subsec:ksemb}
As shown in Th\'eor\'emes 2 and 3 of \cite{bourbaki_algebre_ix}*{\S 9}, $C$ is an Azumaya algebra over its center $Z(C)$. Here, $Z(C)=R$, if $m$ is even, and is finite \'etale of rank $2$ over $R$, if $m$ is odd. Therefore, there exists an $R$-linear \defnword{reduced trace map} $\trd:C\to R$ such that the pairing $(x,y)\mapsto\trd(xy)$ is a non-degenerate symmetric bilinear form on $C$. For any $\delta\in C^\times$ such that $\delta^*=-\delta$, the form $\psi_{\delta}(x,y)=\trd(x\delta y^*)$ defines an $R$-valued symplectic form on $H$.
\begin{lem}\label{cliff:lem:psidelta}
For $\delta$ as above, the line $[\psi_{\delta}]$ in $\Hom(H\otimes H,R)$, spanned by the symplectic form $\psi_{\delta}$, is preserved by $\GSpin(L,Q)$. The similitude character of $\GSpin(L,Q)$ obtained from its action on this line agrees with the spinor norm.
\end{lem}
\qed

\begin{defn}\label{cliff:defn:tensors}
We will need one further piece of notation: For any pair of positive integers $(r,s)$, we set
\[
  H^{\otimes(r,s)}=\underbrace{H^\vee \otimes \cdots\otimes H^\vee}_{r\on{\ times}}
 \otimes  \underbrace{H \otimes \cdots\otimes H}_{s\on{\ times}}
\]
We will also use this notation for objects in other tensor categories without comment.
\end{defn}

Note that we can now think of $\pr$ as an element of $H^{\otimes(2,2)}$.

\subsection{}\label{cliff:subsec:parabolic}
Let $G_0=\SO(L,Q)$. Since there exists a central isogeny $G\to G_0$ of reductive groups, there is a bijective correspondence between parabolic sub-groups of $G$ and those of $G_0$. We want to make this correspondence explicit on the level of linear algebra, for certain parabolic sub-groups. To each isotropic sub-space $L^1\subset L$, we can attach the parabolic sub-group $P_0\subset G_0$ that stabilizes $L^1$. We get a decreasing filtration:
\[
 0=F^{2}L\subset F^1L=L^1\subset F^0L=(L^1)^{\perp}\subset F^{-1}L=L.
\]
Since $L^1$ is isotropic, we have a canonical embedding of $R$-algebras
\[
 \wedge^\bullet L^1\into C\into\End_R(H).
\]
If $N\subset\End_R(H)$, write $\im N$ for the union of the images in $H$ of the endomorphisms in $N$. Similarly, write $\ker N$ for the intersection of the kernels of the elements of $N$. Then we have, for every integer $i=0,\ldots,r+1$:
\[
 \im(\wedge^{i}L^1)=\ker(\wedge^{r-i+1}L^1).
\]
Moreover, $\im(\wedge^iL^1)\subset\im(\wedge^{i-1}L^1)$. So we can define a descending filtration $F^\bullet H$ on $H$ by
\[
 F^iH=\im(\wedge^iL^1).
\]

Suppose that $\mu_0:\Gm\rightarrow G_0$ is a co-character splitting $F^\bullet L$. It gives rise to a splitting
\[
 L=L^1\oplus L^0\oplus L^{-1},
\]
with $F^iL=\bigoplus_{j\geq i}L^j$, and where $\mu_0(z)$ acts on $L^i$ via $z^i$. In particular, $L^{-1}$ is another isotropic direct summand of $L$ that pairs non-degenerately with $L^1$.

Take the increasing filtration $E_iH=\ker(\wedge^{i+1}L^{-1})=\im(\wedge^{r-i}L^{-1})$, and set $H^i=E_iH\cap F^iH$. One easily checks that this is a splitting of $F^\bullet H$.

Let $\mu:\Gm\to\GL(H)$ be the co-character that acts via $z^i$ on $H^i$. By construction, $H^i$ is $C$-stable, and one can check easily that $\mu(\Gm)$ preserves the grading on $H$. So $\mu$ must factor through $\on{U}(H)$. Furthermore, we find that, if $v\in L$ and $i=0,1,\ldots,r$, then:
\[
 v\cdot H^i\subset\begin{cases}
   H^{i+1}&\text{ if $v\in L^1$};\\
   H^i&\text{ if $v\in L^0$};\\
   H^{i-1}&\text{ if $v\in L^{-1}$}.
 \end{cases}
\]
This shows that:
\[
 \mu(z)v\mu(z)^{-1}=\begin{cases}
   z\cdot v&\text{ if $v\in L^1$};\\
   v&\text{ if $v\in L^0$};\\
  z^{-1}\cdot v&\text{ if $v\in L^{-1}$}.
 \end{cases}
\]
In other words, $\mu$ factors through $G$ and is a lift of $\mu_0$. This shows the following:
\begin{prp}\label{cliff:prp:filtrations}
The parabolic sub-group $P\subset G$ lifting $P_0\subset G_0$ is the stabilizer in $G$ of $F^\bullet H$.
\end{prp}
\qed

\section{Local models}\label{sec:lattice}

We fix a prime $p>2$ for the rest of this paper.

\subsection{}
Suppose that $(M,q)$ is a non-degenerate quadratic space over a field $\kappa$, and let $N\subset M$ be a sub-space. Set:
\[
 \mf{a}_{N}=\{f\in\Hom(N,M):\;[f(v),w]_q+[v,f(w)]_q=0\text{, for all $v,w\in N$}\}.
\]
\begin{lem}\label{lattice:lem:surj}
The natural map $\Lie(\SO(M))\to\mf{a}_N$ induced by restriction to $N$ is surjective.
\end{lem}
\begin{proof}
  We will first consider three special cases:
  \begin{itemize}
  \item \emph{$q$ restricted to $N$ is also non-degenerate}: In this case, we have an orthogonal decomposition $M=N\perp N^{\perp}$. So, given $f\in\mf{a}_N$, we can extend it to the element of $\Lie(\SO(M))$ which restricts to $0$ on $N^{\perp}$.

  \item \emph{$N$ is isotropic}: In this case, we can find a splitting $M=N\oplus(N\oplus N')^{\perp}\oplus N'$, where $N'\subset M$ is also isotropic and pairs non-degenerately with $N$, and the restriction of $q$ to $(N\oplus N')^{\perp}$ is non-degenerate. Given $f\in\mf{a}_N$, we can write it in the form $f_1\oplus f_2\oplus f_3$, where $f_1:N\to N$, $f_2:N\to (N\oplus N')^{\perp}$ and $f_3:N\to N'$. The duality between $N'$ and $N$ allow us to view $\dual{f}_1$ as a map $N'\to N'$. Similarly, $\dual{f}_2$ can be viewed as a map $(N\oplus N')^{\perp}\to N'$, and $\dual{f}_3$ as a map $N\to N'$. The condition that $f\in\mf{a}_N$ simply means that $\dual{f}_3=-f_3$. We can now extend $f$ to an element $X\in\Lie(\SO(N))$, which restricts to $-\dual{f}_2$ on $(N\oplus N')^{\perp}$ and $-\dual{f}_1$ on $N'$.

  \item \emph{$N$ contains $\perp{N}$ the image of $f$ lies in $N^{\perp}$}: Again, we can find a splitting $M=N\oplus N'$ with $N=N^{\perp}\oplus (N^{\perp}\oplus N')^{\perp}$. Here $N'$ is isotropic and pairs non-degenerately with $N^{\perp}$. Now, $f$ is of the form $f_1\oplus 0\oplus 0$, where $f_1:(N^{\perp}\oplus N')^{\perp}\to N^{\perp}$. Extend this to the element $X\in\Lie(\SO(N))$ which restricts to $-\dual{f}_1:N'\to (N^{\perp}\oplus N')^{\perp}$ on $N'$ and to $0$ on $N$.
  \end{itemize}
  In general, let $N_0\subset N$ be the radical. Since $N_0$ is isotropic, by the second case above, we can find $X_1\in\Lie(\SO(M))$ such that $X_1\vert_{N_0}\equiv f\vert_{N_0}$. Therefore, replacing $f$ with $f-X_1\vert_N$, we can assume that $f\vert_{N_0}\equiv 0$. In this case, the image of $f$ must lie within $N_0^{\perp}$. Let $\overline{f}:N/N_0\to N_0^{\perp}/N_0$ be the induced map. Now, $N/N_0$ is a non-degenerate sub-space of the non-degenerate quadratic space $N_0^{\perp}/N_0$. Therefore, from the first case treated above, we can find $\overline{X}_2\in\Lie(\SO(N_0^{\perp}/N_0))$ such that $\overline{X}_2\vert_{N/N_0}\equiv\overline{f}$. Lift $\overline{X}_2$ to an element $X_2\in\Lie(\SO(M))$ that stabilizes $N_0$. Then replacing $f$ by $f-X_2\vert_{N}$, we can assume that the image of $f$ lands in $N_0$. Extend $f$ to any map $\tilde{f}:N_0^{\perp}\to N_0$. Then by the third case above, we can find $X\in\Lie(\SO(M))$ such that $X\vert_{N_0^{\perp}}\equiv\tilde{f}$.
\end{proof}

\subsection{}\label{lattice:subsec:maximal}
Suppose that we are given a discrete valuation ring $\Rg$ of mixed characteristic $(0,p)$ with residue field $k$, and a quadratic space $(M,q)$ over $\Rg$ such that $M\bigl[p^{-1}\bigr]$ is non-degenerate over the fraction field $\Rg\bigl[p^{-1}\bigr]$. The quadratic form $q$ endows the discriminant $\disc(M)=\dual{M}/M$ with a natural non-degenerate $\Rg\bigl[p^{-1}\bigr]/\Rg$-valued quadratic form $\overline{q}$.

We say that $M$ is \defnword{maximal} if it is maximal among $\Rg$-lattices in $M\bigl[p^{-1}\bigr]$ on which $q$ takes values in $\Rg$.

\begin{lem}\label{lattice:lem:maximal}
\mbox{}
\begin{enumerate}
\item\label{maximal:equiv}The following are equivalent:
\begin{enumerate}
  \item $M$ is maximal.
  \item $\disc(M)$ is a $k$-vector space and the form $\overline{q}$ is anisotropic.
\end{enumerate}
\item Suppose that $(\tilde{M},\tilde{q})$ is a self-dual quadratic space over $\Rg$ containing $(M,q)$ isometrically as a direct summand. Then $M$ is maximal if and only if $M^{\perp}\subset\tilde{M}$ is maximal.
\end{enumerate}
\end{lem}
\begin{proof}
 First, we claim that $M$ is maximal only if $\disc(M)$ is a $k$-vector space. Otherwise, if $\pi\in\Rg$ is a uniformizer, we can find $m\in\dual{M}$ such that $\pi\cdot m\notin M$ but $\pi^2\cdot m\in M$. But then $M+\langle \pi\cdot m\rangle$ is a lattice bigger than $M$ on which $q$ is $\Rg$-valued.

 To prove (\ref{maximal:equiv}), we can now assume that $\disc(M)$ is a $k$-vector space. Then the assignment $M'\mapsto M'/M$ sets up a bijection between the two following sets:
 \begin{itemize}
   \item Lattices $M'\subset M\bigl[p^{-1}]$ containing $M$ and on which $q$ takes values in $\Rg$.
   \item Sub-spaces of $\disc(M)$ that are isotropic for $\overline{q}$.
 \end{itemize}
 From this, the claimed equivalence is clear.

 For the final assertion, note that the identification $\tilde{M}\xrightarrow{\simeq}\dual{\tilde{M}}$ induced by $\tilde{q}$ gives us canonical isomorphisms:
 \[
  \disc(M)\xleftarrow{\simeq}\frac{\tilde{M}}{M+M^{\perp}}\xrightarrow{\simeq}\disc(M^{\perp}).
 \]
 This gives us an isomorphism $\disc(M)\xrightarrow{\simeq}\disc(M^{\perp})$ that preserves quadratic forms up to sign. We conclude now from (\ref{maximal:equiv}).
\end{proof}

\subsection{}\label{lattice:subsec:parahoric}

Let $(L,Q)$ be a quadratic space over $\Int_{(p)}$ such that $(L_{\Rat},Q)$ is non-degenerate. Suppose that we are given a self-dual quadratic space $(\widetilde{L},\widetilde{Q})$ over $\Int_{(p)}$ admitting $(L,Q)$ as a direct summand. Set $\Lambda=L^{\perp}\subset\widetilde{L}$. Set $\widetilde{G}=\GSpin(\widetilde{L},\widetilde{Q})$, and let $G\subset\widetilde{G}$ be the closed sub-group such that, for any $\Int_{(p)}$-algebra $R$, we have:
\[
  G(R)=\{g\in\widetilde{G}(R): g\rvert_{\Lambda_R}\equiv 1\}.
\]
Note that the central embedding $\Gmh{\Rat}\into G_{\Rat}=\GSpin(L_{\Rat},Q)$ is induced from an embedding $\Gm\into G$. Let $G_0$ (resp. $\widetilde{G}_0$) be the $\Int_{(p)}$-group scheme $G/\Gm$ (resp. $\widetilde{G}/\Gm$), so that $G_{0,\Rat}=\SO(L_{\Rat},Q)$.

\begin{lem}\label{lattice:lem:grpoints}
For any flat $\Int_{(p)}$-algebra $R$, we have:
    \begin{align*}
     G(R)&=G\bigl(R_{\Rat}\bigr)\cap C(L)_R^{\times}\subset C(L)_{R_{\Rat}}^\times;\\
     G_0(R)&=\left\{g\in G_0\left(R_{\Rat}\right):\;gL_R=L_R\text{ and $g$ acts trivially on $\dual{L}_R/L_R$}\right\}.
    \end{align*}
\end{lem}
\begin{proof}
  The obvious inclusion $L\into\widetilde{L}\into C(\widetilde{L})$ gives rise to a canonical map $C(L)\to C(\widetilde{L})$. We claim that this sets up an identification:
  \begin{align}\label{lattice:eqn:commutant}
   C(L)^+=\{z\in C(\widetilde{L})^+:\; vz=zv\in C(\widetilde{L})\text{, for all $v\in\Lambda$}\}.
  \end{align}
  Since both sides of this purported identity are saturated $\Int_{(p)}$-sub-modules of $C(\widetilde{L})^+$, it is enough to show:
  \begin{align}\label{lattice:eqn:commutant2}
   C(L_{\Rat})^+=\{z\in C(\widetilde{L}_{\Rat})^+:\; vz=zv\in C(\widetilde{L}_{\Rat})\text{, for all $v\in\Lambda$}\}.
  \end{align}
  This is easily checked using the orthogonal decomposition $\widetilde{L}_{\Rat}=L_{\Rat}\perp\Lambda_{\Rat}$.

  The description of $G(R)$ for a flat $\Int_{(p)}$-algebra $R$ is now clear, since both $G\bigl(R_{\Rat}\bigr)\cap C(L)_R^\times$ and $G(R)$ can be identified with the set of elements of $\widetilde{G}(R)$ that commute with the left multiplication action of $\Lambda$ on $C(\widetilde{L})_R$.

  Proving the corresponding description of $G_0(R)$ amounts to checking the following thing: Suppose that we have $g\in G_0\left(R_{\Rat}\right)$ such that $gL_R=L_R$. Then:
 \[
   g\widetilde{L}_R=\widetilde{L}_R\;\;\Leftrightarrow\;\;g\text{ acts trivially on }\dual{L}_R/L_R.
 \]

 Indeed, suppose $g\widetilde{L}_R=\widetilde{L}_R$. Given any $v\in\dual{L}_R$, there exists $w\in\dual{\Lambda}_R$ such that the element $v+w$ of $\widetilde{L}_{R_{\Rat}}=L_{R_{\Rat}}\oplus\Lambda_{R_{\Rat}}$ lies in $\widetilde{L}_R$. Therefore, since $g$ acts trivially on $\Lambda_R$, we have:
 \[
  gv-v=g(v+ w)-(v+ gw)=g(v+ w)-(v+ w)\in\widetilde{L}_R\cap L_{R_{\Rat}}=L_R.
 \]
 This shows that $g$ acts trivially on $\dual{L}_R/L_R$.

 On the other hand, suppose that $g$ acts trivially on $\dual{L}_R/L_R$. Every element $\tilde{v}\in\widetilde{L}_R$ can be written uniquely in the form $v+ w$, for some $v\in\dual{L}_R\subset L_{R_{\Rat}}$ and $w\in\dual{\Lambda}_R\subset\Lambda_{R_{\Rat}}$. We now have:
 \[
  g\tilde{v}=g(v+ w)=gv+ w=(gv-v)+v+ w=(gv-v)+\tilde{v}\in\widetilde{L}_{R_{\Rat}}.
 \]
 But, by our hypothesis on $g$, $gv-v\in L_R\subset\widetilde{L}_R$. Therefore, $g\tilde{v}\in\widetilde{L}_R$, for all $\tilde{v}\in\widetilde{L}_R$.
\end{proof}

\begin{lem}\label{lattice:lem:liealg}
Set
  \[
  \mf{a}_{\Lambda}=\{f\in\Hom(\Lambda,\widetilde{L}):\;[f(v),w]_{\widetilde{Q}}+[v,f(w)]_{\widetilde{Q}}=0\text{, for all $v,w\in\Lambda$}\}.
 \]
Then we have a short exact sequence of finite free $\Int_{(p)}$-modules:
\[
 0\to\Lie G_0\rightarrow\Lie\widetilde{G}_0\xrightarrow{X\mapsto X\vert_{\Lambda}}\mf{a}_{\Lambda}\to 0.
\]
\end{lem}
\begin{proof}
  The only non-obvious part is the surjectivity on the right hand side. By Nakayama's lemma, it suffices to prove this after tensoring with $\Field_p$. But then it follows from (\ref{lattice:lem:surj}).
\end{proof}

\begin{lem}\label{lattice:lem:torsor}
Write $\iota_0:\Lambda\into\widetilde{L}$ for the natural embedding. Let $R$ be a $\Int_{(p)}$-algebra. Suppose that we have another isometric embedding $\iota:\Lambda_R\into\widetilde{L}_R$ onto a local direct summand of $\widetilde{L}_R$. Then the functor $\mathcal{P}_{\iota}$ on $R$-algebras given by
    \[
      \mathcal{P}_{\iota}(B)=\bigl\{g\in\widetilde{G}_0(B): g\circ\iota_0=\iota\bigr\}
    \]
    is represented by a scheme that is affine, smooth and faithfully flat over $R$.
\end{lem}
\begin{proof}
It is clear that $\mathcal{P}_{\iota}$ is represented by a closed sub-scheme of $\widetilde{G}_{0,R}$. Moreover, Witt's extension theorem~\cite{bourbaki_algebre_ix}*{\S 4, Th. 1} shows that $\mathcal{P}_{\iota}(\kappa)$ is non-empty for any residue field $\kappa$ of $R$. To finish the proof of the proposition, it suffices to show that $\mathcal{P}_{\iota}$ is formally smooth over $R$.

 Let $B$ be an $R$-algebra, let $I\subset B$ be a square-zero ideal, and let $B_0=B/I$. Suppose that we have $g_0\in\mathcal{P}_{\iota}(B_0)$: We want to find $g\in \mathcal{P}_{\iota}(B)$ mapping to $g_0$. Choose any $\widetilde{g}\in\widetilde{G}_0(B)$ lifting $g_0$. Consider the assignment on $\Lambda_B$ given by $v\mapsto v-\widetilde{g}^{-1}\iota(v)$. This factors through a map $U:\Lambda_{B_0}\to I\otimes_{B_0}\widetilde{L}_{B_0}$:
 \begin{align*}
   \Lambda_B\to \Lambda_{B_0}&\xrightarrow{U} I\otimes_{B_0}\widetilde{L}_{B_0}\xrightarrow{\simeq}I\otimes_B\widetilde{L}_B\into\widetilde{L}_B.
 \end{align*}

 For $v,w\in\Lambda$, we have:
 \begin{align*}
  [U(v),w]_{\widetilde{Q}}+[v,U(w)]_{\widetilde{Q}}&=[v-\widetilde{g}^{-1}\iota(v),w]_{\widetilde{Q}}+[v,w-\widetilde{g}^{-1}\iota(w)]_{\widetilde{Q}}\\
  &=[v-\widetilde{g}^{-1}\iota(v),w-\widetilde{g}^{-1}\iota(w)]_{\widetilde{Q}}+[v,w]_{\widetilde{Q}}-[\widetilde{g}^{-1}\iota(v),\widetilde{g}^{-1}\iota(w)]_{\widetilde{Q}}\\
  &=[v-\widetilde{g}^{-1}\iota(v),w-\widetilde{g}^{-1}\iota(w)]_{\widetilde{Q}}=0.
 \end{align*}
 Here, the last equality holds because $I^2=0$.

 Therefore, $U\in I\otimes_{\Int_{(p)}}\mf{a}_{\Lambda}$, and, by (\ref{lattice:lem:liealg}), we can find $X\in I\otimes_{\Int_{(p)}}\Lie(\widetilde{G}_0)$ such that $X\vert_{\Lambda_{R_0}}\equiv U$. Then $g=\widetilde{g}\circ(1-X)$ is an element of $\mathcal{P}_{\iota}(B)$ that lifts $g_0$.
\end{proof}

\begin{prp}\label{lattice:prp:smoothparahoric}
The group schemes $G$ and $G_0$ are smooth and faithfully flat over $\Int_{(p)}$. In particular, up to (unique) isomorphism, they do not depend on the choice of self-dual quadratic space $\widetilde{L}$ containing $L$.
\end{prp}
\begin{proof}
 It is enough to show that $G_0$ is smooth over $\Int_{(p)}$. This follows from (\ref{lattice:lem:torsor}) with $R=\Int_{(p)}$ and $\iota=\iota_0$.

 Take $R=\Int_{(p)}^{nr}$ to be a strict henselization of $\Int_{(p)}$. That $G$ and $G_0$ are independent of the choice of $\widetilde{L}$ is immediate from (\ref{lattice:lem:grpoints}) and the following assertion~\cite{bruhat_tits_ii}*{1.7.6}: Given a smooth $\Rat$-scheme $X$, a smooth $\Int_{(p)}$-model $\mf{X}$ for $X$ is determined up to unique isomorphism by its set of $R$-valued points $\mf{X}(R)\subset X\left(R_{\Rat}\right)$.
\end{proof}

\subsection{}\label{lattice:subsec:naive}
Let $\on{M}^{\loc}_G$ be the $\Int_{(p)}$-scheme such that, for every $\Int_{(p)}$-algebra $R$, we have:
\[
 \on{M}^{\loc}_G(R)=\{\text{Isotropic lines $F^1L_R\subset L_R$\}}.\footnote{A line is an $R$-sub-module that is locally a direct summand of rank $1$.}
\]

\begin{lem}\label{lattice:lem:mloc}
Let $N\subset L_{\Field_p}$ be the radical. Set $r=\dim L$, $t=\dim N$ and $s=r-t-1$. We will assume that $N\neq L_{\Field_p}$; or, equivalently, $s\geq 0$.
\begin{enumerate}
\item\label{mlocflat}$\on{M}^{\loc}_G$ is flat, projective of relative dimension $r-2$ over $\Int_{(p)}$.
\item\label{mlocsing}The singular locus of $\on{M}^{\loc}_{G,\Field_p}$ consists of lines contained in $N$, and so can be identified with $\bb{P}(N)$. It has co-dimension $s$ in $\on{M}^{\loc}_{G,\Field_p}$. \footnote{Here, $\bb{P}(N)$ denotes the space of lines in $N$.}
\item\label{mlocnorm}$\on{M}^{\loc}_{G,\overline{\Field}_p}$ is an lci variety. It is reduced if and only if $s\geq 1$. It is normal if and only if $s\geq 2$, and smooth if and only if $t=0$.
\end{enumerate}
\end{lem}
\begin{proof}
  (\ref{mlocflat}) is a direct consequence of the hypothesis that $N\neq L_{\Field_p}$. (\ref{mlocsing}) is an easy fact about quadrics. Finally, $\on{M}^{\loc}_{G,\Field_p}$, being a quadric, is an lci variety. The remainder of (\ref{mlocnorm}) follows from (\ref{mlocsing}) and standard criteria for reducedness and normality.
\end{proof}

We now recall some definitions and results from \cite{vasiu:zink}.
\begin{defn}\label{lattice:defn:healthy}
A regular local $\Int_{(p)}$-algebra $R$ with maximal ideal $\mx$ is \defnword{quasi-healthy} if it is faithfully flat over $\Int_{(p)}$, and if every abelian scheme over $\Spec~R\backslash\{\mx\}$ extends uniquely to an abelian scheme over $\Spec~R$.

A regular $\Int_{(p)}$-scheme $X$ is \defnword{healthy} if it is faithfully flat over $\Int_{(p)}$, and if, for every open sub-scheme $U\subset X$ containing $X_{\Rat}$ and all generic points of $X_{\Field_p}$, every abelian scheme over $U$ extends uniquely to an abelian scheme over $X$. It is \defnword{locally healthy} if, for every point $x\in X_{\Field_p}$ of co-dimension at least $2$, the complete local ring $\widehat{\Rg}_{X,x}$ is quasi-healthy.
\end{defn}

\begin{rem}\label{lattice:rem:healthy}
\begin{itemize}
\item Any regular, flat $\Int_{(p)}$-scheme of dimension at most $1$ is trivially healthy.

\item By faithfully flat descent, a regular local ring $R$ is quasi-healthy whenever its completion $\widehat{R}$ is quasi-healthy.

\item If $X$ is locally healthy, then it is healthy. Indeed, suppose that $U\subset X$ is as in the definition of `healthy' above; the complement $X\backslash U$ lies entirely in the special fiber and has co-dimension at least $2$ in $X$. The claim follows by using ascending Noetherian induction on the co-dimension of $X\backslash U$, and repeatedly using quasi-healthiness of the local rings of $X$.

    We do not know if the converse holds.
\end{itemize}
\end{rem}

\begin{thm}[(Vasiu-Zink)]\label{lattice:thm:vasiuzink}
Let $R$ be a regular local, faithfully flat $\Int_{(p)}$-algebra of dimension at least $2$.
\begin{enumerate}
\item\label{vz:main}Suppose that there exists a faithfully flat complete local $R$-algebra $\hat{R}$ that admits a surjection $\hat{R}\twoheadrightarrow W\pow{T_1,T_2}/(p-h)$, where $h\in (T_1,T_2)W\pow{T_1,T_2}$ is a power series that does not belong to the ideal $(p,T_1^p,T_2^p,T_1^{p-1}T_2^{p-1})$. Then $R$ is quasi-healthy.
\item\label{vz:easy}Let $\mx_R\subset R$ be the maximal ideal and suppose that $p\notin\mx_R^p$. Then $R$ is quasi-healthy.
\item\label{vz:smooth}If $R$ is a formally smooth complete local $\Int_{(p)}$-algebra, then $R$ is quasi-healthy.
\end{enumerate}
\end{thm}
\begin{proof}
  Cf.~Theorem 3 and Corollary 4 of~\cite{vasiu:zink}.
\end{proof}

\subsection{}\label{lattice:subsec:irregular}
Suppose that $(L,Q)$ is maximal. The isomorphism $p\dual{L}/pL\xrightarrow{\simeq}\disc(L)$ allows us to identify $\disc(L)$ with the radical $N\subset L_{\Field_p}$. By (\ref{lattice:lem:maximal}), $\disc(L)$ with its induced form $\overline{Q}$ has to be an anisotropic quadratic space over $\Field_p$. This implies that $t=\dim N\leq 2$. If $t=2$, $\disc(L)_{\Field_{p^2}}=N_{\Field_{p^2}}$ admits two lines that are $\overline{Q}$-isotropic. These lines, since every sub-space of $N_{\Field_{p^2}}$ is $Q$-isotropic, can be viewed as points in $\on{M}^{\loc}_G(\Field_{p^2})$. We will call these points \defnword{irregular}.

\begin{prp}\label{lattice:prp:mlochealthy}
If $t\leq 1$, then $\on{M}^{\loc}_G$ is regular and locally healthy. If $t=2$, then the same assertion holds for the complement of the irregular points defined above.
\end{prp}
\begin{proof}
  If $t=0$, then $\on{M}^{\loc}_G$ is smooth over $\Int_{(p)}$ and the result is immediate from (\ref{lattice:thm:vasiuzink}).

  So we assume that $t\geq 1$. The assertions in the proposition can be checked at the complete local rings of $\on{M}^{\loc}_G$ at points valued in algebraically closed fields of characteristic $p$. As such, it suffices to prove them after base-change to $\Int_{p^2}=W\bigl(\Field_{p^2}\bigr)$. Now, we can find a basis for $\Int_{p^2}\otimes L$ so that the quadratic form has the shape:
  \begin{align*}
    \sum_{i=1}^{r-1}X_i^2+pY^2,\;\text{if $t=1$}\qquad;\qquad  \sum_{i=1}^{r-2}X_i^2+pYZ,\;\text{if $t=2$}.
  \end{align*}

  The singular locus of $\Int_{p^2}\otimes\on{M}^{\loc}_G$ is precisely where $p$ and all the co-ordinates $X_i$, $1\leq i\leq r-t$, vanish. When $t=2$, the irregular points are precisely those points in the singular locus where one of $Y$ or $Z$ also vanishes.

  Set:
  \[
   U=\begin{cases}
     \Spec\frac{\Int_{p^2}[u_1,\ldots,u_{r-1}]}{(\sum_iu_i^2+p)},\qquad\qquad\;\;\text{if $t=1$};\\
     \Spec\frac{\Int_{p^2}[u_1,\ldots,u_{r-2},v]}{(\sum_iu_i^2+pv)},\qquad\qquad\text{if $t=2$}.
   \end{cases}
  \]
  Every singular point of $\Int_{p^2}\otimes\on{M}^{\loc}_G$ has a Zariski open neighborhood isomorphic to $U$.

  It is now an easy observation that $U$ is regular everywhere when $t=1$. When $t=2$, it is regular outside of points where all the co-ordinates $u_1,\ldots,u_{r-2},v$ vanish: these are precisely the irregular points defined above.

  When $t=1$, the completion of $U$ at its singular point is
  $\Int_{p^2}\pow{u_1,\ldots,u_{r-2}}/(\sum_iu_i^2+p)$. Here, we assume that $r\geq 3$, since otherwise $\dim M^{\loc}_G\leq 1$, and being healthy is a vacuous condition.

  Similarly, when $t=2$, the completion of $U$ at any regular $\overline{\Field}_p$-valued point in the singular locus of the special fiber is isomorphic to $W(\overline{\Field}_p)\pow{u_1,\ldots,u_{r-2},w}/(\sum_iu_i^2+pw+p)$. Here again we assume that $r\geq 3$.

  We also need to consider the completion of $U$ at the generic point of the singular locus. If we complete instead at an algebraically closed point over this generic point, a quick computation shows that we obtain a ring isomorphic to $W(k)\pow{u_1,\ldots,u_{r-2}}/(\sum_iu_i^2+pv)$. Here, $k$ is an algebraically closed field containing $\Field_{p^2}((v))$, and we view $v$ as an element of $W(k)$ via the Teichm\"uller lift.

  In all three cases, if $\mx$ is the maximal ideal of the complete local ring, we see that $p\in\mx^2\backslash\mx^3$. So we can conclude using (\ref{lattice:thm:vasiuzink})(\ref{vz:easy}) that the complete local rings of $U$ at any regular, singular point are quasi-healthy. Since the complete local rings at the non-singular points are also quasi-healthy by \emph{loc. cit.}, we see that the regular locus of $U$ is locally healthy. This proves the proposition.
\end{proof}

\subsection{}\label{lattice:subsec:mref}
Assume that $t=2$. We will now construct a regular, locally healthy resolution $\on{M}^{\on{ref}}_G$ of $\on{M}^{\on{loc}}_G$. Fix any quadratic extension $F/\Rat$ in which $p$ is inert. The two isotropic lines in $\disc(L)_{\Field_{p^2}}$, via the correspondence noted in the proof of (\ref{lattice:lem:maximal}), determine two self-dual lattices in $L_{F}$ containing $L_{\Reg{F,(p)}}$. Fix one of them and denote it by $L^{\diamond}$.

Given a $\Int_{(p)}$-algebra $R$, we take $\on{M}^{\on{ref}}_G(R)$ to be the set of pairs $(F^1L_R,F^1L^{\diamond}_R)$, where:
\begin{itemize}
  \item $F^1L_R\subset L_R$ is an isotropic line.
  \item $F^1L^{\diamond}_R\subset L^{\diamond}_R=L^{\diamond}\otimes_{\Int_{(p)}}R$ is an isotropic $\Reg{F,(p)}\otimes_{\Int_{(p)}}R$-sub-module that is locally a direct summand of rank $1$.
  \item Under the natural map $\Reg{F,(p)}\otimes_{\Int_{(p)}}L_R\to L^{\diamond}_R$, $\Reg{F,(p)}\otimes F^1L_R$ maps into $F^1L^{\diamond}_R$.
\end{itemize}
\begin{prp}\label{lattice:prp:mrefhealthy}
\mbox{}
\begin{enumerate}
\item\label{mref:healthy}$\on{M}^{\on{ref}}_G$ is represented by a regular, locally healthy projective $\Int_{(p)}$-scheme of relative dimension $r-2$.
\item\label{mref:birat}The natural map $q:\on{M}^{\on{ref}}_G\to\on{M}^{\on{loc}}_G$ is $G$-equivariant and an isomorphism over the regular locus of the target.
\end{enumerate}
\end{prp}
\begin{proof}
  Let $Q(L^{\diamond})$ be the smooth quadric over $\Reg{F,(p)}$ attached to $L^{\diamond}$, and let $\on{M}^{\loc}_{G^{\diamond}}$ be the Weil restriction of $Q(L^{\diamond})$ from $\Reg{F,(p)}$ to $\Int_{(p)}$. By definition, we have an inclusion of functors
  \[
   \on{M}^{\on{ref}}_G\into\on{M}^{\loc}_G\times_{\Int_{(p)}}\on{M}^{\loc}_{G^{\diamond}}.
  \]
  This is easily seen to be a closed immersion, so that $\on{M}^{\on{ref}}_G$ is represented by a projective $\Int_{(p)}$-scheme. Since $G$ acts trivially on $\disc(L)$ its action on $L_{F}$ preserves $L^{\diamond}$. This endows $\on{M}^{\on{ref}}_G$ with a natural $G$-action compatible with its projection onto $\on{M}^{\loc}_G$.

  Fix an identification of the completion of $\Reg{F,(p)}$ with $\Int_{p^2}$. Using the isomorphism of $\Int_{p^2}$-algebras $\Int_{p^2}\otimes_{\Int_{(p)}}\Reg{F,(p)}\xrightarrow{\simeq}\Int_{p^2}\times\Int_{p^2}$, we obtain maps
  \begin{align}\label{healthyref:eqn1}
   \Int_{p^2}\otimes_{\Int_{(p)}}L\to\Int_{p^2}\otimes_{\Int_{(p)}}L^{\diamond}=(\Int_{p^2}\otimes_{\Reg{F,(p)}}L^{\diamond})\oplus(\Int_{p^2}\otimes_{\Reg{F,(p)}}L^{\diamond}).
  \end{align}

  As in the proof of (\ref{lattice:prp:mlochealthy}), we can find compatible bases for $\Int_{p^2}\otimes_{\Int_{(p)}}L$ and $\Int_{p^2}\otimes_{\Reg{F,(p)}}L^{\diamond}$ such that the map in (\ref{healthyref:eqn1}) has the shape:
  \begin{align*}
    \Int_{p^2}\otimes_{\Int_{(p)}}L&\to (\Int_{p^2}\otimes_{\Reg{F,(p)}}L^{\diamond})\oplus(\Int_{p^2}\otimes_{\Reg{F,(p)}}L^{\diamond})\\
    (X_1,\ldots,X_{r-2},Y,Z)&\mapsto \bigl((X_1,\ldots,X_{r-2},pY,Z),(X_1,\ldots,X_{r-2},Y,pZ)\bigr);
  \end{align*}
  and so that the quadratic forms on $\Int_{p^2}\otimes_{\Int_{(p)}}L$ and $\Int_{p^2}\otimes_{\Reg{F,(p)}}L^{\diamond}$ are given by the formulas $\sum_iX_i^2+pYZ$ and $\sum_iX_i^2+YZ$, respectively.

  Let $C\subset\on{M}^{\loc}_G$ be the irregular locus: then $\Int_{p^2}\otimes C\subset\Int_{p^2}\otimes\on{M}^{\loc}_G$ can be identified with the locus in the special fiber where the co-ordinates $X_1,\ldots,X_{r-2}$, as well as (exactly) one of $Y$ or $Z$, vanish. Set $V_1=\on{M}^{\loc}_G\backslash C$; then we see that, as open sub-schemes of $\Int_{p^2}\otimes\on{M}^{\loc}_G$, we have:
  \[
  \Int_{p^2}\otimes V_1=\bigcup_{i=1}^{r-1}\{X_i\neq 0\}\cup\{YZ\neq 0\}.
  \]
  This immediately shows that, if $R$ is an $\Int_{(p)}$-algebra with $F^1L_R\in V_1(R)$, the image of $\Reg{F,(p)}\otimes_{\Int_{(p)}}F^1L_R$ in $L^{\diamond}_R$ is locally a direct summand as an $\Reg{F,(p)}\otimes_{\Int_(p)}R$-module. In other words, the map $q:q^{-1}(V_1)\to V_1$ is an isomorphism. This shows (\ref{mref:birat}).

  Using (\ref{lattice:prp:mlochealthy}), we see that $q^{-1}(V_1)$ is regular and locally healthy.

  To show that $\on{M}^{\on{ref}}_G$ is regular and locally healthy, we need to show that its complete local rings at points in characteristic $p$ are quasi-healthy regular. For this, we can now work in the neighborhood of a singular point in $\Int_{p^2}\otimes\on{M}^{\on{ref}}_G$ where one of $Y$ or $Z$ vanishes.

  Without loss of generality, we can assume that it is the co-ordinate $Y$ that vanishes, so that $Z\neq 0$. If we set $x_i=X_i/Z$ and $y=Y/Z$, just as in the proof of \emph{loc. cit.} we can work over a Zariski open affine neighborhood $V_2\subset\Int_{p^2}\otimes\on{M}^{\loc}_G$ of the form $\Spec A$, with $A=\frac{\Int_{p^2}[x_1,\ldots,x_{r-2},y]}{(\sum_ix_i^2+py)}$.

  Now, for any $\Int_{p^2}$-algebra $R$, the set $q^{-1}(V_2)(R)$ can be identified with the space of pairs $(F^1L_R,F^1(R\otimes_{\Reg{F,(p)}}L^{\diamond}))$, where:
  \begin{itemize}
    \item $F^1L_R\subset L_R$ is an isotropic line spanned by an element with co-ordinates $(x_1,\ldots,x_{r-2},y,1)$, for $x_1,\ldots,x_{r-2},y\in R$.
    \item $F^1(R\otimes_{\Reg{F,(p)}}L^{\diamond})\subset R\otimes_{\Reg{F,(p)}}L^{\diamond}$ is an isotropic line containing an element with co-ordinates $(x_1,\ldots,x_{r-2},y,p)$.
  \end{itemize}
  Therefore, we find that, as schemes over $V_2$, we have:
  \[
   q^{-1}(V_2)\xrightarrow{\simeq}\Proj A[W_1,\ldots,W_{r-2},U,T]/I,
  \]
  where $I$ is the ideal generated by the elements $\sum_{k=1}^{r-2}W_k^2+UT$, $x_jW_i-x_iW_j$, for $1\leq i,j\leq r-2$, $pW_i-x_iU,yW_i-x_iT$, for $1\leq i\leq r-2$, and $yU-pT$.

  The scheme $q^{-1}(V_2)$ is covered by three kinds of affine open sub-schemes: One where $U\neq 0$, one where $T\neq 0$, and one where $W_i\neq 0$, for some $i$. We will consider each in turn.

  Setting $w_i=W_i/U$ and $t=T/U$, we can easily see that $q^{-1}(V_2)\cap\{U\neq 0\}$ is isomorphic to:
  \[
   \Spec\frac{\Int_{p^2}[w_1,\ldots,w_{r-2},t]}{(\sum_kw_k^2+t)}.
  \]
  This is clearly smooth over $\Int_{p^2}$ and is therefore regular and locally healthy.

  Similarly, setting $w_i=W_i/T$ and $u=U/T$ instead, we find that $q^{-1}(V_2)\cap\{T\neq 0\}$ is isomorphic to:
  \[
   \Spec\frac{\Int_{p^2}[w_1,\ldots,w_{r-2},u,y]}{(\sum_kw_k^2+u,uy-p)}=\Spec\frac{\Int_{p^2}[w_1,\ldots,w_{r-2},y]}{(p+y(\sum_kw_k^2))}.
  \]
  This is a regular scheme over $\Int_{p^2}$.

  Consider the complete local ring of $q^{-1}(V_2)\cap\{T\neq 0\}$ at the point where $p$ and all the co-ordinates $y,w_1,\ldots,w_{r-2}$ vanish: It is
  \[
   R=\frac{\Int_{p^2}\pow{w_1,\ldots,w_{r-2},y}}{(p+y(\sum_kw_k^2))}.
  \]
  It admits a surjection to $\frac{\Int_{p^2}\pow{w_1,y}}{(p+yw_1^2)}$, and so we can use (\ref{lattice:thm:vasiuzink})(\ref{vz:main}) to conclude that $R$ is quasi-healthy. One can check that the order of vanishing of $p$ at all other points of $q^{-1}(V_2)\cap\{T\neq 0\}$ is at most $2$, and so we can use assertion (\ref{vz:easy}) of \emph{loc. cit.} to conclude that all complete local rings of $q^{-1}(V_2)\cap\{T\neq 0\}$ are quasi-healthy. Hence this affine open is also locally healthy.

  Finally, setting $w_j=W_j/W_1$ for $j\neq 1$, $u=U/W_1$ and $t=T/W_1$, we find that $q^{-1}(V_2)\cap\{W_1\neq 0\}$ is isomorphic to:
  \[
   \Spec\frac{\Int_{p^2}[x_1,w_2,\ldots,w_{r-2},u,t]}{(ux_1-p,1+\sum_{k=1}^{r-2}w_k^2+ut)}.
  \]
  Set $B'=\frac{\Int_{p^2}[x_1,u]}{(ux_1-p)}$; then $q^{-1}(V_2)\cap\{W_1\neq 0\}$ is smooth over $\Spec B'$, and so it suffices to check the criterion of (\ref{lattice:thm:vasiuzink})(\ref{vz:easy}) for the complete local rings of $B'$, which is easy. This shows that $q^{-1}(V_2)\cap\{W_1\neq 0\}$ is locally healthy, and an identical proof shows that $q^{-1}(V_2)\cap\{W_i\neq 0\}$ is locally healthy for any $i$.
\end{proof}

\begin{rem}\label{lattice:rem:mrefind}
  The map $\on{M}^{\on{ref}}_G\to\on{M}^{\on{loc}}_G$ is, up to unique isomorphism, independent of the choice of both $F$ and $L^{\diamond}$. In fact, it is simply the blow-up of the singular locus of $\on{M}^{\on{loc}}_G$. 
\end{rem}

\section{GSpin Shimura varieties}\label{sec:spin}

Let $(L,Q)$ be a quadratic space over $\Int_{(p)}$ of signature $(n,2)$ with $n\geq 1$. By this, we mean that the largest positive definite sub-space of $L_{\Real}$ has dimension $n$. Let $G$ (resp. $G_0$) be the smooth $\Int_{(p)}$-group scheme attached to $L$ in (\ref{lattice:subsec:parahoric}), so that $G_{\Rat}=\GSpin(L_{\Rat},Q)$ (resp. $G_{0,\Rat}=\SO(L_{\Rat},Q)$).

\subsection{}\label{spin:subsec:complex}
Let $X$ be the space of oriented negative definite $2$-planes in $L_{\Real}$. The points of $X$ correspond to certain Hodge structures of weight $0$ on the vector space $L_{\Rat}$, polarized by $Q$: Fix $\mb{h}\in X$, and suppose that $(e_{\mb{h}},f_{\mb{h}})$ is an oriented, orthogonal basis for the oriented negative definite $2$-plane attached to $\mb{h}$, with $Q(e_{\mb{h}})=Q(f_{\mb{h}})=-1$. Also fix a square root of $-1$, $\sqrt{-1}\in\Comp$. Set
\[
  L_{\mb{h}}^{p,q}=\begin{cases}
    \langle e_{\mb{h}}+\sqrt{-1}f_{\mb{h}}\rangle\subset L_{\Comp}&\text{if $(p,q)=(-1,1)$};\\
    \langle e_{\mb{h}},f_{\mb{h}}\rangle^{\perp}\subset L_{\Comp}&\text{if $(p,q)=(0,0)$};\\
    \langle e_{\mb{h}}-\sqrt{-1}f_{\mb{h}}\rangle\subset L_{\Comp}&\text{if $(p,q)=(1,-1)$};\\
    0&\text{otherwise}.
  \end{cases}
\]
Then $L_{\mb{h}}$ is a $\Int_{(p)}$-Hodge structure of weight $0$ with underlying $\Int_{(p)}$-module $L$; the associated $\Rat$-Hodge structure is polarized by $Q$. In fact, each $\mb{h}\in X$ gives rise to a unique homomorphism $\bb{S}\to G_{\Real}$ which induces the Hodge structure $L_{\mb{h}}$ on $L$ and whose restriction to the diagonal sub-group $\Gmh{\Real}\subset\bb{S}$ is the canonical central embedding $\Gmh{\Real}\into G_{\Real}$. Here, of course, $\bb{S}=\Res_{\Comp/\Real}\Gmh{\Real}$ is the Deligne torus.

The map carrying $\mb{h}$ to the line $L_{\mb{h}}^{1,-1}$ embeds $X$ as an open sub-space of the quadric $\check{X}\subset\bb{P}(L_{\Comp})$ determined by $Q$. The two connected components of $X$ are switched by complex conjugation on $\check{X}$. The pairs $(G_{\Rat},X)$ and $(G_{0,\Rat},X)$ are Shimura data, which, since $n\geq 1$, have reflex field $\Rat$.

\subsection{}\label{spin:subsec:shimura}
Fix a compact open sub-group $K\subset G(\Adele_f)$ with image $K_0\subset G_0(\Adele_f)$, and let $\Sh_K\coloneqq\Sh_K(G_{\Rat},X)$ and $\Sh_{K_0}\coloneqq\Sh_{K_0}(G_{0,\Rat},X)$ be the associated Shimura varieties over $\Rat$. We will assume that $K$ is of the form $K_pK^p$, where $K_p=G(\Int_p)\subset G(\Rat_p)$ and $K^p\subset G(\Adele_f^p)$. We will also assume that $K^p$ is chosen to be small enough, so that $\Sh_K$ is a smooth variety and not just an algebraic space. By weak approximation for $G_{\Rat}$ (which can be deduced from weak approximation for its derived group, which is simply connected~\cite{platrap}*{Theorem 7.8}), we have identifications of complex analytic varieties:
\begin{align}\label{spin:eqn:uniformization}
 \Sh_{K,\Comp}^{\an}&=G(\Int_{(p)})\backslash\bigl(X\times G(\Adele^p_f)/K^p\bigr);\\
 \Sh_{K_0,\Comp}^{\an}&=G_0(\Int_{(p)})\backslash\bigl(X\times G_0(\Adele^p_f)/K_0^p\bigr).
\end{align}
From this description, we find that the map $\Sh_K\to\Sh_{K_0}$ is a finite (\'etale) Galois cover with Galois group
\[
\Delta(K)\coloneqq\Adele_f^\times/\Rat^{>0}(K\cap\Adele_f^\times)=\Adele_f^{p,\times}/\Int_{(p)}^{>0}\bigl(K^p\cap\Adele_f^{p,\times}\bigr).
\]
Here, we are viewing $\Adele_f^\times$ as a central sub-group of $G(\Adele_f^\times)$.

\subsection{}\label{spin:subsec:variations}
Let $R\subset\Real$ be a $\Int_{(p)}$-algebra. Recall that a variation of (pure) $R$-Hodge structures over a smooth complex algebraic variety $S$ is a pair $(\bm{U}_B,F^\bullet(\bm{U}_B\otimes\Reg{S^{\an}}))$, where $\bm{U}_B$ is a local system of finite free $R$-modules over $S^{\an}$; and $F^\bullet(\bm{U}_B\otimes\Reg{S^{\an}})$ is a descending filtration by sub-vector bundles over $S^{\an}$ such that, for every point $s\in S(\Comp)$, the induced pair $(\bm{U}_{B,s},F^\bullet(\bm{U}_{B,s}\otimes_R\Comp))$ is a pure $R$-Hodge structure. An \defnword{algebraic} variation of Hodge structures over $S$ is a tuple $(\bm{U}_B,\bm{U}_{\dR},F^\bullet\bm{U}_{\dR},\iota)$, where $\bm{U}_{\dR}$ is a vector bundle over $S$ equipped with an integrable connection and a filtration $F^\bullet\bm{U}_{\dR}$; and $\iota:\bm{U}^{\an}_{\dR}\xrightarrow{\simeq}\bm{U}_B\otimes\Reg{S^{\an}}$ is a parallel isomorphism (the right hand side being endowed with the trivial integrable connection) of vector bundles over $S^{\an}$ such that the pair $(\bm{U}_B,\iota(F^\bullet\bm{U}^{\an,\dR}))$ is a variation of $R$-Hodge structures.

There is a natural exact tensor functor from the category of algebraic $R$-representations of $G$ (resp. $G_0$) to the category of algebraic variations of $R$-Hodge structures on $\Sh_{K,\Comp}$ (resp. $\Sh_{K_0,\Comp}$).\footnote{Such a result is true for any Shimura variety.} This is reasonably well-known, but, so as to fix notation, we will now briefly describe it for $G$ and $\Sh_K$; the situation for $G_0$ and $\Sh_{K_0}$ is completely analogous.

Suppose that we are given an algebraic $R$-representation $U$ of $G$. We can view the $R$-module $U$ as a representation of the discrete group $G(\Int_{(p)})$. Using the uniformization in (\ref{spin:eqn:uniformization}), we find that the constant local system $U\times X\times G(\Adele_f^p)/K^p$ (where we equip $U$ with the discrete topology) over $X\times G(\Adele_f^p)/K^p$ descends to an $R$-local system $\bm{U}_B$ over $\Sh_{K,\Comp}^{\an}$.

We also have the trivial vector bundle $(U\otimes_R\Comp)\times X$ over $X$, where we equip $U\otimes_{R}\Comp$ with the natural complex topology. This has a natural, $G(\Real)$-equivariant descending filtration by vector sub-bundles $F^\bullet((U\otimes_R\Comp)\times X)$ such that, at any point $\mb{h}\in X$, the induced filtration $F^{\bullet}(U_{\mb{h}}\otimes_R\Comp)$ is simply the Hodge filtration induced by the homomorphism $\bb{S}\to G_{\Real}$ attached to $\mb{h}$. Being $G(\Real)$-equivariant, and in particular $G(\Int_{(p)})$-equivariant, the pair $\bigl((U\otimes_R\Comp)\times X,F^\bullet((U\otimes_R\Comp)\times X)\bigr)$ descends to a filtered vector bundle over $\Sh_{K,\Comp}^{\an}$; we will denote this descent by $(\bm{U}_{\dR,\Comp}^{\an},F^\bullet\bm{U}^{\an}_{\dR,\Comp})$. By construction, $\bm{U}_{\dR,\Comp}^{\an}$ is canonically isomorphic to $\bm{U}_B\otimes_R\Reg{\Sh_{K,\Comp}^{\an}}$, and we find that the pair $(\bm{U}_B,F^\bullet\bm{U}^{\an}_{\dR,\Comp})$ is a variation of $R$-Hodge structures over $\Sh_{K,\Comp}^{\an}$.

Now, $\bm{U}_{\dR,\Comp}^{\an}$ algebraizes to an algebraic vector bundle with integrable connection $\bm{U}_{\dR,\Comp}$. This is essentially due to Baily-Borel~\cite{baily_borel}*{Theorem 10.14}; cf. also~\cite{harris:arithmetic_i}*{(3.1)}). Moreover, by the projectivity of Grassmannians, the filtration $F^\bullet\bm{U}^{\an}_{\dR,\Comp}$ algebraizes to a filtration $F^\bullet\bm{U}_{\dR,\Comp}$. This finishes our construction of the algebraic variation of $R$-Hodge structures attached to the representation $U$: we will denote it by $\mathbb{V}_{\Comp}(U)$. One can check that $U\mapsto\mathbb{V}_{\Comp}(U)$ is functorial, exact and respects tensor operations.

As shown in \cite{deligneshimura}*{\S~1.1}, the variations of Hodge structures obtained in this fashion satisfy Griffiths's transversality: The connection on $\bm{U}_{\dR,\Comp}$ carries $F^i\bm{U}_{\dR,\Comp}$ to $F^{i-1}\bm{U}_{\dR,\Comp}\otimes\Omega^1_{\Sh_{K,\Comp}/\Comp}$.

\subsection{}\label{spin:subsec:hodgetorsor}
Let $C=C(L,Q)$ be the Clifford algebra for $(L,Q)$. The above construction applied to the representation $H=C$, on which $G$ acts via left multiplication, produces an algebraic variation of $\Int_{(p)}$-Hodge structures $\mathbb{V}_{\Comp}(H)=(\bm{H}_B,\bm{H}_{\dR,\Comp},F^\bullet\bm{H}_{\dR,\Comp})$ over $\Sh_{K,\Comp}^{\an}$. Since the right $C$-action and grading on $H$ are $G$-equivariant, they are both naturally inherited by $\mathbb{V}_{\Comp}(H)$.

The tensor $\pr\in H_{\Rat}^{\otimes(2,2)}$ is $G$-invariant, and we can view it as a map of $G$-representations $\pr:\Rat\to H_{\Rat}^{\otimes(2,2)}$, with $G$ acting trivially on $\Rat$. The functoriality of our construction now shows that $\pr$ induces a map of algebraic variations of $\Rat$-Hodge structures $\mathbb{V}_{\Comp}(\Rat)\to\mathbb{V}_{\Comp}(H^{\otimes(2,2)}_{\Rat})=\mathbb{V}_{\Comp}(H_{\Rat})^{\otimes(2,2)}$.

Explicitly, this means that we have a global section $\pr_B\in H^0\bigl(\Sh_{K,\Comp}^{\an},\bm{H}_{B}^{\otimes(2,2)}\otimes\Rat\bigr)$ and a section $\pr_{\dR,\Comp}\in H^0\bigl(\Sh_{K,\Comp},F^0\bm{H}_{\dR,\Comp}^{\otimes(2,2)})$ that is parallel for the connection on $\bm{H}_{\dR,\Comp}^{\otimes(2,2)}$. Moreover, $\pr_B$ is carried to $\pr_{\dR,\Comp}$ under the comparison isomorphism
\[
 \bm{H}_{B}^{\otimes(2,2)}\otimes\Reg{\Sh_{K,\Comp}^{\an}}\xrightarrow{\simeq}\bm{H}^{\otimes(2,2),\an}_{\dR,\Comp}.
\]
In particular, for every point $s\in\Sh_K(\Comp)$, the fiber $\pr_{B,s}\in\bm{H}_{B,s}^{\otimes(2,2)}\otimes\Rat$ is a Hodge tensor.

We can view $\pr$ as an idempotent endomorphism of $\mathbb{V}_{\Comp}\bigl(H^{\otimes(1,1)}_{\Rat}\bigr)$, and by construction its image is precisely $\mathbb{V}_{\Comp}(L_{\Rat})=(\bm{L}_{B,\Rat},\bm{L}_{\dR,\Comp},F^\bullet\bm{L}_{\dR,\Comp})$, the algebraic variation of $\Rat$-Hodge structures attached to the representation $L_{\Rat}$.

In fact, it follows from (\ref{cliff:rem:nondeg}) that the image under $\pr$ of $\mathbb{V}_{\Comp}(H^{\otimes(1,1)})\subset\mathbb{V}_{\Comp}(H^{\otimes(1,1)}_{\Rat})$ is exactly the variation of $\Int_{(p)}$-Hodge structures $\mathbb{V}_{\Comp}(\dual{L})$.

Moreover, the $G$-equivariant map $L_{\Rat}\to\dual{L}_{\Rat}$ induced by the pairing gives rise to an isomorphism $\mathbb{V}_{\Comp}(L_{\Rat})\xrightarrow{\simeq}\mathbb{V}_{\Comp}(\dual{L}_{\Rat})$. This corresponds to a pairing $\mathbb{V}_{\Comp}(L_{\Rat})\otimes\mathbb{V}_{\Comp}(L_{\Rat})\to\mathbb{V}_{\Comp}(\Rat)$, which is in fact a polarization of variations of Hodge structures. $\mathbb{V}_{\Comp}(L)$ is precisely the pre-image of $\mathbb{V}_{\Comp}(\dual{L})$ under this isomorphism.

In sum, we find that we can recover the tuple $\mathbb{V}_{\Comp}(L)=(\bm{L}_B,\bm{L}_{\dR,\Comp},F^\bullet\bm{L}_{\dR,\Comp})$ from the data of $\mathbb{V}_{\Comp}(H)$ and the idempotent operator $\pr\in\mathbb{V}_{\Comp}(H_{\Rat}^{\otimes(2,2)})$.

\subsection{}\label{spin:subsec:hodgeemb}
For $\delta\in C\cap C_{\Rat}^\times$ satisfying $\delta^*=-\delta$, set $\mathcal{G}_{\delta,\Rat}=\GSp(H_{\Rat},\psi_{\delta})$, so that we have an embedding $G_{\Rat}\into\mathcal{G}_{\delta,\Rat}$. Let $\mathcal{X}$ be the space of Lagrangian sub-spaces $W\subset H_{\Comp}$ (with respect to the form $\psi_{\delta}$) such that the Hermitian form $\sqrt{-1}\psi_{\delta}(w_1,\bar{w}_2)$ restricts to a (positive or negative) definite form on $W$: this is simply the union of the Siegel half-spaces attached to $(H,\psi_{\delta})$.
\begin{lem}\label{spin:lem:ksemb}
One can choose $\delta$ so that the embedding $G_{\Rat}\into\mathcal{G}_{\delta,\Rat}$ induces an embedding of Shimura data $(G_{\Rat},X)\into (\mathcal{G}_{\delta,\Rat},\mathcal{X})$.
\end{lem}
\qed

Let $\mathcal{K}_p\subset\mathcal{G}_{\delta,\Rat}(\Rat_p)$ be the stabilizer of $H_{\Int_p}$. Then we have $K_p\subset\mathcal{K}_p\cap G(\Rat_p)$. Let $K=K_pK^p\subset G(\Adele_f)$ be as above; then, for any compact open $\mathcal{K}^p\subset\mathcal{G}_{\delta}(\Adele_f^p)$ containing $K_p$, we have a finite, unramified\footnote{Note that, over $\Comp$ and locally in the complex analytic topology, this map is isomorphic to the closed immersion $X\into\mathcal{X}$.} map of canonical models of Shimura varieties over $\Rat$:
\[
  \Sh_K\to\Sh_{\mathcal{K}}(\mathcal{G}_{\delta,\Rat},\mathcal{X}).
\]
Here, $\mathcal{K}=\mathcal{K}_p\mathcal{K}^p$. We will call this a \defnword{Kuga-Satake map}.

\subsection{}\label{spin:subsec:primetop}
$\Sh_{\mathcal{K}}\coloneqq\Sh_{\mathcal{K}}(\mathcal{G}_{\delta,\Rat},\mathcal{X})$ has a natural moduli description. To describe this, we will work with abelian schemes up to prime-to-$p$ isogeny. More precisely, given a scheme $T$, the category $AV_{(p)}(T)$ of abelian schemes up to prime-to-$p$ isogeny has for its objects abelian schemes $A$ over $T$, where, for two abelian schemes $A$ and $B$ over $T$, the space of morphisms from $A$ to $B$ is the $\Int_{(p)}$-module
\[
  \Hom(A,B)_{(p)}=\Hom(A,B)\otimes\Int_{(p)}.
\]
Given an abelian scheme $A$ over $T$, a \defnword{quasi-polarization} (or simply \defnword{polarization}) of $A$ in $AV_{(p)}(T)$ will be an element $\lambda\in\Hom(A,\dual{A})_{(p)}$ that is a positive multiple of a polarization $\lambda':A\to\dual{A}$.

Given any scheme $T$, a prime $\ell\neq p$ invertible in $T$, and an abelian scheme $f:A\to T$, we can consider the associated relative first $\ell$-adic cohomology sheaf $R^1f_*\underline{\Rat}_{\ell}$. If $p$ is invertible in $T$, we can also consider the the $p$-adic sheaf $R^1f_*\underline{\Int}_p$. Both these constructions are invariants of the prime-to-$p$ isogeny class of $A$. Let $\underline{\Rat}_{\ell}(-1)$ be the Tate twist: it is the relative first $\ell$-adic cohomology of $\Gmh{T}$ over $T$.

Given an abelian scheme $f:A\to T$ over a $\Int_{(p)}$-scheme $T$, the following constructions are invariants of its prime-to-$p$ isogeny class: the $p$-divisible group $A[p^{\infty}]$; the $p$-adic \'etale cohomology $R^1f_*\underline{\Int}_p$; for $\ell\neq p$, the rational $\ell$-adic cohomology $R^1f_*\underline{\Rat}_{\ell}$; and the $\Adele_f^p$-adic cohomology $R^1f_*\underline{\Adele}_f^p$. If $T$ is an $\Field_p$-scheme, we also have the degree $1$ crystalline cohomology of $A$ over $T$: For instance, this can be viewed as the Dieudonn\'e crystal associated with the $p$-divisible group $A[p^{\infty}]$.

\subsection{}\label{spin:subsec:primetopsheaves}
For some purposes, it is useful to embed $AV_{(p)}(T)$ in the category of group schemes over $T$. We will follow \cite{harris:oscillator}*{\S~1} for this. Given an abelian scheme $A\to T$, let $\mathcal{T}^p(A)$ be the inverse system consisting of prime-to-$p$ isogenies\footnote{These are finite, flat homomorphisms, whose kernel is an \'etale group scheme of order not divisible by $p$.} $B\to A$. Then $\mathcal{T}^p(A)$ has a co-final system consisting of the finite \'etale multiplication-by-$n$ endomorphisms $[n]:A\to A$ for $n\in\Int$ with $p\nmid n$. Therefore, the inverse limit
\[
 \hat{A}^{(p)}=\varprojlim_{(B\to A)\in\mathcal{T}^p(A)}B
\]
exists in the category of group schemes over $T$. One sees that $A\mapsto\hat{A}^{(p)}$ is a fully faithful functor from $AV^{(p)}(T)$ to the category of $T$-group schemes.

The main purpose of this construction is the following abuse of terminology: Given $A\in AV^{(p)}(T)$, and a $T$-scheme $T'$, we will write $H^0(T',A)$ for the $\Int_{(p)}$-module $H^0(T',\hat{A}^{(p)})$. This allows us to speak of `sections of $A$'.

From now on we will suppress the qualifying phrase `up to prime-to-$p$ isogeny': \emph{all abelian schemes will only be considered in the prime-to-$p$ isogeny category.} The cohomological constructions we are concerned with will all be invariants of the prime-to-$p$ isogeny class.

\subsection{}
Suppose that $T$ is a $\Int_{(p)}$-scheme and $f:A\to T$ is an abelian scheme. Given a polarization $\lambda:A\to\dual{A}$, we get an induced non-degenerate Poincar\'e pairing of $\Adele_f^p$-sheaves
\[
  \psi_{\lambda}:R^1f_*\underline{\Aff}^p_{f}\otimes R^1f_*\underline{\Aff}^p_{f}\rightarrow\underline{\Aff}_{f}^p(-1).
\]

Suppose that we are also given an isomorphism
\[
 \eta:H\otimes\underline{\Adele}^p_f\xrightarrow{\simeq}R^1f_*\underline{\Aff}^p_f
\]
of $\Adele^p_f$-sheaves over $T$. Then we obtain two different non-degenerate pairings on $H\otimes\underline{\Adele}^p_f$: The first is the constant pairing into $\underline{\Adele}_f^p$ arising from $\psi_{\delta}$, which we will again call $\psi_{\delta}$; and the second is the pairing $\eta^*\psi_{\lambda}$ into $\underline{\Adele}_f^p(-1)$ obtained by pulling back $\psi_{\lambda}$ along $\eta$. In particular, the existence of $\eta$ implies that $\underline{\Adele}_f^p(-1)$ is trivializable over $T$. We say that $\eta$ \defnword{preserves polarizations} if, for some choice of isomorphism $\underline{\Adele}_f^p(-1)\xrightarrow{\simeq}\underline{\Adele}_f^p$, the pairings $\eta^*\psi_{\lambda}$ and $\psi_{\delta}$ agree.

In the above situation, we will consider the \'etale sheaf $I^p(A,\lambda)$ over $T$, whose sections are polarization preserving isomorphisms of $\Adele_f^p$-sheaves $\eta:H\otimes\underline{\Adele}^p_f\xrightarrow{\simeq}R^1f_*\underline{\Aff}^p_f$. Note that $I^p(A,\lambda)$ is a pseudo-torsor under $\mathcal{G}_{\delta}(\Adele^p_f)$ via its action through pre-composition.

For any $\Int_{(p)}$-scheme $T$, let $\mathcal{S}_{\mathcal{K}}(T)$ be the set of isomorphism classes of tuples $(A,\lambda,[\eta])$, where:
\begin{itemize}
  \item $(A,\lambda)$ is a polarized abelian scheme over $T$.
  \item $[\eta]$ is a \defnword{$\mathcal{K}^p$-level structure}: it is a section of the quotient sheaf $I^p(A,\lambda)/\mathcal{K}^p$.
\end{itemize}
For $\mathcal{K}$ sufficiently small, the functor $\mathcal{S}_{\mathcal{K}}$ is (represented by) a quasi-projective scheme over $\Int_{(p)}$, whose generic fiber is canonically identified with $\Sh_{\mathcal{K}}\coloneqq\Sh_{\mathcal{K}}(\mathcal{G}_{\delta,\Rat},\mathcal{X})$.

\subsection{}
Over $\mathcal{S}_{\mathcal{K}}$, we have the tautological tuple $(A,\lambda,[\eta])$. Let $(A^{\KS}_{\Sh_{K}},\lambda^{\KS}_{\Sh_K},[\eta^{\KS}])$ be the induced tuple over $\Sh_K$. We will refer to $A^{\KS}_{\Sh_K}$ as the \defnword{Kuga-Satake abelian scheme} over $\Sh_K$.

The identification of $\Sh_{\mathcal{K}}$ with the generic fiber of the moduli scheme $\mathcal{S}_{\mathcal{K}}$ has the following property: Over $\Sh^{\an}_{K,\Comp}$, the algebraic variation of $\Int_{(p)}$-Hodge structures obtained from the degree $1$ cohomology of $A^{\KS}_{\Sh_{K,\Comp}^{\an}}$ is canonically identified with $\mathbb{V}_{\Comp}(H)$.

Since $\mathbb{V}_{\Comp}(H)$ carries a $\Int/2\Int$-grading and a right $C$-action, we conclude that $A^{\KS}_{\Sh_{K,\Comp}^{\an}}$, and hence $A^{\KS}_{\Sh_{K,\Comp}}$, admits a canonical $\Int/2\Int$-grading and a (left) $C$-action.\footnote{The right action is converted to a left, because cohomology is a contravariant functor.}Here, we are using the (anti-)equivalence of categories between abelian schemes over $\Sh_{K,\Comp}$ and polarizable variations of $\Int_{(p)}$-Hodge structures over $\Sh_{K,\Comp}$ of weight $1$.

The above implies that the degree $1$ relative Betti cohomology of $A^{\KS}_{\Sh_{K,\Comp}^{\an}}$ with coefficients in $\Int_{(p)}$ can be identified with $\bm{H}_B$ as a $\Int_{(p)}$-local system. Similarly, the degree $1$ relative de Rham cohomology of $A^{\KS}_{\Sh_{K,\Comp}}$ can be identified with $(\bm{H}_{\dR,\Comp},F^\bullet\bm{H}_{\dR,\Comp})$ as a filtered vector bundle with integrable connection.

Let $\bm{H}_{\ell}$ (resp. $\bm{H}_p$) to be the relative degree $1$ \'etale cohomology of $A^{\KS}_{\Sh_K}$ over $\Sh_K$ with coefficients in $\Rat_{\ell}$, for $\ell\neq p$ (resp. $\Int_p$). Then, over $\Sh_{K,\Comp}^{\an}$, Artin's comparison theorem gives us canonical isomorphisms of local systems:
\begin{align}\label{spin:eqn:artincomp}
 \alpha_{\ell}:\bm{H}_B\otimes\Rat_{\ell}&\xrightarrow{\simeq}\bm{H}_{\ell}\rvert_{\Sh_{K,\Comp}^{\an}};\;\text{for $\ell\neq p$}\\
 \alpha_p:\bm{H}_B\otimes\Int_p&\xrightarrow{\simeq}\bm{H}_p\rvert_{\Sh_{K,\Comp}^{\an}}.
\end{align}
Let $(\bm{H}_{\dR,\Rat},F^\bullet\bm{H}_{\dR,\Rat})$ be the relative degree $1$ de Rham cohomology of $A^{\KS}_{\Sh_K}$ over $\Sh_K$: this is equipped with the Hodge filtration and the Gauss-Manin connection. Over $\Sh_{K,\Comp}^{\an}$, there now exists a canonical de Rham comparison isomorphism, parallel for the trivial connection on the left hand side:
\begin{align}\label{spin:eqn:derhamcomp}
\alpha_{\dR}:\bm{H}_B\otimes\Reg{\Sh_{K,\Comp}^{\an}}&\xrightarrow{\simeq}\bm{H}_{\dR,\Comp}^{\an}\xrightarrow{\simeq}\bm{H}_{\dR,\Rat}\rvert_{\Sh_{K,\Comp}^{\an}}.
\end{align}
Note that the first isomorphism in this composition carries $\pr_B\otimes 1$ to $\pr_{\dR,\Comp}$.

\begin{prp}\label{spin:prp:prdescent}
\mbox{}
\begin{enumerate}
\item~\label{prdescent:action}The structures of the $\Int/2\Int$-grading and $C$-action on $A^{\KS}_{\Sh_{K,\Comp}}$ descend (necessarily uniquely) to $A^{\KS}_{\Sh_K}$.
\item~\label{prdescent:etale}For any prime $\ell$, the global section $\alpha_{\ell}(\pr_B\otimes 1)$ of $\bm{H}_{\ell}^{\otimes(2,2)}\otimes\Rat$ over $\Sh_{K,\Comp}^{\an}$ arises from a (necessarily unique) section:
    \[
     \pr_{\ell}\in H^0\bigl(\Sh_K,\bm{H}_{\ell}^{\otimes(2,2)}\otimes\Rat\bigr).
    \]
\item~\label{prdescent:derham}The section $\pr_{\dR,\Comp}$ of $\bm{H}_{\dR,\Comp}^{\otimes(2,2)}$ descends (necessarily uniquely) to a parallel section
\[
 \pr_{\dR,\Rat}\in H^0\bigl(\Sh_K,F^0\bm{H}^{\otimes(2,2)}_{\dR,\Rat}\bigr).
\]
\end{enumerate}
\end{prp}
\begin{proof}
  This can be extracted from \cite{kis3}*{\S 2.2}. We sketch the proof:

  There is a canonical pro-finite Galois cover $\Sh_{K^p}\to\Sh_K$ with Galois group $K_p$:
  \[
   \Sh_{K^p}=\varprojlim_{K'_p\subset K_p}\Sh_{K'_pK^p}.
  \]
  Here $K'_p$ runs over the compact open sub-groups of $K_p$.

  If we fix a connected component $S\subset\Sh_K$ and a geometric point $\overline{s}\to S$, pulling the pro-finite cover $\Sh_{K_p}$ back over $S$ gives us a $K_p$-torsor over $S$, which in turn corresponds to a map $\pi_1(S,\overline{s})\to K_p$. The restriction of $\bm{H}_p$ to $S$ is precisely the $p$-adic sheaf attached to the composite representation
  \[
   \pi_1(S,\overline{s})\to K_p=G(\Int_p)\subset\Aut(H\otimes\Int_p).
  \]

  Since the $\Int/2\Int$-grading and $C$-action on $H\otimes\Int_p$ are $K_p$-invariant, and hence $\pi_1(S,\overline{s})$-invariant, we find that the corresponding structures on $\bm{H}_B\otimes\Int_p$ over $S_{\Comp}^{\an}$ must in fact descend to structures on $\bm{H}_p$ over $S$. From this, one can deduce (\ref{prdescent:action}). The essential point is the following simple consequence of Galois descent: Suppose that $A$ and $B$ are abelian varieties over a characteristic $0$ field ${\kappa}$, and $f:A_{\overline{{\kappa}}}\to B_{\overline{{\kappa}}}$ is a map of abelian varieties over an algebraic closure $\overline{{\kappa}}/{\kappa}$ such that the induced map on $p$-adic Tate modules $T_p(f):T_p(A_{\overline{{\kappa}}})\to T_p(B_{\overline{{\kappa}}})$ is equivariant for the action of the absolute Galois group $\Gal(\overline{{\kappa}}/{\kappa})$. Then $f$ is defined over $\kappa$.

  Similarly, since the tensor $\pr\in H^{\otimes(2,2)}\otimes\Rat_p$ is also $K_p$-invariant, $\alpha_p(\pr_B\otimes 1)$ must descend to a section $\pr_p\in H^0\bigl(\Sh_K,\bm{H}_p^{\otimes(2,2)}\otimes\Rat_p\bigr)$. This shows the case $\ell=p$ of (\ref{prdescent:etale}). The de Rham case (\ref{prdescent:derham}) and the $\ell\neq p$ case of (\ref{prdescent:etale}) follow from this and the fact that all Hodge cycles on abelian varieties are absolutely Hodge~\cite{dmos}*{Ch. I}; cf. the proof of~\cite{kis3}*{2.2.2}. There are two main points~\cite{dmos}*{\S I.2}:
  \begin{itemize}
  \item Given a projective variety $X$ over $\kappa$ and an algebraically closed overfield $K\supset\overline{\kappa}$, any absolute Hodge cycle $\beta$ over $X_{K}$ is already defined over $X_{\overline{\kappa}}$; that is, the de Rham realization $\beta_{\dR}$ of $\beta$ lies in $H_{\dR}^\bullet(X_{\overline{\kappa}}/\overline{\kappa})^{\otimes}$.
  \item If one cohomological realization of $\beta$ is fixed by $\Gal(\overline{\kappa}/\kappa)$, then all of its realizations are.
  \end{itemize}
  We apply this to the situation where $\kappa$ is the function field of an irreducible component $S$ as above, $K$ is the function field of an irreducible component of $S_{\Comp}$, $X$ is the fiber of $A^{\KS}_{\Sh_K}$ over $\kappa$, and $\beta$ is the Hodge (hence absolutely Hodge) cycle over $X_K$ obtained from the fiber of the Betti realization $\pr_B$.
\end{proof}

\subsection{}\label{spin:subsec:lrealizations}
The descent of the realizations of $\pr$ proven above now allows us to descend $\mathbb{V}_{\Comp}(L)$ over $\Sh_K$:

For a prime $\ell\neq p$, let $\bm{L}_{\ell}\subset\bm{H}_{\ell}^{\otimes(1,1)}$ be the image of the idempotent operator $\pr_{\ell}$. Then $\bm{L}_{\ell}$ is a $\Rat_{\ell}$-local system over $\Sh_K$ equipped with a non-degenerate pairing $\bm{L}_{\ell}\times\bm{L}_{\ell}\to\underline{\Rat}_{\ell}$. Over $\Sh_{K,\Comp}^{\an}$, we have a canonical comparison isomorphism (respecting pairings):
\[
 \bm{L}_B\otimes\Rat_{\ell}\xrightarrow{\simeq}\bm{L}_{\ell}\rvert_{\Sh_{K,\Comp}^{\an}}.
\]

For $\ell=p$, as in (\ref{spin:subsec:hodgetorsor}), we can show that the image of $\bm{H}_p^{\otimes(1,1)}$ under $\bm{\pr}_p$ is a descent $\dual{\bm{L}}_p$ of the $\Int_p$-local system $\dual{\bm{L}}_B\otimes\Int_p$ over $\Sh_{K,\Comp}^{\an}$. It is equipped with a quadratic form with values in $\underline{\Rat}_p$, and $\bm{L}_p\subset\dual{\bm{L}}_p$ is recovered as the largest sub-local system whose pairing with $\dual{\bm{L}}_p$ takes values in $\underline{\Int}_p$. Again, we have a canonical isometric comparison isomorphism:
\[
 \bm{L}_B\otimes\Int_p\xrightarrow{\simeq}\bm{L}_p\rvert_{\Sh_{K,\Comp}^{\an}}.
\]

Similarly, let $\bm{L}_{\dR,\Rat}\subset\bm{H}_{\dR,\Rat}^{\otimes(1,1)}$ be the image of the idempotent operator $\pr_{\dR,\Rat}$. By construction, $\pr_{\dR,\Rat}$ respects the Hodge filtration and is parallel for the Gauss-Manin connection. Therefore, $\bm{L}_{\dR,\Rat}$ inherits the connection as well as a filtration $F^\bullet\bm{L}_{\dR,\Rat}$, and $(\bm{L}_{\dR,\Rat},F^\bullet\bm{L}_{\dR,\Rat})$ is a descent of $(\bm{L}_{\dR,\Comp},F^\bullet\bm{L}_{\dR,\Comp})$ as a filtered vector bundle with flat connection.

\subsection{}\label{spin:subsec:levelstructure}
Let $\bm{H}_{\Adele_f^p}$ be the $\Adele_f^p$-valued degree $1$ \'etale cohomology of $A^{\KS}_{\Sh_K}$: it can be viewed as a descent of the analytic local system $\bm{H}_B\otimes_{\Int_{(p)}}\Adele_f^p$. There is a unique idempotent operator $\pr_{\Adele_f^p}$ on $\bm{H}_{\Adele_f^p}^{\otimes(1,1)}$, such that for any prime $\ell\neq p$, the $\ell$-adic component of $\pr_{\Adele_f^p}$ is $\pr_{\ell}$. Let $\bm{L}_{\Adele_f^p}\subset\bm{H}_{\Adele_f^p}^{\otimes(1,1)}$ be the image of $\pr_{\Adele_f^p}$: Its $\ell$-adic component for any $\ell\neq p$ is simply $\bm{L}_{\ell}$.

Let $I^p_{G}$ be the sub-sheaf of $I^p(A^\KS_{\Sh_K},\lambda^\KS_{\Sh_K})$ consisting of $C$-equivariant, $\Int/2\Int$-graded isomorphisms
\[
  \eta:H\otimes\underline{\Adele}_f^p\xrightarrow{\simeq}\bm{H}_{\Adele_f^p}
\]
that carry $L\otimes\underline{\Adele}_f^p\subset H^{\otimes(1,1)}\otimes\underline{\Adele}_f^p$ onto $\bm{L}_{\Adele_f^p}\subset\bm{H}^{\otimes(1,1)}_{\Adele_f^p}$. Note that $G(\Adele_f^p)$ naturally acts on the right on $I^p_{G}$ via pre-composition, making it a torsor under $G(\Adele_f^p)$. We have a natural map of quotient sheaves
\[
  I^p_{G}/K^p\rightarrow I^p(A^\KS_{\Sh_K},\lambda^\KS_{\Rat})/\mathcal{K}^p.
\]

A section of $I^p_G/K^p$ will be called a \defnword{$K^p$-level structure}.

\begin{prp}\label{spin:prp:etalerealization}
There is a canonical $K^p$-level structure $[\eta_G]\in H^0(\Sh_K,I^p_G/K^p)$ such that $[\eta^\KS]$ is its image in $I^p(A^\KS_{\Sh_K},\lambda^\KS_{\Sh_K})/\mathcal{K}^p$.
\end{prp}
\begin{proof}
 Consider the pro-\'etale cover $\mathcal{S}_{\mathcal{K}_p}\to\mathcal{S}_{\mathcal{K}}$. Here,
 \[
  \mathcal{S}_{\mathcal{K}_p}=\varprojlim_{\mathcal{K}^{',p}\subset \mathcal{K}^p}\mathcal{S}_{\mathcal{K}_p\mathcal{K}^{',p}},
 \]
 where $\mathcal{K}^{',p}$ varies over the compact open sub-groups of $\mathcal{K}^p$. Then there is a canonical isomorphism of $\mathcal{G}_{\delta}(\Adele_f^p)$-torsors:
 \[
  I^p(A,\lambda)\xrightarrow{\simeq}\mathcal{S}_{\mathcal{K}_p}\times^{\mathcal{K}^p}\mathcal{G}_{\delta}(\Adele_f^p)\coloneqq (\mathcal{S}_{\mathcal{K}_p}\times\mathcal{G}_{\delta}(\Adele_f^p))/\mathcal{K}_p.
 \]
 The canonical $\mathcal{K}$-level structure $[\eta]$ over $\mathcal{S}_{\mathcal{K}}$ is now obtained from the tautological section $\eta\in H^0(\mathcal{S}_{\mathcal{K}_p},I^p(A,\lambda))$.

 Similarly, the $G(\Adele_f^p)$-torsor $I^p_G$ over $\Sh_K$ has a canonical reduction of structure group to a $K^p$-torsor pro-represented by the pro-finite cover $\Sh_{K_p}\to\Sh_K$ where:
 \[
  \Sh_{K_p}=\varprojlim_{K^{',p}\subset K^p}\Sh_{K_pK^{',p}}.
 \]
 Here, $K^{',p}$ varies over the compact open sub-groups of $K^p$.

 Therefore, the image $[\eta_G]\in H^0(\Sh_K,I^p_G/K^p)$ of the tautological section $\eta_G\in H^0(\Sh_{K_p},I^p_G)$ is the $K^p$-level structure we seek.
\end{proof}

\subsection{}\label{spin:subsec:derhamtorsor}
Notice that, for any algebraic representation $U$ of $G$, the $G(\Adele_f^p)$-torsor $I^p_G$ allows us to functorially descend the $\Adele_f^p$-adic local system $\bm{U}_B\otimes\Adele_f^p$ over $\Sh_K$. We take it to be the contraction product:
\[
 \bm{U}_{\Adele_f^p}\coloneqq I^p_G\times^{G(\Adele_f^p)}U(\Adele_f^p).
\]

Similarly, the $K_p$-torsor $I_{p,G}\coloneqq\Sh_{K^p}\to\Sh_{K}$ used in the proof of (\ref{spin:prp:prdescent}) allows us to functorially descend the $p$-adic sheaf $\bm{U}_B\otimes\Int_p$:
\[
 \bm{U}_p\coloneqq I_{p,G}\times^{G(\Int_p)}U(\Int_p).
\]

We can also descend the filtered vector bundle with connection $(\bm{U}_{\dR,\Comp},F^{\bullet}\bm{U}_{\dR,\Comp})$ canonically to a pair $(\bm{U}_{\dR,\Rat},F^\bullet\bm{U}_{\dR,\Rat})$ over $\Sh_K$. To do this, consider the functor $\mathcal{P}_{\dR,\Rat}$ on $\Sh_K$-schemes that assigns to any $\Sh_K$-scheme $T$ the set:
\[
  \mathcal{P}_{\dR,\Rat}(T)=\begin{pmatrix}
    C\text{-equivariant $\Int/2\Int$-graded $\Reg{T}$-module isomorphisms }\\\xi:H\otimes_{\Int_{(p)}}\Reg{T}\xrightarrow{\simeq}\bm{H}_{\dR,\Rat,T}\\\text{carrying $L\otimes\Reg{T}\subset\bm{H}^{\otimes(2,2)}\otimes\Reg{T}$ onto $\bm{L}_{\dR,\Rat,T}\subset\bm{H}^{\otimes(2,2)}_{\dR,\Rat,T}$}.
  \end{pmatrix}
\]
$G_{\Rat}$ acts on $\mathcal{P}_{\dR,\Rat}$ by pre-composition. By working over $\Sh_{K,\Comp}^{\an}$, we can show that $\mathcal{P}_{\dR,\Rat}$ is a $G_{\Rat}$-torsor over $\Sh_K$. Observe that the connection on $\bm{H}_{\dR,\Rat}$ equips $\mathcal{P}_{\dR,\Rat}$ with an integrable connection. For any representation $U$ as above, we can now take $\bm{U}_{\dR,\Rat}$ to be the contraction product:
\[
 \bm{U}_{\dR,\Rat}=\mathcal{P}_{\dR,\Rat}\times^{G_{\Rat}}U\coloneqq(\mathcal{P}_{\dR,\Rat}\times_{\Spec\Rat}U)/G_{\Rat}.
\]
Here, we are viewing $U$ as a vector bundle over $\Spec\Rat$, and $G_{\Rat}$ acts diagonally on the product $\mathcal{P}_{\dR,\Rat}\times U$.

To construct the filtration $F^\bullet\bm{U}_{\dR,\Rat}$, we fix an isotropic line $F^1L\subset L$. The stabilizer of this line is a parabolic sub-group $P_{\Rat}\subset G_{\Rat}$. Now consider the sub-functor $\mathcal{P}_{\dR,P_{\Rat}}\subset\mathcal{P}_{\dR,\Rat}$ given by:
\[
 \mathcal{P}_{\dR,P_{\Rat}}(T)=\bigl(\xi\in\mathcal{P}_{\dR,\Rat}(T):\;\xi(F^1L\otimes\Reg{T})=F^1\bm{L}_{\dR,\Rat,T}\bigr).
\]
This is a $P_{\Rat}$-torsor over $\Sh_K$.

Choose any co-character $\mu:\Gmh{\Rat}\to P_{\Rat}$ splitting the $2$-step filtration $0\subset F^1H_{\Rat}\subset H_{\Rat}$ with $F^1H_{\Rat}=\ker(F^1L_{\Rat})$ (cf.~\ref{cliff:subsec:parabolic}). On all $G_{\Rat}$-representations $U$, the action of $\mu(\Gm)$ produces a grading and hence a filtration $F^\bullet U_{\Rat}$ that is stabilized by $P_{\Rat}$, and is independent of the choice of $\mu$. We now set:
\[
 F^\bullet\bm{U}_{\dR,\Rat}\coloneqq\mathcal{P}_{\dR,P_{\Rat}}\times^{P_{\Rat}}F^\bullet U.
\]

The pair $(\bm{U}_{\dR,\Rat},F^\bullet\bm{U}_{\dR,\Rat})$ is the descent that we seek.

It is easy to check that when $U=H$ (for which it is essentially tautological) or $U=L$, these constructions agree with the ones already given above.

In sum, we have defined a functor $\bb{V}_{\Rat}$ from the category of algebraic representations $U$ of $G$ to the category of tuples
\[
\bigl(\bm{U}_B,(\bm{U}_p,\alpha_p),(\bm{U}_{\Adele^p_f},\alpha_{\Adele_f^p}),(\bm{U}_{\dR,\Rat},F^\bullet\bm{U}_{\dR,\Rat},\alpha_{\dR})\bigr),
\]
where $\bm{U}_B$ is a local system over $\Sh_{K,\Comp}^{\an}$; $\bm{U}_p$ (resp. $\bm{U}_{\Adele^p_f}$) are locally constant sheaves over $\Sh_K$ equipped with comparison isomorphisms:
\begin{align*}
 \alpha_p:\bm{U}_B\otimes\Int_p&\xrightarrow{\simeq}\bm{U}_p\rvert_{\Sh_{K,\Comp}^{\an}};\\
 \alpha_{\Adele_f^p}:\bm{U}_B\otimes\Adele^p_f&\xrightarrow{\simeq}\bm{U}_{\Adele_f^p}\rvert_{\Sh_{K,\Comp}^{\an}}.
\end{align*}
Moreover, $(\bm{U}_{\dR,\Rat},F^\bullet\bm{U}_{\dR,\Rat})$ is a filtered vector bundle over $\Sh_K$ with integrable connection, equipped with a parallel comparison isomorphism:
\begin{align*}
  \alpha_{\dR}:\bm{U}_B\otimes\Reg{\Sh_{K,\Comp}^{\an}}&\xrightarrow{\simeq}\bm{U}_{\dR,\Rat}\rvert_{\Sh_{K,\Comp}^{\an}}.
\end{align*}
Finally, we require that $\bigl(\bm{U}_B,\alpha_{\dR}^{-1}(F^\bullet\bm{U}_{\dR,\Rat}\rvert_{\Sh_{K,\Comp}^{\an}})\bigr)$ is a variation of Hodge structures over $\Sh_{K,\Comp}^{\an}$.

The functor $\bb{V}_{\Comp}$ from (\ref{spin:subsec:variations}) factors through $\bb{V}_{\Rat}$ in the obvious way.

\section{Integral canonical models I: the self-dual case}~\label{sec:self-dual}

Throughout this section, we will assume that $(L,Q)$ is a self-dual quadratic space over $\Int_{(p)}$.

\subsection{}\label{spin:subsec:intcan}
We return now to the Kuga-Satake map  $\Sh_K\to\Sh_{\mathcal{K}}$. We will always assume that $K^p$ and $\mathcal{K}^p$ are chosen such that $\Sh_{\mathcal{K}}$ admits the above description as a fine moduli scheme over $\Rat$ with integral model $\mathcal{S}_{\mathcal{K}}$ over $\Int_{(p)}$.

\begin{defn}\label{spin:defn:extension}
A pro-scheme $X$ over $\Int_{(p)}$ satisfies the \defnword{extension property} if, for any regular, locally healthy $\Int_{(p)}$-scheme $S$, any map $S\otimes\Rat\to X$ extends to a map $S\to X$.
\end{defn}

Let $\Sh_{K_p}$ (resp. $\Sh_{K_{0,p}}$) be the pro-variety attached to the inverse system $\{\Sh_{K_pK^p}\}$ (resp. $\{\Sh_{K_{0,p}K^p_0}\}$); here, $K^p$ varies over the compact open sub-groups of $G(\Adele_f^p)$.

\begin{defn}\label{spin:defn:intcanonical}
A model $\Ss_{K_p}$ (resp. $\Ss_{K_{0,p}}$) for $\Sh_{K_p}$ (resp. $\Sh_{K_{0,p}}$) over $\Int_{(p)}$ is an \defnword{integral canonical model} if it is regular, locally healthy, and has the extension property. If $\Ss_{K_p}$ is an integral canonical model for $\Sh_{K_p}$, and if $K=K_pK^p$ is a compact open in $G(\Adele_f)$, then we will call $\Ss_K\coloneqq \Ss_{K^p}/K^p$ the \defnword{integral canonical model} for $\Sh_K$. A similar convention will hold for quotients $\Ss_{K_0}\coloneqq\Ss_{K_{0,p}}/K^p_0$ by compact open sub-group $K^p_0\subset G_0(\Adele_f^p)$.
\end{defn}

For any compact open sub-group $\mathcal{K}'\subset \mathcal{K}$, with $\mathcal{K}'_p=\mathcal{K}_p$, set $K'=\mathcal{K}'\cap G(\Adele_f)$. Let $\Ss_{K'}$ be the normalization of the Zariski closure of $\Sh_{K'}$ in $\mathcal{S}_{\mathcal{K'}}$. Consider the pro-scheme:
\begin{align*}
  \Ss_{K_p}&=\varprojlim_{K'\subset K}\Ss_{K'}.
\end{align*}
Here, $K'$ varies over compact open sub-groups of $K$ with $K'_p=K_p$.

We have the following result, due (independently) to Kisin~\cite{kis3}*{2.3.8, 3.4.14}, and Vasiu~\cite{vasiu:preab}.
\begin{thm}\label{spin:thm:kisin}
$\Ss_{K_p}$ is a smooth integral canonical model for $\Sh_{K_p}$ over $\Int_{(p)}$. Moreover, the finite Galois cover $\Sh_{K_p}\to\Sh_{K_{0,p}}$ extends to a pro-finite Galois cover $\Ss_{K_p}\to\Ss_{K_{0,p}}$, where $\Ss_{K_{0,p}}$ is a smooth integral canonical model for $\Sh_{K_{0,p}}$ over $\Int_{(p)}$.
\end{thm}
\qed

Though we will not present the proof in its entirety, we will need some of its ingredients, which we describe now. They are extracted mainly from~\cite{kis3}.

\subsection{}\label{spin:subsec:cris}
By construction, the polarized abelian scheme $(A^{\KS}_{\Sh_K},\lambda^{\KS}_{\Sh_K})$ extends to a polarized abelian scheme $(A^{\KS},\lambda^{\KS})$ over $\Ss_K$. By the theory of Ner\'on models (cf.~\cite{falchai}*{I.2.7}), $A^{\KS}$ has a unique $\Int/2\Int$-grading and $C$-action extending those on $A^{\KS}_{\Sh_K}$.

For any $\Field_p$-scheme $S$, let $(S/\Int_p)_{\cris}$ be the big crystalline site for $S$ over $\Spec~\Int_p$ (cf.~\cite{berthelot_messing}*{p. 178}), and let $\Reg{S}^{\cris}$ be the structure sheaf of $(S/\Int_p)_{\cris}$. Recall that an object in $(S/\Int_p)_{\cris}$ is a triple $(U,T,\gamma)$, where $U$ is an $S$-scheme, $U\into T$ is a nilpotent thickening of $\Int_p$-schemes with ideal of definition $\J_{(U\into T)}$, and $\gamma$ is a divided power structure on $\J_{(U\into T)}$ that is compatible with the natural divided power structure on the ideal $p\Reg{T}$. For any sheaf $G$ over $(S/\Int_p)_{\cris}$, and any object $(U,T,\gamma)$ in $(S/\Int_p)_{\cris}$, we denote by $G_T$ the restriction of $G$ to the fppf site over $T$.

Let $A^\KS_{\Field_p}$ be the fiber of $A^\KS$ over $\Ss_{K,\Field_p}$. Let $\bm{H}_{\cris}$ be the first crystalline cohomology of $A^\KS_{\Field_p}$ over $\Ss_{K,\Field_p}$. This is a crystal of locally free $\Reg{\Ss_{K,\Field_p}}^{\cris}$-modules over $(\Ss_{K,\Field_p}/\Int_p)_{\cris}$. Let $\hat{\Ss}_{K,\Int_p}$ be the completion of $\Ss_K$ along its special fiber. We then have a natural identification
\[
  \bm{H}_{\dR}\rvert_{\hat{\Ss}_{K,\Int_p}}=\varprojlim_n\bm{H}_{\cris,\Ss_{K,\Int/p^n}}
\]
of coherent sheaves over $\hat{\Ss}_{K,\Int_p}$. In other words, for any $\Ss_K$-scheme $T$ in which $p$ is nilpotent, we have a canonical identification of coherent sheaves
\[
  \bm{H}_{\dR}\rvert_T=\bm{H}_{\cris,T}.
\]

$\bm{H}_{\cris}$ has more structure: it is an $F$-crystal. More precisely, let $\Fr$ be the absolute Frobenius endomorphism on $\Ss_{K,\Field_p}$. Then $\Fr^*\bm{H}_{\cris}$ is identified with the relative crystalline cohomology of the Frobenius pull-back $\Fr^*A^\KS_{\Field_p}$, and the relative Frobenius map $A^\KS_{\Field_p}\to\Fr^*A^\KS_{\Field_p}$ induces a map of crystals
\[
  \bm{F}:\Fr^*\bm{H}_{\cris}\rightarrow\bm{H}_{\cris}.
\]
If $T$ is any $\Ss_{K,\Field_p}$-scheme, we get an induced map of coherent sheaves
\[
  \bm{F}:\Fr^*\bm{H}_{\dR}\rvert_T\rightarrow\bm{H}_{\dR}\rvert_{T}.
\]
The kernel of this map is precisely the Hodge filtration $\Fr^*F^1\bm{H}_{\dR}\rvert_{T}\subset\Fr^*\bm{H}_{\dR}\rvert_{T}$.

\subsection{}
Suppose that we have a point $s\in\Ss_K(k)$, where $k$ is a perfect field of characteristic $p$. Set $W=W(k)$; then the restriction $\bm{H}_{\cris}$ (resp. $\bm{H}_{\cris}^{\otimes(1,1)}$) to $(\Spec k/\Int_p)_{\cris}$ corresponds to the ${W}$-module $H^1_{\cris}(A^\KS_s/{W})$ (resp. $\End\bigl(H^1_{\cris}(A^\KS_s/{W})\bigr)$). If $\sigma$ is the canonical Frobenius lift on ${W}$, $\bm{F}$ induces a map
\[
  \bm{F}_s:\sigma^*\bm{H}_{\cris,s}\rightarrow \bm{H}_{\cris,s},
\]
giving $\bm{H}_{\cris,s}$ the structure of an $F$-crystal over ${W}$. Conjugation by $\bm{F}_s$ induces an $F$-isocrystal structure on $\bm{H}_{\cris,s}[p^{-1}]^{\otimes(1,1)}$.

\begin{prp}\label{spin:prp:kisinpadichodge}
Assume that $W$ admits an embedding into $\Comp$. Then there exists a canonical $F$-invariant tensor $\pr_{\cris,s}\in\bm{H}_{\cris,s}^{\otimes(2,2)}$ such that, given any finite extension $E/W_{\Rat}$, an algebraic closure $\overline{E}/E$ of $E$, and any lift $\widetilde{s}\in\Ss_K(\Reg{E})$ of $s$, the following properties hold:
\begin{enumerate}[itemsep=0.12in]
  \item\label{intphodge:bo}Let $\widetilde{s}_E\in\Sh_K(E)$ be the attached $E$-valued point. Under the Berthelot-Ogus comparison isomorphism
 \[
  \bm{H}_{\dR,\tilde{s}_E}\xrightarrow{\simeq}\bm{H}_{\cris,s}\otimes_WE
 \]
 $\pr_{\dR,\Rat,\widetilde{s}_E}\otimes 1$ is carried to $\pr_{\cris,s}\in\bm{H}_{\cris,s}^{\otimes(2,2)}$.
  \item\label{intphodge:comp}Let $\widetilde{s}_{\overline{E}}\in\Sh_K(\overline{E})$ be the attached $\overline{E}$-valued point. The crystalline comparison isomorphism
\[
\bm{H}_{p,\widetilde{s}_{\overline{E}}}\otimes_{\Int_p}\Bcris\xrightarrow{\simeq}\bm{H}_{\cris,s}\otimes_{{W}}\Bcris
\]
respects grading, is $C$-equivariant and carries $\pr_{p,\widetilde{s}_{\overline{E}}}\otimes 1$ to $\pr_{\cris,s}\otimes 1$.

\item\label{intphodge:trivial}There exists a $C$-equivariant isomorphism of $\Int/2\Int$-graded ${W}$-modules
  \[
    H\otimes_{\Int_{(p)}}{W}\xrightarrow{\simeq}\bm{H}_{\cris,s}
  \]
  carrying $\pr$ to $\pr_{\cris,s}$.
\item\label{intphodge:vcris}$\pr_{\cris,s}$ is an idempotent projector on $\bm{H}_{\cris,s}^{\otimes(1,1)}$. Set
  \[
  \bm{L}_{\cris,s}=\im~\pr_{\cris,s}\subset\bm{H}_{\cris,s}^{\otimes(1,1)}.
  \]
  Then $\bm{L}_{\cris,s}$ is a self-dual quadratic space over ${W}$ that is isometric to $L\otimes{W}$.
\item\label{intphodge:hodge}Set $\bm{L}_{\dR,s}=\bm{L}_{\cris,s}\otimes k$; then the Hodge filtration $F^1\bm{H}_{\dR,s}\subset\bm{H}_{\dR,s}$ is $\GSpin(\bm{L}_{\dR,s})$-split. More precisely, there exists a canonical isotropic line $F^1\bm{L}_{\dR,s}\subset\bm{L}_{\dR,s}$ such that
    \[
    F^1\bm{H}_{\dR,s}=\ker\bigl(F^1\bm{L}_{\dR,s}\bigr)=\im\bigl(F^1\bm{L}_{\dR,s}\bigr).
    \]
\end{enumerate}
\end{prp}
\begin{proof}
  Fix $E/\Rat_p$ and $\widetilde{s}$ as in the proposition. Then $\pr_{p,\widetilde{s}_{\overline{E}}}$ is a Galois-invariant tensor in $\bm{H}_{p,\widetilde{s}_{\overline{E}}}^{\otimes(2,2)}$. Therefore, under the crystalline comparison isomorphism, it is carried to an $F$-invariant tensor $\pr_{\cris,s}\in\bm{H}_{\cris,s}^{\otimes(2,2)}\bigl[p^{-1}\bigr]$. In fact, $\pr_{\cris,s}$ belongs to $\bm{H}_{\cris,s}^{\otimes(2,2)}$~\cite{kis3}*{1.3.6(1),1.4.3(1)}. We have to show that it has the desired properties.

  The comparison isomorphisms are functorial and are therefore $C$-equivariant and respect gradings. The de Rham comparison isomorphism
  \[
   \bm{H}_{p,\widetilde{s}_{\overline{E}}}\otimes_{\Int_p}\Bdr\xrightarrow{\simeq}\bm{H}_{\dR,\tilde{s}_E}\otimes_{E}\Bdr
  \]
  respects $\Int/2\Int$-grading, is $C$-equivariant and carries $\pi_{p,\widetilde{s}_{\overline{E}}}\otimes 1$ to $\pi_{\dR,\Rat,\widetilde{s}_E}\otimes 1$. This is a result of Blasius-Wintenberger~\cite{blasius}, when $\widetilde{s}_E$ arises from a point valued in a number field. For the generality we need, cf.~\cite{moonen}*{5.6.3}.

  The crystalline and de Rham comparison isomorphisms are compatible with the Berthelot-Ogus isomorphism. This implies that the isomorphism in (\ref{intphodge:bo}) carries $\pr_{\dR,\Rat,\widetilde{s}_E}$ to $\pr_{\cris,s}$. In particular, (\ref{intphodge:bo}) holds for all lifts $\widetilde{s}$ if and only if (\ref{intphodge:comp}) does. That (\ref{intphodge:bo}) is true for all lifts follows from a parallel transport argument; cf. the proof of \cite{kis3}*{2.3.5}.

  Now, \cite{kis3}*{1.4.3(3)} shows that there exists a $C$-equivariant $\Int/2\Int$-graded isomorphism
  \[
   \bm{H}_{p,\widetilde{s}_{\overline{E}}}\otimes_{\Int_p}{W}\xrightarrow{\simeq}\bm{H}_{\cris,s}
  \]
  carrying $\pr_{p,\widetilde{s}_{\overline{E}}}$ to $\pr_{\cris,s}$. The main input is (\ref{cliff:lem:pr}), which shows that $\GSpin(L)$ is the point-wise stabilizer in $\GL_{C}^+(H)$ of $\pr$.

  So to show (\ref{intphodge:trivial}) we only have to observe that there exists a $C$-equivariant graded isomorphism
  \[
   H\otimes_{\Int_{(p)}}\Int_p\xrightarrow{\simeq}\bm{H}_{p,\widetilde{s}_{\overline{E}}}
  \]
  carrying $\pr$ to $\pr_{p,\widetilde{s}_{\overline{E}}}$. Indeed, we can fix an embedding $\sigma:\overline{E}\into\Comp$ and identify $\bm{H}_{p,\widetilde{s}_{\overline{E}}}$ with the Betti cohomology $\bm{H}_{B,\sigma(\widetilde{s}_{\overline{E}})}\otimes\Int_p$. Now the assertion is clear from the construction of the sheaf $\bm{H}_B$ in (\ref{spin:subsec:variations}).

  (\ref{intphodge:vcris}) is immediate from (\ref{intphodge:trivial}), and the first assertion of (\ref{intphodge:hodge}) follows from \cite{kis3}*{1.4.3(4)}. The second assertion of (\ref{intphodge:hodge}) follows from the discussion in (\ref{cliff:subsec:parabolic}). Note that we are viewing $\bm{L}_{\dR,s}$ as a sub-space of $\bm{H}_{\dR,s}^{\otimes(1,1)}$ via the canonical identification $\bm{H}_{\cris,s}\otimes k=\bm{H}_{\dR,s}$, and we are identifying $\GSpin(\bm{L}_{\dR,s})$ with the sub-group of $C$-equivariant, graded isomorphisms of $\bm{H}_{\dR,s}$, which preserve $\bm{L}_{\dR,s}$ under conjugation.
\end{proof}

\subsection{}\label{spin:subsec:strongdiv}
Suppose that we are given a lift $\widetilde{s}\in\Ss_K(W)$. This equips $\bm{H}_{\cris,s}$ with a filtration $F^1\bm{H}_{\cris,s}$: It is the pull-back of the Hodge filtration $F^1\bm{H}_{\dR,\widetilde{s}}\subset\bm{H}_{\dR,\widetilde{s}}$. This filtration is \defnword{strongly divisible}. That is, we have:
\[
 \bm{F}_s\biggl(\sigma^*\bigl(p^{-1}F^1\bm{H}_{\cris,s}+\bm{H}_{\cris,s}\bigr)\biggr)=\bm{H}_{\cris,s}.
\]
This is equivalent to the assertion that $\on{Fr}^*F^1\bm{H}_{\dR,s}$ is the kernel of the mod-$p$ Frobenius $\on{Fr}^*\bm{H}_{\dR,s}\to\bm{H}_{\dR,s}$.

If we now endow $\bm{H}^{\otimes(1,1)}_{\cris,s}$ with its induced filtration and $\bm{H}_{\cris,s}^{\otimes(1,1)}[p^{-1}]$ with the conjugation action of $\bm{F}_s$, then we again obtain a strongly divisible module, in the sense that the following identity holds (cf.~\cite{laffaille}*{4.2}):
\begin{align}\label{spin:eqn:strongdivh}
    \bm{F}_s\biggl(\sigma^*\bigl(p^{-1}F^1\bm{H}^{\otimes(1,1)}_{\cris,s}+F^0\bm{H}^{\otimes(1,1)}_{\cris,s}+p\bm{H}^{\otimes(1,1)}_{\cris,s}\bigr)\biggr)=\bm{H}^{\otimes(1,1)}_{\cris,s}.
\end{align}
Applying the $\bm{F}$-equivariant, filtration preserving projector $\pr_{\cris,s}$ to (\ref{spin:eqn:strongdivh}) now gives us:
\begin{align}\label{spin:eqn:strongdiv}
    \bm{F}_s\biggl(\sigma^*\bigl(p^{-1}F^1\bm{L}_{\cris,s}+F^0\bm{L}_{\cris,s}+p\bm{L}_{\cris,s}\bigr)\biggr)=\bm{L}_{\cris,s}.
\end{align}

\subsection{}\label{spin:subsec:explicit}
Choose any co-character $\mu_0:\Gm\otimes k\rightarrow\GSpin(\bm{L}_{\dR,s})$ splitting the Hodge filtration, and let $\mu:\Gm\otimes{W}\rightarrow\GSpin(\bm{L}_{\cris,s})$ be any lift of $\mu_0$. It determines a lift $F^1\bm{H}_{\cris,s}\subset\bm{H}_{\cris,s}$ of the Hodge filtration as well as a splitting $\bm{H}_{\cris,s}=F^1\bm{H}_{\cris,s}\oplus\overline{F}^1\bm{H}_{\cris,s}$. Since $\mu$ factors through $\GSpin(\bm{L}_{\cris,s})$, it induces a splitting
\[
 \bm{L}_{\cris,s}=F^1\bm{L}_{\cris,s}\oplus\bm{L}^0_{\cris,s}\oplus\overline{F}^1\bm{L}_{\cris,s},
\]
where $F^1\bm{L}_{\cris,s}\subset\bm{L}_{\cris,s}$ is an isotropic line lifting $F^1\bm{L}_{\dR,s}$. We again have:
\[
 F^1\bm{H}_{\cris,s}=\ker(F^1\bm{L}_{\cris,s})=\im(F^1\bm{L}_{\cris,s}).
\]
As usual, we are viewing $\bm{L}_{\cris,s}$ as a space of endomorphisms of $\bm{H}_{\cris,s}$.

Let $U\subset\GL(\bm{H}_{\cris,s})$ be the opposite unipotent attached to this splitting: Namely, it is the unipotent radical of the parabolic sub-group attached to the filtration $\overline{F}^1\bm{H}_{\cris,s}$. Its intersection $U_G$ with $\GSpin(\bm{L}_{\cris,s})$ is again the unipotent radical of a parabolic sub-group of $\GSpin(\bm{L}_{\cris,s})$.

Let $\widehat{U}$ (resp. $\widehat{U}_G$) be the completion of $U$ (resp. $U_G$) along the identity section. Let $R$ (resp. $R_G$) be the ring of formal functions on $\widehat{U}$ (resp. $\widehat{U}_G$); then we have a surjection $R\twoheadrightarrow R_G$ of formally smooth ${W}$-algebras.

Set $\bm{H}_R=\bm{H}_{\cris,s}\otimes_{{W}}R$, and equip it with the constant filtration arising from the filtration $F^1\bm{H}_{\cris,s}$. Choose compatible isomorphisms
\[
 R\xrightarrow{\simeq}{W}\pow{t_1,\ldots,t_r};\;\;R_G\xrightarrow{\simeq}{W}\pow{t_1,\ldots,t_d}
\]
such that the identity sections are identified with the maps $t_i\mapsto 0$. Equip $R$ with the Frobenius lift $\varphi:t_i\mapsto t_i^p$. Equip $\bm{H}_R$ with the map:
\[
 \bm{F}_R:\varphi^*\bm{H}_R=\sigma^*\bm{H}_{\cris,s}\otimes_{{W}}R\xrightarrow{\bm{F}_s\otimes 1}\bm{H}_{\cris,s}\otimes_{{W}}R=\bm{H}_R\xrightarrow{g}\bm{H}_R,
\]
where $g\in\widehat{U}(R)$ is the tautological element.

By \cite{moonen}*{Thm. 4.4}, there is a unique topologically quasi-nilpotent integrable connection $\nabla_R:\bm{H}_R\to\bm{H}_R\otimes_R\fdiff{R/W}$ for which $\bm{F}_R$ is parallel. The tuple $(\bm{H}_R,\bm{F}_R,\nabla_R)$ determines, and is determined by, a unique $F$-crystal over $(\Spec(R\otimes\Field_p)/\Int_p)_{\cris}$, which we will denote by $\bm{H}$. The evaluation of $\bm{H}$ on the pro-nilpotent divided power thickening $\Spec(R\otimes\Field_p)\into\Spf R$ is identified with $\bm{H}_R$. Similarly, a change of scalars along $R\to R_G$ gives us a tuple $(\bm{H}_{R_G},\bm{F}_{R_G},\nabla_{R_G},F^1\bm{H}_{R_G})$ over $R_G$. This corresponds to a unique $F$-crystal $\bm{H}_G$ over $(\Spec(R_G\otimes\Field_p)/\Int_p)_{\cris}$.

Observe that $\bm{H}_{R_G}=\bm{H}_{\cris,s}\otimes_WR_G$ is equipped with the constant tensor $\pr_{R_G}=\pr_{\cris,s}\otimes 1\in\bm{H}_{R_G}^{\otimes(2,2)}$. We can view $\pr_{R_G}$ as an idempotent operator on $\bm{H}_{R_G}^{\otimes(1,1)}$. Write $\bm{L}_{R_G}$ for its image: this is a direct summand of $\bm{H}_{R_G}^{\otimes(1,1)}$, which can be identified with $\bm{L}_{\cris,s}\otimes R_G$.

\begin{prp}\label{spin:prp:moonen}
\mbox{}
\begin{enumerate}
\item~\label{moonen:tensor}The tensor $\pr_{R_G}\in\bm{H}_{R_G}^{\otimes(2,2)}$ is parallel, lies in $F^0\bm{H}_{R_G}^{\otimes(2,2)}$, and is $F$-invariant in $\bm{H}_{R_G}^{\otimes(2,2)}\bigl[p^{-1}\bigr]$.
\item~\label{moonen:strongdiv}The direct summand $\bm{L}_{R_G}\subset\bm{H}_{R_G}^{\otimes(1,1)}$ is stable under the connection $\nabla_{R_G}$ and $\bm{L}_{R_G}\bigl[p^{-1}\bigr]$ is stable under the conjugation action of $\bm{F}_{R_G}$. Furthermore, $\bm{L}_{R_G}$ with its induced filtration $F^\bullet\bm{L}_{R_G}$ is strongly divisible:
    \[
    \bm{F}_{R_G}\biggl(\varphi^*\bigl(p^{-1}F^1\bm{L}_{R_G}+F^0\bm{L}_{R_G}+p\bm{L}_{R_G}\bigr)\biggr)=\bm{L}_{R_G}
    \]
\end{enumerate}
\end{prp}
\begin{proof}
  (\ref{moonen:tensor}) is essentially by construction; cf.~\cite{moonen}*{4.8}. (\ref{moonen:strongdiv}) also follows by construction: We can identify $(\bm{L}_{R_G},F^\bullet\bm{L}_{R_G})$ with $(\bm{L}_{\cris,s},F^\bullet\bm{L}_{\cris,s})\otimes_WR_G$, so that $\bm{F}_{R_G}$ is identified with $g\circ (\bm{F}_s\otimes 1)$. Here, $g\in U_G(R_G)$ is the tautological element. The result now is a consequence of (\ref{spin:eqn:strongdiv}).
\end{proof}

\subsection{}
Now assume that $k=\overline{\Field}_p$, so that $s$ is a closed point. Set $R_s=\widehat{\Rg}_{\Ss_K,s}$, and let $\widehat{U}_{s}$ be the completion of $\Ss_K$ at $s$. Let $\widehat{U}'=\Spf R'$ be the universal deformation space for the abelian variety $A^\KS_s$; then $\widehat{U}_s$ is naturally identified with a closed formal sub-scheme of $\widehat{U}'$. Restricting $\bm{H}_{\cris}$ to $\Spec(R'\otimes\Field_p)$ gives rise to an $F$-crystal $\bm{H}'$ over $(\Spec(R'\otimes\Field_p)/\Int_p)_{\cris}$. Evaluating $\bm{H}'$ along the pro-nilpotent divided power thickening $\Spec(R'\otimes\Field_p)\into \Spf(R')$ gives us a finite free $R'$-module $\bm{H}_{R'}$, which can be identified with the evaluation of $\bm{H}_{\dR}$ on $\Spec R'$. The Hodge filtration on $\bm{H}_{\dR}$ equips $\bm{H}_{R'}$ with a direct summand $F^1\bm{H}_{R'}\subset\bm{H}_{R'}$.

The $W$-algebra $R$ admits the identity section $j:R\to W$ with $t_i\mapsto 0$ for all $i$. Also, by Grothendieck-Messing theory~\cite{messing}*{V.1.6}, the lift $F^1\bm{H}_{\cris,s}\subset\bm{H}_{\cris,s}$ corresponds to a lift $\tilde{s}:\Spf W\to\Spf(R')$ of $s$. In fact, we have a canonical isomorphism of $W$-modules $\bm{H}_{\dR,\tilde{s}}\xrightarrow{\simeq}\bm{H}_{\cris,s}$ that identifies $F^1\bm{H}_{\dR,\tilde{s}}$ with $F^1\bm{H}_{\cris,s}$.
Write $j':R'\to W$ for the section corresponding to the lift $\tilde{s}$.

By construction, we now have canonical isomorphisms: $\iota_0:\bm{H}_R\otimes_{R,j}W\xrightarrow{\simeq}\bm{H}_{\cris,s}$ and $\iota'_0:\bm{H}_{R'}\otimes_{R',j'}W\xrightarrow{\simeq}\bm{H}_{\cris,s}$.

\begin{thm}\label{spin:thm:kisinprec}
\mbox{}
\begin{enumerate}[itemsep=0.11in]
\item~\label{kisinprec:faltings}There exists an isomorphism (necessarily unique) of augmented $W$-algebras $f:(R,j)\xrightarrow{\simeq} (R',j')$ with the following property: Let $f^*:(\Spec(R\otimes\Field_p)/\Int_p)_{\cris}\to (\Spec(R'\otimes\Field_p)/\Int_p)_{\cris}$ be the induced map on crystalline sites. Then there exists an isomorphism (necessarily unique) of $F$-crystals $\alpha:f^*\bm{H}\xrightarrow{\simeq}\bm{H}'$ such that:
    \begin{enumerate}
    \item The induced isomorphism $\alpha_R:\bm{H}_{R}\otimes_{R,f}R'\xrightarrow{\simeq}\bm{H}_{R'}$ carries $F^1\bm{H}_{R}\otimes_{R,f}R'$ onto $F^1\bm{H}_{R'}$.
    \item The induced isomorphism $j^*\alpha_R:\bm{H}_{R}\otimes_{R,j}W\xrightarrow{\simeq}\bm{H}_{R'}\otimes_{R',j'}W$ is equal to $(\iota'_0)^{-1}\circ\iota_0$.
    \end{enumerate}
\item~\label{kisinprec:kisin}The induced isomorphism $\Spf(f):\widehat{U}'\xrightarrow{\simeq}\widehat{U}$ carries $\widehat{U}_s$ onto $\widehat{U}_G$.
\item~\label{kisinprec:tensors}The restriction of the tensor $\pr_{\dR,\Rat}$ to $\Spec(R_s\otimes\Rat)$ extends to a parallel section
\[
 \pr_{\dR,R_s}\in H^0\bigl(\Spec R_s,F^0\bm{H}_{\dR}^{\otimes(2,2)}\bigr)^{\nabla=0}.
\]
$\pr_{\dR,R_s}$ is an idempotent operator on $\bm{H}_{\dR}^{\otimes(1,1)}\rvert_{\Spec R_s}$.
\end{enumerate}
\end{thm}
\begin{proof}
  (\ref{kisinprec:faltings}) is due to Faltings; cf.~\cite{faltings}*{\S 7}. (\ref{kisinprec:kisin}) is shown during the course of the proof of \cite{kis3}*{2.3.5}.

  (\ref{kisinprec:tensors}) is a special case of~\cite{kis3}*{2.3.9}. We sketch the proof: From $f$ and $\alpha_R$, we obtain isomorphisms $f_G:R_G\xrightarrow{\simeq}R_s$, and $\alpha_{R_G}:\bm{H}_{R_G}\otimes_{R_G,f_G}R_s\xrightarrow{\simeq}\bm{H}_{\dR,R_s}$. From (\ref{spin:prp:moonen}), we then obtain a parallel tensor:
  \[
   \alpha_{R_G}(\pr_{R_G})\in H^0\bigl(\Spec R_s,F^0\bm{H}_{\dR}^{\otimes(2,2)}\bigr)^{\nabla=0}.
  \]
  By construction, the evaluation of this tensor along the map $R_s\xrightarrow{\tilde{s}}W\into W_{\Rat}$ agrees with that of $\pr_{\dR,\Rat}$. Therefore, $\alpha_{R_G}(\pr_{R_G})\vert_{\Spec(R_s\otimes\Rat)}$ and $\pr_{\dR,\Rat}\vert_{\Spec(R_s\otimes\Rat)}$ are both parallel tensors in $\bm{H}_{\dR}^{\otimes(2,2)}\vert_{\Spec(R_s\otimes\Rat)}$ that agree at a point. They must therefore agree everywhere. Thus, $\pr_{\dR,R_s}\coloneqq\alpha_{R_G}(\pr_{R_G})$ is the extension we seek in (\ref{kisinprec:tensors}).
\end{proof}

\begin{corollary}\label{spin:cor:derhamrealization}
\mbox{}
\begin{enumerate}[itemsep=0.13in]
  \item\label{derham:pr}$\pr_{\dR,\Rat}$ extends to a parallel section $\pr_{\dR}$ of $F^0\bm{H}_{\dR}^{\otimes(2,2)}$ over $\Ss_K$. Moreover, $\pr_{\dR}$ is an idempotent projector on $\bm{H}_{\dR}^{\otimes(1,1)}$, whose image is a vector sub-bundle $\bm{L}_{\dR}$ extending $\bm{L}_{\dR,\Rat}$.
  \item\label{derham:torsor}Consider the functor $\mathcal{P}_{\dR}$ on $\Ss_K$-schemes given by
\[
  \mathcal{P}_{\dR}(T)=\begin{pmatrix}
    C\text{-equivariant $\Int/2\Int$-graded $\Reg{T}$-module isomorphisms }\\\xi:H\otimes_{\Int_{(p)}}\Reg{T}\xrightarrow{\simeq}\bm{H}_{\dR,T}\\\text{carrying $L\otimes\Reg{T}\subset H^{\otimes(1,1)}\otimes\Reg{T}$ onto $\bm{L}_{\dR}\subset\bm{H}_{\dR}^{\otimes(1,1)}$}.
  \end{pmatrix}
\]
Then $\mathcal{P}_{\dR}$ is a $G$-torsor over $\Ss_K$ with generic fiber $\mathcal{P}_{\dR,\Rat}$.
\end{enumerate}
\end{corollary}
\begin{proof}
We only need to prove these properties over the formal completion $\widehat{U}_s$ at a point $s\in\Ss_K(\overline{\Field}_p)$.  (\ref{derham:pr}) is immediate from (\ref{spin:thm:kisinprec})(\ref{kisinprec:tensors}).

To prove (\ref{derham:torsor}), it is enough to show that, if we view $\Spec R_G$ as an $\Ss_K$-scheme via the maps $\Spec R_G\xrightarrow{\simeq}\Spec R_s\to\Ss_K$ of (\ref{spin:thm:kisinprec}), the set $\mathcal{P}_{\dR}(\Spec R_G)$ is non-empty.

In other words, we have to show that there exists a $C$-equiviariant, graded isomorphism of $R_G$-modules $H\otimes_{\Int_{(p)}}R_G\xrightarrow{\simeq}\bm{H}_{R_G}$ carrying $L\otimes R_G$ onto $\bm{L}_{R_G}\coloneqq\im(\pr_{R_G})$. This is immediate from (\ref{spin:prp:kisinpadichodge})(\ref{intphodge:trivial}) and the very construction of $\bm{H}_{R_G}$.
\end{proof}

\subsection{}\label{spin:subsec:crysreal}
Since $\Ss_{K,\Int_p}$ is smooth over $\Int_p$, as a crystal, $\bm{H}_{\cris}$ is determined by $\bm{H}_{\dR}\rvert_{\hat{\Ss}_{K,\Int_p}}$ equipped with its Gauss-Manin connection. In particular, evaluation along the formal PD thickening $\Ss_{K,\Field_p}\into\hat{\Ss}_{K,\Int_p}$ gives us a canonical isomorphism:
\[
H^0\bigl((\Ss_{K,\Field_p}/\Int_p)_{\cris},\bm{H}^{\otimes}_{\cris}\bigr) \xrightarrow{\simeq} H^0\bigl(\hat{\Ss}_{K,\Int_p},\bm{H}^{\otimes}_{\dR}\bigr)^{\nabla=0}.
\]

Therefore, there exists a unique global section
\[
\pr_{\cris}\in H^0\bigl((\Ss_{K,\Field_p}/\Int_p)_{\cris},\bm{H}^{\otimes(2,2)}_{\cris}\bigr),
\]
whose evaluation at $\Ss_{K,\Field_p}\into\hat{\Ss}_{K,\Int_p}$ is the restriction of the parallel section $\pr_{\dR}\in H^0\bigl(\Ss_K,\bm{H}_{\dR}^{\otimes(2,2)}\bigr)$.

Again, we can view $\pr_{\cris}$ as an idempotent endomorphism of the crystal $\bm{H}^{\otimes(1,1)}_{\cris}$, and we denote by $\bm{L}_{\cris}$ the image of $\pr_{\cris}$ in $\bm{H}^{\otimes(1,1)}_{\cris}$. If $s\to \Ss_{K,\Field_p}$ is a point valued in an algebraically closed field $k(s)$, then the restriction of $\bm{\pr}_{\cris}$ to $(\Spec k(s)/\Int_p)_{\cris}$ determines and is determined by its evaluation along the formal thickening $\Spec k(s)\into\Spf W(k(s))$.

Suppose that we have a lift $\widetilde{s}\in\Ss_K({W})$ of $s$. Then there is a natural isomorphism $\bm{H}_{\cris,s}\xrightarrow{\simeq}\bm{H}_{\dR,\widetilde{s}}$. By the definition of $\pr_{\cris}$, its evaluation over $\Spec k(s)\into\Spf W(k(s))$ must map to the tensor $\pr_{\dR,\widetilde{s}}\in\bm{H}_{\dR,\widetilde{s}}^{\otimes(2,2)}$ under this isomorphism. In particular, we find that this evaluation is exactly the tensor $\pr_{\cris,s}$ defined in (\ref{spin:prp:kisinpadichodge}). Similarly, the restriction of $\bm{L}_{\cris}$ to $(\Spec k(s)/\Int_p)_{\cris}$ determines and is determined by the $F$-isocrystal $\bm{L}_{\cris,s}$ seen in \emph{loc. cit.}

\subsection{}\label{spin:subsec:kodairaspencer}
The Gauss-Manin connection on $\bm{H}^{\otimes(2,2)}_{\dR}$ restricts to a connection on $\bm{L}_{\dR}$. The filtration $F^\bullet\bm{L}_{\dR}$ on $\bm{L}_{\dR}$ satisfies Griffith's transversality. That is, for any integer $i$, we have
      \[
      \nabla(F^i\bm{L}_{\dR})\subset F^{i-1}\bm{L}_{\dR}\otimes\Omega^1_{\Ss_K/\Int_{(p)}}.
      \]
Indeed, this can be checked over $\Comp$, where it holds by construction. Thus we obtain $\Reg{\Ss_K}$-linear Kodaira-Spencer maps:
\begin{align}\label{spin:eqn:ks1}
       \gr^0_F\bm{L}_{\dR}&\to\gr^{-1}_F\bm{L}_{\dR}\otimes\Omega^1_{\Ss_K/\Int_{(p)}};
\end{align}
\begin{align}\label{spin:eqn:ks2}
F^1\bm{L}_{\dR}&\to\gr^0_F\bm{L}_{\dR}\otimes\Omega^1_{\Ss_K/\Int_{(p)}}.
\end{align}
\begin{prp}\label{spin:prp:kodairaspencer}
The map in (\ref{spin:eqn:ks1}) is an isomorphism and the map in (\ref{spin:eqn:ks2}) is injective with image a local direct summand of the right hand side.
\end{prp}
\begin{proof}
  Both statements can be studied over the complete local ring $R_s$ at a point $s\in\Ss_K(\overline{\Field}_p)$. As in the proof of (\ref{spin:cor:derhamrealization}), we can work instead over the ring $R_G$, and study the Kodaira-Spencer maps induced by $\nabla_{R_G}$:
  \begin{align*}
    \gr^0_F\bm{L}_{R_G}\to\gr^{-1}_F\bm{L}_{R_G}\otimes\fdiff{R_G/W}\;;\; F^1\bm{L}_{R_G}\to\gr^{0}_F\bm{L}_{R_G}\otimes\fdiff{R_G/W}.
  \end{align*}

  We can identify $\fdiff{R_G/W}$ canonically with $\Hom(\Lie U_G,R_G)$, and so we can rewrite these maps equivalently as:
  \begin{align*}
    \Lie U_G\otimes R_G\to\Hom(\gr^0_F\bm{L}_{R_G},\gr^{-1}_F\bm{L}_{R_G})\;;\; \Lie U_G\otimes R_G\to\Hom(F^1\bm{L}_{R_G},\gr^0_F\bm{L}_{R_G}).
  \end{align*}
  The proposition now amounts to showing that both these maps are isomorphisms. By Nakayama's lemma, it suffices to do so after a change of scalars along the augmentation map $j:R_G\to W$. But then we obtain maps:
  \begin{align*}
    \Lie U_G\to\Hom(\gr^0_F\bm{L}_{\cris,s},\gr^{-1}_F\bm{L}_{\cris,s})\;;\;\Lie U_G\to\Hom(F^1\bm{L}_{\cris,s},\gr^0_F\bm{L}_{\cris,s}).
  \end{align*}
  We claim that the second of these maps is (up to sign) the natural map obtained by the action of $U_G$ on $\bm{L}_{\cris,s}$, and that the first is obtained from the second via the self-duality of $\bm{L}_{\cris,s}$. The proposition will follow immediately from this claim.

  Consider instead the Kodaira-Spencer map on $\bm{H}_{R_G}$: $F^1\bm{H}_{R_G}\to\gr^{1}_F\bm{H}_{R_G}\otimes\fdiff{R_G/W}$, which we can view this as a map:
  \begin{align}\label{spin:eqn:ks4}
   \Lie U_G\otimes R_G\to\Hom(F^1\bm{H}_{R_G},\gr^1_F\bm{H}_{R_G}).
  \end{align}
  Tensoring (\ref{spin:eqn:ks4}) along $j$ gives us:
  \[
   \Lie U_G\to\Hom(F^1\bm{H}_{\cris,s},\gr^1_F\bm{H}_{\cris,s}).
   \]
  To prove our claim, it is enough to show that, up to sign, this map is identified with the canonical inclusion $\Lie U_G\into\Hom(F^1\bm{H}_{\cris,s},\gr^1_F\bm{H}_{\cris,s})$ induced by the action of $U_G$ on $\bm{H}_{\cris,s}$. This last assertion can be extracted from \cite{moonen}*{4.5}; cf. also~\cite{mp:thesis}*{1.4.2.2(3)}.
\end{proof}

\subsection{}
Over $\Ss_K$ we now have two tautological line bundles: First, we have the Hodge or canonical bundle $\omega^\KS$ attached to the top exterior power of the sheaf of differentials $\Omega^1_{A^\KS/\Ss_K}$. Second, we have the line bundle $F^1\bm{L}_{\dR}$. These are closely related, as the next result shows.

We will need a little preparation. Fix an isotropic line $F^1L\subset L$, and let $P\subset G$ be the parabolic sub-group stabilizing it. Let $F^1H=\ker(F^1L)\subset H$ be the corresponding isotropic sub-space of $H$. The $G$-torsor $\mathcal{P}_{\dR}$ introduced in (\ref{spin:cor:derhamrealization})(\ref{derham:torsor}) has a natural reduction of structure group to a $P$-torsor $\mathcal{P}_{\dR,P}$. Indeed, we can take $\mathcal{P}_{\dR,P}$ to be the sub-functor of $\mathcal{P}_{\dR}$ such that, for and $\Ss_K$-scheme $T$, we have:
\[
 \mathcal{P}_{\dR,P}(T)=\{\xi\in\mathcal{P}_{\dR}(T):\;\xi(F^1H\otimes\Reg{T})=F^1\bm{H}_{\dR,T}\}.
\]
The proof of \emph{loc. cit.} shows that this is indeed a $P$-torsor. Given such a torsor, one immediately gets a functor from $\Int_{(p)}$-representations of the group scheme $P$ to vector bundles over $\Ss_K$. More precisely, given a $\Int_{(p)}$-representation $U$ of $P$, we can view it as a trivial vector bundle over $\Spec\Int_{(p)}$ with a $P$-action, and then take the corresponding $\Ss_K$-vector bundle to be the quotient $(\mathcal{P}_{\dR,P}\times_{\Spec\Int_{(p)}} U)/P$, where $P$ acts diagonally.

\begin{prp}\label{spin:prp:hodgebundle}
There exists a canonical isomorphism of line bundles:
\[
 \omega^{\KS,\otimes 2}\xrightarrow{\simeq}\bigl((F^1\bm{L}_{\dR})(-1)\bigr)^{\otimes 2^{n+1}}.
\]
Here, the $(-1)$ denotes the twist by the (trivial) line bundle attached to the spinor character $\nu:P\to\Gm$. In particular, $F^1\bm{L}_{\dR}$ is a relatively ample line bundle for $\Ss_{K}$ over $\Int_{(p)}$.
\end{prp}
\begin{proof}
  This follows from an argument of Maulik~\cite{maulik}*{\S 5}. The main point is that both bundles involved are canonical extensions over the integral canonical model of automorphic line bundles.

$\omega^{\KS}$ is the line bundle attached via $\mathcal{P}_{\dR,P}$ to the $P$-representation $\det(F^1H)$, and $F^1\bm{L}_{\dR}(-1)$ is the line bundle attached to the representation $F^1L(\nu)$.

The left multiplication map $L\otimes H\to H$ induces an isomorphism of $P$-representations
  \[
   \gr^{-1}_FL\otimes F^1H\xrightarrow{\simeq}\gr^0_FH.
  \]
  Therefore, we have a canonical isomorphism of $P$-representations
  \[
   \det(H)\xrightarrow{\simeq}\det(F^1H)\otimes\det(\gr^0_FH)\xrightarrow{\simeq}\det(F^1H)^{\otimes 2}\otimes(\gr^{-1}_FL)^{\otimes 2^{n+1}}.
  \]
  Since $\dual{(\gr^{-1}_FL)}\xrightarrow{\simeq}F^1L$, this gives us a canonical isomorphism of $P$-representations
  \[
   \det(F^1H)^{\otimes 2}\xrightarrow{\simeq}(F^1L)^{\otimes 2^{n+1}}\otimes\det(H).
  \]
  On the other hand, the symplectic form $\psi_{\delta}$ on $H$ induces a canonical isomorphism of $P$-representations
  \[
   \gr^F_0H\xrightarrow{\simeq}\dual{(F^1H)}(\nu),
  \]
  This shows that we have:
  \[
   \det(H)\xrightarrow{\simeq}\Int_{(p)}(\nu^{2^{n+1}}),
  \]
  completing the proof of the claimed isomorphism.

  The last statement of the lemma follows, since it is known that the bundle $\widetilde{\omega}^{\KS}$ is relatively ample; cf. for example~\cite{lan}*{7.2.4.1(2)}.
\end{proof}

\subsection{}\label{spin:subsec:canonical_filt}
Fix a perfect field $k$ of characteristic $p$ as usual, and set $W=W(k)$. Fix a point $s\in\Ss_K(k)$, and consider the Frobenius map $\bm{F}:\on{Fr}^*\bm{H}_{\dR,s}\to\bm{H}_{\dR,s}$: Its image is a sub-space $\bm{C}^1_s\subset\bm{H}_{\dR,s}$. We claim that there exists an isotropic line $\bm{N}^1_s\subset\bm{L}_{\dR,s}$ such that, in the notation of~\eqref{cliff:subsec:parabolic}, $\bm{C}^1_s=\ker(\bm{N}^1_s)$.

Indeed, fix a lift $\tilde{s}\in\Ss_K(W)$, and let $\mu$ be a co-character of $\GSpin(\bm{L}_{\dR,\tilde{s}})$ that splits the Hodge filtration on $\bm{H}_{\dR,\tilde{s}}$. We obtain a decomposition into weight spaces for $\mu$:
\begin{align*}
\bm{H}_{\dR,\tilde{s}}&=F^1\bm{H}_{\dR,\tilde{s}}\oplus \bm{H}^0_{\dR,\tilde{s}};\\
\bm{L}_{\dR,\tilde{s}}&=F^1\bm{L}_{\dR,\tilde{s}}\oplus \bm{L}^0_{\dR,\tilde{s}}\oplus\bm{L}^{-1}_{\dR,\tilde{s}}.
\end{align*}

By strong divisibility~\eqref{spin:subsec:strongdiv}, $\bm{C}^1_s$ is simply the image in $\bm{L}_{\dR,s}$ of $\bm{F}_s(\sigma^*\bm{H}^0_{\dR,s})$. Now, $\bm{L}^{-1}_{\dR,\tilde{s}}$ is an isotropic line satisfying
\[
 \ker(\bm{L}^{-1}_{\dR,\tilde{s}})=\bm{H}^0_{\dR,\tilde{s}}.
\]

Moreover, by~\eqref{spin:eqn:strongdiv}, $\bm{F}_s(p\sigma^*\bm{L}^{-1}_{\dR,\tilde{s}})$ is an isotropic line in $\bm{L}_{\dR,\tilde{s}}$. One can now easily check that the image of this line in $\bm{L}_{\dR,s}$ is our desired $\bm{N}^1_s$. 

\begin{lem}\label{spin:lem:conj_hodge}
With the notation as above, let $\bm{\Lambda}\subset\bm{L}_{\dR,\tilde{s}}$ be a direct summand satisfying $\bm{F}_s(\bm{\Lambda})=\bm{\Lambda}$. Let $\overline{\bm{\Lambda}}$ be the image of $\bm{\Lambda}$ in $\bm{L}_{\dR,s}$. Then $\bm{N}^1_s$ is contained in $\overline{\bm{\Lambda}}$ only if $\bm{F}^1\bm{L}_{\dR,s}$ is also contained in $\overline{\bm{\Lambda}}$.
\end{lem}
\begin{proof}
Choose a generator $w$ for $\bm{L}^{-1}_{\dR,\tilde{s}}$; then, by the construction above, the image of $\bm{F}_s(pw)$ in $\bm{L}_{\dR,s}$ is a generator for $\bm{N}^1_s$. Therefore, if $\bm{N}^1_s$ is contained in $\overline{\bm{\Lambda}}$, then we can find $v\in\bm{\Lambda}$ such that $\bm{F}_s(pw)-v\in p\bm{L}_{\dR,\tilde{s}}$. 

But, by our hypothesis, $v=\bm{F}_s(v')$, for $v'\in\bm{\Lambda}$. Then, by~\eqref{spin:eqn:strongdiv}, we must have:
\[
 pw-v'\in F^1\bm{L}_{\dR,\tilde{s}}+p\bm{L}_{\dR,\tilde{s}}.
\]
This implies that $F^1\bm{L}_{\dR,s}$ is generated by the image of $v'\in\bm{\Lambda}$.
\end{proof}

\section{Special endomorphisms}\label{sec:special}

We will now drop the self-duality assumption on $(L,Q)$ until further notice.

\subsection{}
Let $T$ be an $\Sh_{K,\Comp}$-scheme; then functoriality of cohomology gives us a natural map
\[
  \End(A^\KS_T)_{(p)}\rightarrow H^0\bigl(T^{\an},\bm{H}_{B}^{\otimes(1,1)}\bigr).
\]
\begin{defn}\label{special:defn:complex}
An endomorphism $f\in\End(A^\KS_T)_{(p)}$ is \defnword{special} if it gives rise to a section of $\bm{L}_B\subset\bm{H}_B^{\otimes(1,1)}$ under the above map. It follows from the definition that $f$ is special if and only if its fiber $f_s$ at every point $s\to T$ is special. In fact, it is enough to require this for one point $s$ in each connected component of $T^{\an}$.
\end{defn}

Let $T$ be any $\Sh_K$-scheme; then, for any prime $\ell$, we have a natural map
\[
  \End(A^\KS_T)_{(p)}\rightarrow H^0\bigl(T,\bm{H}_{\ell}^{\otimes(1,1)}\bigr).
\]

\begin{defn}\label{special:defn:ell-special}
Fix a prime $\ell\neq p$. An endomorphism $f\in\End(A^\KS_T)_{(p)}$ is \defnword{$\ell$-special} if it gives rise to a section of $\bm{L}_{\ell}\subset\bm{H}_{\ell}^{\otimes(1,1)}$ under the above map. We say that $f$ is \defnword{$p$-special} if it gives rise to a section of $\bm{L}_p\subset\bm{H}_p^{\otimes(1,1)}$ under the corresponding $p$-adic realization map.

For any prime $\ell$, we denote the space of $\ell$-special endomorphisms by $L_{\ell}(A^\KS_T)$.
\end{defn}

One immediately sees that $f$ is $\ell$-special if and only if, in every connected component of $T$, there exists a point $s$ such that the fiber $f_s$ at $s$ is $\ell$-special. In particular, $\ell$-specialness is a condition that can be checked at geometric points.

\begin{lem}\label{special:lem:complex}
Suppose that $T$ is an $\Sh_K$-scheme and that $f\in\End(A^\KS_T)_{(p)}$. Then the following are equivalent:
\begin{enumerate}
  \item $f$ is $\ell$-special for all primes $\ell$.
  \item $f$ is $\ell$-special for some prime $\ell$.
  \item The restriction of $f$ over $T\otimes_{\Rat}\Comp$ is special.
\end{enumerate}
\end{lem}
\begin{proof}
  Let $s\to T$ be any $\Comp$-valued point. Then it is clear from the definitions that $f_s$ is special if and only if it is $\ell$-special for some (hence any) prime $\ell$. Using this, the lemma easily follows.
\end{proof}

\begin{defn}\label{special:defn:char0}
  Let $T$ and $f$ be as above. Then we say that $f$ is \defnword{special} if it satisfies any of the equivalent conditions of (\ref{special:lem:complex}). We denote the space of special endomorphisms of $A^\KS_T$ by $L(A^\KS_T)$.
\end{defn}

Over $\Comp$, $L(A^\KS_s)$ can be described purely Hodge theoretically:
\begin{prp}\label{special:prp:complexpoints}
If $s\in \Sh_K(\Comp)$, then
  \[
    L(A^\KS_s)=\bm{L}_{B,s}\cap(\bm{L}_{B,s}\otimes\Comp)^{0,0}.
  \]
In particular, for any $\Sh_K$-scheme $T$, $\rk_{\Int_{(p)}}L(A^\KS_T)\leq n$.
\end{prp}
\begin{proof}
  This is clear from the definitions.
\end{proof}

\subsection{}
We now assume that $(L,Q)$ is self-dual, and turn to the investigation of specialness over the integral canonical model $\Ss_K$. For any $\Ss_K$-scheme $T$ and any $\ell\neq p$, the definition of an $\ell$-special endomorphism carries over directly from (\ref{special:defn:ell-special}). We will now develop a version of $p$-specialness that works also over the integral model.
\begin{defn}\label{lifts:defn:p-special-point}
Suppose that $s\to \Ss_{K,\Field_p}$ is a $k$-valued point. Then we obtain a map
\[
  \End(A^\KS_s)_{(p)}\rightarrow \End\bigl(\bm{H}_{\cris,s}\bigr).
\]
An endomorphism $f\in\End(A^\KS_s)_{(p)}$ is \defnword{$p$-special} if it gives rise to an element of $\bm{L}_{\cris,s}$ under the above map.
\end{defn}

\begin{lem}\label{lifts:lem:p-special}
Let $T$ be a $\Ss_K$-scheme in which $p$ is locally nilpotent, and suppose that we have $f\in\End(A^\KS_T)_{(p)}$. Then the following are equivalent:
\begin{enumerate}
  \item For every point $s\to T$ valued in a perfect field, the fiber $f_s\in\End(A^\KS_s)_{(p)}$ is $p$-special.
  \item In every connected component of $T$, there exists a point $s$ valued in a perfect field such that the fiber $f_s$ is $p$-special.
\end{enumerate}
\end{lem}
\begin{proof}
This is an immediate consequence of the definition, the fact that the endomorphism scheme $\underline{\End}(A^\KS)_{(p)}$ of $A^\KS$ is locally Noetherian, and (\ref{lifts:lem:crystalline_sections}) below, applied to the crystal $\bm{H}^{\otimes(1,1)}_{\cris}/\bm{L}_{\cris}$.
\end{proof}

\begin{lem}\label{lifts:lem:crystalline_sections}
Suppose that $T$ is a connected, locally Noetherian $\Field_p$-scheme, and that $\bm{M}$ is a crystal of vector bundles over $T$ equipped with a global section $e\in \Gamma\bigl((T/\Int_p)_{\cris},\bm{M}\bigr)$. Suppose that, for some point $x\rightarrow T$, the induced global section
\[
e_x\in \Gamma \bigl((\Spec k(x)/\Int_p)_{\cris},\bm{M}\vert_x\bigr)
\]
vanishes. Then, for every point $y\rightarrow T$, the induced global section
\[
e_y\in \Gamma \bigl((\Spec k(y)/\Int_p)_{\cris},\bm{M}\vert_y\bigr)
\]
also vanishes.
\end{lem}
\begin{proof}
  Since $T$ is connected, locally Noetherian, the lemma will follow if we can prove the following claim.
  \begin{claim}
    Suppose that $x$ and $y$ are points of $T$ such that $x$ is a specialization of $y$, and such that the prime ideal corresponding to $y$ has height $1$ in $\Reg{T,x}$; then $e_x$ vanishes if and only if $e_y$ vanishes.
  \end{claim}
  Let $x^{\mathrm{perf}}$ be the point attached to a perfect closure of $k(x)$. By \cite{berthelot_messing}*{1.3.5}, $e_x$ vanishes if and only if $e_{x^{\mathrm{perf}}}$ vanishes. So we can assume that $k=k(x)$ is perfect.

  Let $\mathfrak{p}_y\subset\Reg{T,x}$ be the prime ideal corresponding to $y$, and let $R$ be the completion of the normalization of $\Reg{T,x}/\mathfrak{p}_y$. By our hypotheses, $R$ is an equicharacteristic complete DVR with residue field $k$, and so is isomorphic to $k\pow{t}$. By pulling $\bm{M}$ back to $\Spec R$, we are further reduced to the situation where $T=\Spec k\pow{t}$, $x=\Spec k$ and $y=\Spec k((t))$.

  Let $M$ be the evaluation of $\bm{M}$ along the formal divided power thickening $\Spec k\pow{t}\into\Spf W\pow{t}$: It is equipped with a flat, topologically quasi-nilpotent connection $\nabla:M\to M\otimes\widehat{\Omega}^1_{W(k)\pow{t}}$. In other words, we have a derivation $D:M\to M$ over $\frac{d}{dt}$ such that a sufficiently large iteration of $D$ carries $M$ into $pM$. The crystal $\bm{M}$ is determined by the finite free $W(k)\pow{t}$-module $M$ and   the connection $\nabla$. In particular, the global sections of $\bm{M}$ are identified with the module of horizontal elements $M^{\nabla=0}$.

  The corresponding crystal over $x$ is the one attached to the $W(k)$-module $M_0=M/tM$. The restriction from global sections over $T$ to global sections over $x$ is just the reduction-mod-$t$ map $M^{\nabla=0}\to M_0$. It is easy to see that this map is injective: Indeed, suppose that we have $m\in tM$ such that $D(m)=0$, and suppose that $n\in\Int_{>0}$ is the largest integer such that $m\in t^nM$ (such an $n$ exists if and only if $m$ is non-zero). Write $m=t^nm'$, for some $m'\notin tM$. We then have:
  $$0=D(m)=D(t^nm')=nt^{n-1}m'+t^nD(m').$$
  Dividing by $t^{n-1}$, this gives us $nm'\in tM$, which implies that $m'\in tM$, contradicting our assumption that $m$ is non-zero. So we find that a global section of a crystal over $T$ is $0$ precisely if it restricts to $0$ over $x$.

  By \cite{berthelot_messing}*{1.3.5} again, restriction from global sections over $T$ to global sections over $y$ is injective. Therefore, a global section of a crystal over $T$ is $0$ precisely when it restricts to $0$ over $y$. This proves the claim and the lemma.
\end{proof}

\begin{defn}\label{lifts:defn:p-special}
Let $T$ be an $\Ss_K$-scheme and let $f\in\End(A^\KS_T)_{(p)}$. If $p\Reg{T}=0$, we will say that $f$ is \defnword{$p$-special} if it satisfies the equivalent conditions of (\ref{lifts:lem:p-special}). In general, we will say that $f$ is $p$-special, if its restriction to $T\otimes\Field_p$ and $T\otimes\Rat$ are both $p$-special.

We will say that $f$ is \defnword{special} if it is $\ell$-special for \emph{every} prime $\ell$.

Given any prime $\ell$, write $L_{\ell}(A^\KS_T)$ for the space of $\ell$-special endomorphisms, and $L(A^\KS_T)$ for the space of special endomorphisms. By definition, we have:
\[
 L(A^\KS_T)=\bigcap_{\ell\text{ prime}}L_{\ell}(A^\KS_T).
\]
\end{defn}

\begin{lem}\label{special:lem:rosati}
The space $L(A^\KS_T)\subset\End(A^\KS_T)_{(p)}$ is point-wise fixed by the Rosati involution induced from the polarization $\lambda^\KS$. In particular, $f\mapsto f\circ f$ defines a positive definite $\Int_{(p)}$-quadratic form on $L(A^\KS_T)$.
\end{lem}
\begin{proof}
 Consider the $G$-equivariant symplectic pairing $\psi_{\delta}:H_{\Rat}\otimes H_{\Rat}\to\Rat(\nu)$: By construction, for any prime $\ell\neq p$, the corresponding pairing of $\ell$-adic sheaves $\bm{H}_{\ell}\otimes\bm{H}_{\ell}\to\underline{\Rat}_{\ell}(-1)$ is, up to scalars, identified with the polarization pairing induced from $\lambda^\KS$.

 $\psi_{\delta}$ gives rise to an isomorphism of $G$-representations $\lambda(\psi_{\delta}):H_{\Rat}\xrightarrow{\simeq}\dual{H}_{\Rat}(\nu)$.
 If we identify $H^{\otimes(1,1)}$ with $\End(H)$, we now obtain a $G$-equivariant involution:
 \begin{align*}
  i_{\delta}:H_{\Rat}^{\otimes(1,1)}&\to H_{\Rat}^{\otimes(1,1)};\\
  f&\mapsto\lambda(\psi_{\delta})^{-1}\circ\dual{f}(\nu)\circ\lambda(\psi_{\delta}).
 \end{align*}
 Further, for every $v\in L_{\Rat}$, $v^*=v\in C_{\Rat}$. Therefore, for $z_1,z_2\in H_{\Rat}$, we have:
 \[
  \psi_{\delta}(vz_1,z_2)=\trd(vz_1\delta z_2^*)=\trd(z_1\delta z_2^*v)=\trd(z_1\delta(vz_2)^*)=\psi_{\delta}(z_1,vz_2).
 \]
 This shows that $i_{\delta}(v)=v$, for all $v\in L_{\Rat}$.

 For any prime $\ell\neq p$, the $G$-equivariant map $i_{\delta}$ induces a map of sheaves over $\Sh_K$ (and hence $\Ss_K$): $\bm{i}_{\delta}:\bm{H}_{\ell}^{\otimes(1,1)}\to\bm{H}_{\ell}^{\otimes(1,1)}$. By construction, $\bm{i}_{\delta}\rvert_{\bm{L}_{\ell}}\equiv 1$.

 Now, the restriction of $\bm{i}_{\delta}$ to the $\ell$-adic realization of any endomorphism of $A^\KS$ is exactly the Rosati involution arising from $\lambda^\KS$. This gives us the first assertion. The second follows from the positivity of the Rosati involution.
\end{proof}

We have:
\begin{lem}\label{lifts:lem:p-special-consistent}
Let $T$ be an $\Ss_K$-scheme such that every generic point of $T\otimes\Field_p$ is the specialization of a point in $T\otimes\Rat$. Then an endomorphism $f\in\End(A^\KS_T)_{(p)}$ is $p$-special over $T\otimes\Rat$ if and only if it is $p$-special over $T\otimes\Field_p$. In particular, in this situation, if $f$ is $p$-special, then it is in fact special.
\end{lem}
\begin{proof}
  First assume that $T=\Spec\Reg{E}$, for some complete discrete valuation ring $\Reg{E}$ with characteristic $0$ fraction field $E$ and characteristic $p$ perfect residue field $k$. Let $s=\Spec k$, and let $\overline{\eta}=\Spec\overline{E}$, for some algebraic closure $\overline{E}/E$. Then the result holds because the $p$-adic comparison isomorphism for $A^\KS_T$ carries $\bm{L}_{p,\overline{\eta}}\otimes\Bcris$ into $\bm{L}_{\cris,s}\otimes\Bcris$ (cf.~(\ref{spin:prp:kisinpadichodge})(\ref{intphodge:comp})).

  For general $T$, in every connected component of $T\otimes\Field_p$, we can find a point $s$ valued in an algebraically closed field that is the specialization of a point valued in a complete discrete valuation field of mixed characteristic $(0,p)$. Now we can apply the result of the previous paragraph to $s$.

  The last assertion follows because $\ell$-specialness is independent of the prime $\ell$ in characteristic $0$.
\end{proof}

\subsection{}
Fix a point $x_0\in\Ss_K(k)$, where, as always, $k$ is a perfect field of characteristic $p$. Set $W=W(k)$, and let $R_{x_0}$ be the complete local ring of ${\Ss_{K,W}}$ at $x_0$; let $\mx\subset R_{x_0}$ be its maximal ideal. Also, let $\widehat{U}=\Spf R_{x_0}$ be the corresponding formal scheme over $W$.

Let $\Lambda$ be a quadratic space over $\Int_{(p)}$; we assume that $\Lambda$ is finite free over $\Int_{(p)}$. Suppose that we have an isometric map $\iota_0:\Lambda_0\to L(A^{\KS}_{x_0})$. We will consider the deformation functor $\Def_{(x_0,\iota_0)}$, defined as follows: For any $B\in\Art_W$, we have
\[
  \Def_{(x_0,\iota_0)}(B)=\bigl\{(x,\iota):\; x\in\widehat{U}(B);\;\iota:\Lambda_0\to L(A^{\KS}_x)\text{ isometric map lifting $\iota_0$}\bigr\}.
\]
Here, $\Art_W$ is the category of local artinian $W$-algebras with residue field $k$.

If $\Lambda=\{v\}$ consists of a single element, and if $\iota_0(v)=f_0$, we will write $\Def_{(x_0,f_0)}$ for the corresponding deformation functor.

The functor $\Def_{(x_0,\iota_0)}$ is represented by a closed formal sub-scheme $\widehat{U}_{\iota_0}\subset\widehat{U}$. This can be seen from the fact that the endomorphism scheme of an abelian scheme is representable and unramified over the base. Again, if $\Lambda=\{v\}$ is a singleton with $\iota_0(v)=f_0$, we will write $\widehat{U}_{f_0}$ for the corresponding formal sub-scheme.

\subsection{}
Suppose that we have a surjection $\Rg\rightarrow \overline{\Rg}$ in $\Art_W$, whose kernel $I$ admits nilpotent divided powers. Suppose also that we have $(\overline{x},\overline{\iota})\in\widehat{U}_{\iota_0}(\overline{\Rg})$ giving rise to an abelian scheme $A^{\KS}_{\overline{x}}$ over $\overline{\Rg}$ equipped with an isometric map $\overline{\iota}:\Lambda\to L(A^{\KS}_{\overline{x}})$ lifting $\iota_0$.

Let $\bm{H}_{\Rg}$ be the $\Rg$-module obtained by evaluating $\bm{H}_{\cris,\Spec\overline{\Rg}}$ on $\Spec\Rg$, and let $\bm{L}_{\Rg}\subset\End(\bm{H}_{\Rg})$ be the corresponding quadratic space. Denote by $\bm{H}_{\overline{\Rg}}$ and $\bm{L}_{\overline{\Rg}}$ the induced modules over $\overline{\Rg}$; then $\bm{H}_{\overline{\Rg}}=\bm{H}_{\dR,\overline{x}}$ is equipped with its Hodge filtration $F^1\bm{H}_{\overline{\Rg}}$.

The crystalline realization $L(A^{\KS}_{\overline{x}})\to \bm{L}_{\Rg}$ composed with $\overline{\iota}$ gives us a map $\iota_{\Rg}:\Lambda\to\bm{L}_{\Rg}$. The image of $\iota_{\Rg}(\Lambda)$ in $\bm{L}_{\overline{\Rg}}$ preserves the Hodge filtration $F^1\bm{H}_{\overline{\Rg}}$ and so lands in $F^0\bm{L}_{\overline{\Rg}}$.

By Serre-Tate (cf.~\cite{katz:serre-tate}*{1.2.1}) and Grothendieck-Messing (cf.~\cite{messing}*{V.1.6}), we have a natural bijection
\begin{align*}
  \begin{pmatrix}
    \text{Isomorphism classes of }\\\text{abelian schemes over $\Rg$ lifting $A^\KS_{\overline{x}}$}
  \end{pmatrix}&\xrightarrow{\simeq}
  \begin{pmatrix}
    \text{Direct summands $F^1\bm{H}_{\Rg}\subset \bm{H}_{\Rg}$}\\\text{ lifting $F^1\bm{H}_{\overline{\Rg}}$}.
  \end{pmatrix}
\end{align*}
This works as follows: For any lift $A_x$ of $A^\KS_{\overline{x}}$, we have a canonical identification $H^1_{\dR}(A_x/\Rg)=\bm{H}_{\Rg}$ that carries the Hodge filtration on $H^1_{\dR}(A_x/\Rg)$ to the corresponding summand $F^1_x\bm{H}_{\Rg}\subset \bm{H}_{\Rg}$.

\begin{prp}\label{lifts:prp:nillifts}
The bijection above induces further bijections
\begin{align*}
  \begin{pmatrix}    \texte{Lifts $x\in\widehat{U}(\Rg)$ of $\overline{x}$}  \end{pmatrix}&
  \xrightarrow{\simeq}\begin{pmatrix} \texte{Isotropic lines $F^1\bm{L}_{\Rg}\subset\bm{L}_{\Rg}$ lifting $F^1\bm{L}_{\overline{\Rg}}$}  \end{pmatrix};
\end{align*}
\begin{align*}
  \begin{pmatrix}    \texte{Lifts $(x,\iota)\in\widehat{U}_{\iota_0}(\Rg)$ of $(\overline{x},\overline{\iota})$}  \end{pmatrix}&\xrightarrow{\simeq}\begin{pmatrix}
    \texte{Isotropic lines $F^1\bm{L}_{\Rg}\subset\bm{L}_{\Rg}$ lifting $F^1\bm{L}_{\overline{\Rg}}$}\\
    \texte{and orthogonal to the sub-space $\iota_{\Rg}(\Lambda)$}\end{pmatrix}.\\
\end{align*}
In particular, both the sets in the latter bijection are empty, unless $F^1\bm{L}_{\overline{\Rg}}$ lies in the image in $\bm{L}_{\overline{\Rg}}$ of the sub-module $\iota_{\Rg}(\Lambda)^{\perp}\subset\bm{L}_{\Rg}$.
\end{prp}
\begin{proof}
 In the first of the claimed bijections, there is a natural map in one direction: Given a lift $x\in\widehat{U}(\Rg)$ and the identification $\bm{L}_{\dR,x}=\bm{L}_{\Rg}$, the Hodge filtration $F^1\bm{L}_{\dR,x}$ gives us an isotropic line $F^1\bm{L}_{\Rg}$ lifting $F^1\bm{L}_{\overline{\Rg}}$. Further, an endomorphism $f\in L(A^{\KS}_{\overline{x}})$ lifts to an endomorphism of $A^\KS_x$ if and only if its crystalline realization $f_{\Rg}\in\bm{L}_{\Rg}$ preserves the Hodge filtration $F^1\bm{H}_{\Rg}$. Since $F^1\bm{H}_{\Rg}$ is the kernel of any generator of $F^1\bm{L}_{\Rg}$, it is easy to see that $f_{\Rg}$ preserves $F^1\bm{H}_{\Rg}$ if and only if it is orthogonal to $F^1\bm{L}_{\Rg}$.

 So it is enough to show that the first map is a bijection. For this, we can work successively with the thickenings $\Rg/I^{[r-1]}\twoheadrightarrow\Rg/I^{[r]}$ (where $I^{[r]}$ denotes the $r^{\on{th}}$-divided power of $I$), and assume that $I^2=0$. If $\mx_{\Rg}\subset\Rg$ is the maximal ideal, we can even work with successively with the thickenings $\Rg/\mx^{r-1}_{\Rg}I\twoheadrightarrow\Rg/\mx_{\Rg}^rI$, and further assume that $\mx_{\Rg} I=0$. In this case, we find that both sides of the map in question are vector spaces over $k$ of the same dimension, namely $n\cdot\dim_kI$, and that the map is a map of $k$-vector spaces. Since it is clearly injective, we see that the map must in fact be a bijection.
 
 As for the final assertion, we have only included it to highlight the fact that, when $\iota_{\Rg}(\Lambda)$ is not a direct summand of $\bm{L}_{\Rg}$, the formation of its orthogonal complement is not well-behaved with respect to arbitrary base-change. So, even though $F^1\bm{L}_{\overline{\Rg}}$ is orthogonal to $\iota_{\overline{\Rg}}(\Lambda)$, it need not be in the image of the orthogonal complement of $\iota_{\Rg}(\Lambda)$. 
\end{proof}

\begin{corollary}\label{lifts:cor:locus}
Let the notation be as above. Suppose that we have a lift $x\in\widehat{U}(\Rg)$ of $\overline{x}$ corresponding to an isotropic line $F^1\bm{L}_{\Rg}\subset\bm{L}_{\Rg}$. Let $J\subset \Rg$ be the smallest ideal such that $\overline{\iota}$ lifts to an isometric map $\Lambda\to L\bigl(A^\KS_x\otimes_{\Rg}(\Rg/J)\bigr)$. Then $J$ is generated by the elements $[\overline{\iota}(v),w]$, where $v$ varies over the elements of $\Lambda$, and $w$ is any basis element of $F^1\bm{L}_{\Rg}$. In particular, $\widehat{U}_{\iota_0}\subset\widehat{U}$ is cut out by $r=\rank\iota_0(\Lambda)$ equations.
\end{corollary}
\begin{proof}
  The assertion about $J$ is immediate.

  Set $R=\widehat{\Rg}_{{\Ss_K},x_0}$. Let $I\subset R$ be the ideal defining $\widehat{U}_{\iota_0}$. Applying the first assertion with $\Rg=R/\mx I$ and $\overline{\Rg}=k$ shows that $I/\mx I$ is generated by $r$ elements. Now the last assertion follows from Nakayama's lemma.
\end{proof}

\subsection{}
Assume that $\iota_0$ maps $\Lambda$ injectively onto a direct summand of $L(A^\KS_{s_0})$. Let $\iota_{\dR}:\Lambda\otimes R_{\iota_0}\to \bm{L}_{\dR,R_{\iota_0}}$ be the de Rham realization of the universal isometry $\iota:\Lambda\to L\bigl(A^\KS\vert_{\widehat{U}_{\iota_0}}\bigr)$: It factors through $F^0\bm{L}_{\dR,R_{\iota_0}}$. Write $\bm{\Lambda}_{\dR}$ for the image of this map, and let $\overline{\bm{\Lambda}}_{\dR}$ be the image of $\bm{\Lambda}_{\dR}$ in $\gr^0_F\bm{L}_{\dR,R_{\iota_0}}$. 

The Kodaira-Spencer map over $R$ restricts to a map
\begin{align*}
 \gr^0_F\bm{L}_{\dR,R_{\iota_0}}&\to\gr^{-1}_F\bm{L}_{\dR,R_{\iota_0}}\otimes\widehat{\Omega}^1_{R_{\iota_0}/W}.
\end{align*}
Since $\bm{\Lambda}_{\dR}$ is generated by sections that are parallel for the connection, this map factors as:
\begin{align}\label{lifts:eqn:ksiota0}
\frac{\gr^0_F\bm{L}_{\dR,R_{\iota_0}}}{\overline{\bm{\Lambda}}_{\dR}}&\to\gr^{-1}_F\bm{L}_{\dR,R_{\iota_0}}\otimes\widehat{\Omega}^1_{R_{\iota_0}/W}.
 \end{align}

As a direct consequence of~\eqref{lifts:prp:nillifts} (use $\Rg=k[\epsilon]$ and $\overline{\Rg}=k$), we have:
\begin{corollary}\label{lifts:cor:smooth}
The formal scheme $\widehat{U}_{\iota_0}$ is formally smooth over $W$ if and only if $\bm{\Lambda}_{\dR}$ maps isomorphically onto $\overline{\bm{\Lambda}}_{\dR}$. In this case, the Kodaira-Spencer map in~\eqref{lifts:eqn:ksiota0} is an isomorphism
\end{corollary}
\qed

\subsection{}
Assume that $\Lambda=\{v\}$ is a singleton with $\iota_0(v)=f_0\neq 0$. For simplicity, set $R=R_{x_0}$, and let $I_{f_0}\subset R$ be the ideal defining $\widehat{U}_{f_0}\subset\widehat{U}$, and set $R_{f_0}=R/I_{f_0}$, so that $\widehat{U}_{f_0}=\Spf R_{f_0}$. It follows from (\ref{lifts:cor:locus}) that $I_{f_0}$ is principal, generated by a single element $a_{f_0}$. Let $f_{0,\dR}\in\bm{L}_{\dR,x_0}$ be the de Rham realization of $f_{0}$. Most of the following result is essentially a retread of \cite{deligne:k3liftings}*{Prop. 1.5}.
\begin{prp}\label{lifts:prp:principal}
\mbox{}
\begin{enumerate}
\item\label{principal:flat}$\widehat{U}_{f_0}$ is flat over $\Int_p$; that is, $p\nmid a_{f_0}$.
\item\label{principal:regular}If $f_{0,\dR}\neq 0$ and $\nu_p(f_0\circ f_0)\neq 1$, then $\widehat{U}_{f_0}$ is formally smooth.
\end{enumerate}
\end{prp}
\begin{proof}
For the first statement, we use the argument from the proof of \cite{deligne:k3liftings}*{1.6}. As in \emph{loc. cit.}, we reduce immediately to the following assertion: $f_0$ does not propagate to a special endomorphism of $\widetilde{A}^{\KS}_{R\otimes\Field_p}$. Suppose that such a propagation did exist; then we can consider its crystalline realization $f_R\in\bm{L}_{\dR,R}$. Choose $k\in\Int_{\geq 0}$ minimal with respect to the condition that $p^{-k}f_R$ belongs to $\bm{L}_{\dR,R}$.

Fix a Frobenius lift $\varphi:R\to R$; this, combined with the $F$-crystal structure on $\bm{H}_{\dR,R}$, endows $\bm{L}_{\dR,R}[p^{-1}]$ with a $\varphi$-semi-linear Frobenius endomorphism $F$. Now, $F(f_R)=f_R$, which implies that $F(p^{-k}f_R)=p^{-k}f_R$. By strong divisibility~(\ref{spin:prp:moonen})(\ref{moonen:strongdiv}), $p^{-k}f_R$ lies in $F^0\bm{L}_{\dR,R}+p\bm{L}_{\dR,R}$. In particular, the image $\bar{f}_R$ of $p^{-k}f_R$ in $\bm{L}_{\dR,R\otimes\Field_p}$ is a non-zero horizontal element that lies in $F^0\bm{L}_{\dR,R\otimes\Field_p}$.

But (\ref{spin:prp:kodairaspencer}) shows that the connection on $\bm{L}_{\dR,R\otimes\Field_p}$ induces an $R$-linear Kodaira-Spencer isomorphism
\[
 \gr^0_F\bm{L}_{\dR,R\otimes\Field_p}\xrightarrow{\simeq}\gr^{-1}_F\bm{L}_{\dR,R\otimes\Field_p}\otimes\widehat{\Omega}^1_{R\otimes\Field_p/k}.
\]
This shows that $\bar{f}_R$ must actually lie in $F^1\bm{L}_{\dR,R\otimes\Field_p}$. But once again the connection on $\bm{L}_{\dR,R\otimes\Field_p}$ sets up an $R$-linear embedding:
\[
 F^1\bm{L}_{\dR,R\otimes\Field_p}\into\gr^0_F\bm{L}_{\dR,R\otimes\Field_p}\otimes\widehat{\Omega}^1_{R\otimes\Field_p/k}.
\]
This shows that $\bar{f}_R=0$, which is a contradiction.

We move on to (\ref{principal:regular}): Since it is defined by a single equation within the formally smooth formal scheme $\widehat{U}$, $\widehat{U}_{f_0}$ is formally smooth precisely when the map on tangent spaces
\[
 \widehat{U}_{f_0}(k[\epsilon])\to\widehat{U}(k[\epsilon])
\]
is not bijective.

Therefore, by applying (\ref{lifts:prp:nillifts}) with $\Rg=k[\epsilon]$, $\overline{\Rg}=k$ and $\Lambda=\{v\}$, with $\iota_0(v)=f_0$, we see that $\widehat{U}_{f_0}$ fails to be formally smooth precisely when every isotropic lift $F^1(\bm{L}_{\dR,x_0}\otimes k[\epsilon])$ of $F^1\bm{L}_{\dR,x_0}$ is orthogonal to $f_{0,\dR}$. It is easy to check that this can happen if and only if $F^1\bm{L}_{\dR,x_0}$ contains $f_{0,\dR}$.

Suppose therefore that $f_{0,\dR}$ is contained in $F^1\bm{L}_{\dR,x_0}$. Choose any isotropic line $F^1\bm{L}_{\cris,x_0}\subset \bm{L}_{\cris,x_0}$ lifting $F^1\bm{L}_{\dR,x_0}$. If $f_{0,\cris}\in\bm{L}_{\cris,x_0}$ is the crystalline realization of $f_0$, then $F^1\bm{L}_{\cris,x_0}$ is generated by an element $w$ of the form $f_{0,\cris}+pv$, for some $v\in\bm{L}_{\cris,x_0}$. We have:
\begin{align}\label{principal:eqn1}
w-pv=f_{0,\cris}=F(f_{0,\cris})=F(w)-pF(v)
\end{align}

By strong admissibility of $\bm{L}_{\cris,x_0}$, we have $F(w)\in p\bm{L}_{\cris,x_0}$. This shows:
\begin{align}\label{principal:eqn2}
F(v)=\frac{1}{p}F(w)+v-\frac{1}{p}w\in p^{-1}\bm{L}_{\cris,x_0}\backslash\bm{L}_{\cris,x_0}.
\end{align}
Applying strong admissibility once again, we conclude that $v\notin F^0\bm{L}_{\cris,x_0}+p\bm{L}_{\cris,x_0}$; here, $F^0\bm{L}_{\cris,x_0}=(F^1\bm{L}_{\cris,x_0})^{\perp}$. In particular, $[f_{0,\cris},v]$ belongs to $W^\times$.

We now have:
\begin{align}\label{principal:eqn3}
 0=[f_{0,\cris}+pv,f_{0,\cris}+pv]=[f_{0,\cris},f_{0,\cris}]+2p[f_{0,\cris},v]+p^2[v,v].
\end{align}
Since $[f_{0,\cris},v]$ is a unit, this implies that $\nu_p(f_0\circ f_0)=1$.
\end{proof}

\begin{corollary}\label{lifts:cor:pimpliesell}
For every $\Ss_K$-scheme $T$ and every prime $\ell$, we have:
\[
 L_p(A^\KS_T)\into L_{\ell}(A^\KS_T).
\]
In particular, $L(A^\KS_T)=L_p(A^\KS_T)$.
\end{corollary}
\begin{proof}
  It follows from the definitions that it suffices to prove the corollary when $T$ is a point $x_0:\Spec k\to\Ss_K$. If $f_0\in L_p(A^\KS_{x_0})$, then by (\ref{lifts:prp:principal}), $\widehat{U}_{f_0}$ is flat. This implies that there exists a finite extension $E/W_{\Rat}$, and a lift $x:\Spec\Reg{E}\to\Ss_K$ of $x_0$ such that $f_0$ lifts to a special endomorphism $f$ of $A^\KS_x$. Now (\ref{lifts:lem:p-special-consistent}) shows that $f$, and hence $f_0$, is $\ell$-special for every $\ell$.
\end{proof}

\begin{rem}\label{lifts:rem:p-specialequalsspecial}
From now on, we will refer to $p$-special endomorphisms simply as special endomorphisms.
\end{rem}

\subsection{}\label{lifts:subsec:ortho}
The special endomorphisms are, in a precise sense, objects defined over the integral model $\Ss_{K_0}$ of the orthogonal Shimura variety $\Sh_{K_0}$. Let $\Delta(K)$ be as in (\ref{spin:subsec:shimura}): it is the Galois group of the finite \'etale cover $\Ss_K\to\Ss_{K_0}$. Given $[z]\in\Delta(K)$ attached to an element $z\in\Adele_f^{p,\times}$, its action on $\Ss_{K}$ carries $(A^{\KS},[\eta^{\KS}_G])$ to $(A^{\KS},[\eta^{\KS}_G\circ z])$. Here, $[\eta^{\KS}_G]$ is the canonical $K^p$-level structure on $A^{\KS}$ (cf.~\ref{spin:prp:etalerealization}).

On the other hand, viewing $z$ as a tuple $(z_{\ell})_{\ell\neq p}$, we can construct the element $n(z)=\prod_{\ell\neq p}\ell^{-\nu_{\ell}(z)}\in\Int_{(p)}^\times$. This acts on $A^{\KS}$ by multiplication and induces an isomorphism of pairs
  \[
   (A^{\KS},[\eta_G^{\KS}\circ z])\xrightarrow{\simeq}(A^{\KS},[\eta^{\KS}_G]).
  \]
Conjugation by $n(z)$ now produces a canonical identification between the endomorphism sheaves $\underline{\End}\bigl([z]^*A^{\KS}\bigr)_{(p)}$ and $\underline{\End}\bigl(A^{\KS}\bigr)_{(p)}$. This gives us a canonical descent datum that allows us to descend the sheaf $\underline{\End}\bigl(A^\KS\bigr)_{(p)}$ to a sheaf $\bm{E}$ of $\Int_{(p)}$-algebras over $\Ss_{K_0}$.

It follows now that, given a map $T\to\Ss_{K_0}$, we can always attach to it the `endomorphism' algebra $\bm{E}(T)$. Moreover, $\bm{E}$ comes equipped with canonical realization functors into $\bm{H}^{\otimes(1,1)}_{p}$ over the generic fiber, and into $\bm{H}^{\otimes(1,1)}_{\cris}$ over the special fiber. In particular, we can speak of the space of `special endomorphisms' $L(T)\subset\bm{E}(T)$, whose realizations at every closed point land in $\bm{L}_?\subset\bm{H}^{\otimes(1,1)}_{?}$, where $?=p$, for a point in the generic fiber, and $?=\cris$, for a point in the special fiber. If $T$ is in fact an $\Ss_{K}$-scheme, then we will have $L(T)=L(A^{\KS}_T)$.

\section{Cycles defined by special endomorphisms}\label{sec:cycles}

In this section $(L,Q)$ will be a quadratic space over $\Int_{(p)}$ of signature $(n,2)$. We will \emph{not} assume it to be self-dual.

\subsection{}\label{special:subsec:embedding}
Suppose that $(\widetilde{L},\widetilde{Q})$ is another $\Int_{(p)}$-quadratic space of signature $(n+d,2)$ equipped with an embedding
\[
 (L,Q)\into (\widetilde{L},\widetilde{Q}),
\]
so that $L$ is a direct summand of $\widetilde{L}$. Let $\Lambda=L^{\perp}\subset\widetilde{L}$; by our assumptions, it is positive definite over $\Real$.
Consider the Shimura datum $(\widetilde{G}_{\Rat},\widetilde{X})$, where $\widetilde{G}$ is the smooth $\Int_{(p)}$-group scheme attached to $\widetilde{L}$: We have an embedding
\[
  (G_{\Rat},X)\into (\widetilde{G}_{\Rat},\widetilde{X}).
\]
Set $\widetilde{K}_p=\widetilde{G}(\Int_p)$, and let $\widetilde{K}^p\subset\widetilde{G}(\Adele_f^p)$ be a compact open with $K^p\subset\widetilde{K}^p$. We then get a map
\[
  \Sh_{K}\to\Sh_{\widetilde{K}}(\widetilde{G}_{\Rat},\widetilde{X})
\]
of Shimura varieties over $\Rat$, which is finite and unramified. Here, as usual, $\widetilde{K}=\widetilde{K}_p\widetilde{K}^p$. For simplicity, write $\Sh_{\widetilde{K}}$ for the second Shimura variety. We will also assume that $\widetilde{K}^p$ is small enough, so that $\Sh_{\widetilde{K}}$ is an algebraic variety, and so that we have the polarized Kuga-Satake abelian scheme $(\widetilde{A}^\KS_{\Sh_{\widetilde{K}}},\widetilde{\lambda}^{\KS})$ over it.

\subsection{}\label{special:subsec:serretensor}
We will need to briefly review the Serre tensor construction. Let $B$ be a semi-simple associative algebra over $\Rat$, and let $\Rg\subset B$ be a $\Int_{(p)}$-order. Let $M$ be a finitely generated, projective $\Rg$-module. Suppose that we are given a $\Int_{(p)}$-scheme $S$ and an abelian scheme $A\to S$ equipped with an action of $\Rg$. Then there is a canonical abelian scheme $\SHom_{\Rg}(M,A)$ over $S$ with the property that, for any $S$-scheme $T$, we have an identification of $\Int_{(p)}$-modules (cf.~\ref{spin:subsec:primetopsheaves}):
\[
 H^0(T,\SHom_{\Rg}(M,A))=\Hom_{\Rg}(M,H^0(T,A)).
\]
This essentially follows from \cite{lan}*{5.2.3.9}. We will only need the construction when $M$ is in fact a free $\Rg$-module, in which case $\SHom_{\Rg}(M,A)$ can be constructed as a product of $\rank_{\Rg}M$ copies of $A$.

Let $\bm{H}(A)$ be a degree-$1$ (de Rham, $p$-adic, $\ell$-adic or crystalline) cohomology sheaf over $S$ attached to $A$: It has a right action by $\Rg$. Then we have a canonical isomorphism:
\[
 \bm{H}(A)\otimes_{\Rg}M\xrightarrow{\simeq}\bm{H}\bigl(\SHom_{\Rg}(M,A)\bigr).
\]

\subsection{}
Now, $\widetilde{C}=C(\widetilde{L})$ is a free module over $C$ via the left multiplication action (this follows for instance from Corollaire 3 of \cite{bourbaki_algebre_ix}*{\S 9, $n^{\circ}$ 3}). So we can apply the Serre tensor construction to obtain the abelian scheme $\SHom_{C}(\widetilde{C},A^{\KS}_{\Sh_K})$.
\begin{lem}\label{special:lem:ks2ks}
$\SHom_{C}(\widetilde{C},A^\KS_{\Sh_K})$ has a natural $\Int/2\Int$-grading, as well as a $\widetilde{C}$ action compatible with the grading. Moreover, there exists a canonical $\widetilde{C}$-equivariant isomorphism of $\Int/2\Int$-graded abelian schemes over $\Sh_K$:
\[
 i:\widetilde{A}^\KS_{\Sh_K}\xrightarrow{\simeq}\SHom_{C}(\widetilde{C},A^\KS_{\Sh_K}).
\]
\end{lem}
\begin{proof}
 The $\Int/2\Int$-grading is simply the diagonal grading, and the action of $\widetilde{C}$ is via pre-composition by right multiplication.

 The proof is now quite standard, and essentially comes down to the existence of the $\widetilde{C}$-equivariant isomorphism of $\Int/2\Int$-graded $G$-representations:
 \begin{align}
  H\otimes_{C}\widetilde{C}&\rightarrow\widetilde{H}\label{special:eqn:kstoks};\\
  w\otimes z&\mapsto w\cdot z.
 \end{align}

 It gives rise to an isomorphism of tuples $\mathbb{V}_{\Rat}(H\otimes_C\widetilde{C})\xrightarrow{\simeq}\mathbb{V}_{\Rat}(\widetilde{H})$. Over $\Sh_{K,\Comp}$, $\bb{V}_{\Comp}(H\otimes_C\widetilde{C})$ (resp. $\bb{V}_{\Comp}(\widetilde{H})$) is the $\Int_{(p)}$-variation of Hodge structures obtained from the cohomology of $\SHom_{C}(\widetilde{C},A^\KS_{\Sh_{K,\Comp}})$ (resp. $\widetilde{A}^\KS_{\Sh_{K,\Comp}}$). So the induced isomorphism $\bb{V}_{\Comp}(H\otimes_C\widetilde{C})\xrightarrow{\simeq}\bb{V}_{\Comp}(\widetilde{H})$ gives us an isomorphism of abelian schemes
 \[
  i:\widetilde{A}^\KS_{\Sh_{K,\Comp}}\xrightarrow{\simeq}\SHom_{C}(\widetilde{C},A^\KS_{\Sh_{K,\Comp}}).
 \]
 But, by the functoriality of $\bb{V}_{\Rat}$, the $p$-adic realization of $i$ is defined over $\Sh_K$. From this, as in the proof of (\ref{spin:prp:prdescent}), we conclude that $i$ must be defined over $\Sh_K$.
\end{proof}

\begin{prp}\label{special:prp:lambdaemb}
There exists a canonical isometric embedding $\Lambda\into L(\widetilde{A}^\KS_{\Sh_K})$ mapping onto a direct summand. Furthermore, for any $\Sh_K$-scheme $T$, there exists a canonical embedding $L(A^\KS_T)\into L(\widetilde{A}^\KS_T)$ also mapping onto the direct summand. Under the canonical bilinear pairing $f\mapsto f\circ f$ on $L(\widetilde{A}^\KS_T)$, $L(A^\KS_T)$ is identified with the orthogonal complement of $\Lambda$.
\end{prp}
\begin{proof}
 Viewing $\Lambda$ as a trivial representation of $G$, the natural embedding $\Lambda\into\widetilde{L}$ is a map of $G$-representations. Applying the functor $\bb{V}_{\Rat}$, we obtain a map $\bb{V}_{\Rat}(\Lambda)\into\bb{V}_{\Rat}(\widetilde{L})$ and thus an embedding:
 \[
 \Lambda\into \Hom\bigl(\bb{V}_{\Rat}(\Int_{(p)}),\bb{V}_{\Rat}(\widetilde{L})\bigr)=L(\widetilde{A}^\KS_{\Sh_K}).
 \]

 Given an $\Sh_K$-scheme $T$, we obtain an embedding of $\Int_{(p)}$-algebras:
 \[
  \End(A^\KS_T)_{(p)}\into\End(\widetilde{A}^\KS_T)_{(p)}.
 \]
 This is defined as follows: Given an endomorphism $f$ of $A^\KS_T$, we obtain an endomorphism of $\widetilde{A}^\KS_T$ carrying a map $\varphi:\widetilde{C}\to H^0(T,A^\KS)$ to the map $f\circ\varphi$. Here, we are using (\ref{special:lem:ks2ks}) to identify $\widetilde{A}^\KS_T$ with $\SHom_C(\widetilde{C},A^\KS_T)$. This embedding is compatible under cohomological realizations with the map $\bb{V}_{\Rat}(H^{\otimes(1,1)})\into\bb{V}_{\Rat}(\widetilde{H}^{\otimes(1,1)})$ obtained from the inclusion of $G$-representations $H^{\otimes(1,1)}\into H^{\otimes(1,1)}\otimes_C\widetilde{C}\xrightarrow{\simeq}\widetilde{H}^{\otimes(1,1)}$.

 In fact, this embedding carries $L(A^\KS_T)$ onto $\Lambda^{\perp}\subset\widetilde{L}(A^\KS_T)$. To see this, observe that the embedding $L\into\widetilde{L}$ of $G$-representations gives us a canonical map of $p$-adic sheaves $\bm{L}_p\to\widetilde{\bm{L}}_p\vert_{\Sh_K}$. We can now identify both $L(A^\KS_T)$ and $\Lambda^{\perp}$ with those elements of $\widetilde{L}(A^\KS_T)$ whose $p$-adic realization is a section of $\bm{L}_p\subset\widetilde{\bm{L}}_p\vert_{\Sh_K}$.
\end{proof}

\subsection{}
Recall from (\ref{spin:subsec:levelstructure}) that, over $\Sh_{\widetilde{K}}$ (resp. $\Sh_K$), we have the canonical $\widetilde{G}(\Adele_f^p)$-torsor $I^p_{\widetilde{G}}$ (resp. $G(\Adele_f^p)$-torsor $I^p_G$). There is a canonical $G(\Adele_f^p)$-equivariant map $I^p_G\to I^p_{\widetilde{G}}\rvert_{\Sh_K}$, which we can describe explicitly. Recall that, $I^p_G$ parameterizes $C$-equivariant graded isomorphisms
\[
 \eta:H\otimes\underline{\Adele}_f^p\xrightarrow{\simeq}\bm{H}_{\Adele_f^p}
\]
carrying $L\otimes\underline{\Adele}_f^p$ onto $\bm{L}_{\Adele_f^p}$.

Tensoring both sides of such an isomorphism with $\widetilde{C}$ over $C$ and using the isomorphisms $H\otimes_C\widetilde{C}\xrightarrow{\simeq}\widetilde{H}$ and $\bm{H}_{\Adele_f^p}\otimes_C\widetilde{C}\xrightarrow{\simeq}\widetilde{\bm{H}}_{\Adele_f^p}$, we obtain a $\widetilde{C}$-equivariant graded isomorphism
\[
 \widetilde{\eta}:\widetilde{H}\otimes\underline{\Adele}_f^p\xrightarrow{\simeq}\widetilde{\bm{H}}_{\Adele_f^p}.
\]
It carries $L\otimes\underline{\Adele}_f^p\subset\widetilde{H}^{\otimes(1,1)}\otimes\underline{\Adele}_f^p$ onto $\bm{L}_{\Adele_f^p}\subset\bm{\widetilde{H}}^{\otimes(1,1)}_{\Adele_f^p}$. Moreover, if $\iota_0:\Lambda\into\widetilde{L}$ is the natural inclusion, $\widetilde{\eta}$ carries $\iota_0\otimes 1$ to the isometry $\bm{\iota}_{\Adele_f^p}:\Lambda\otimes\underline{\Adele}_f^p\into \widetilde{L}\otimes\underline{\Adele}_f^p$ induced by the inclusion $\Lambda\into L(\widetilde{A}^{\KS}_{\Sh_K})$. This essentially shows:
\begin{prp}\label{special:prp:levelstructures}
There is a canonical map $I^p_G\to I^p_{\widetilde{G}}\vert_{\Sh_K}$. It identifies $I^p_G$ with the sub-sheaf $I^p_{\bm{\iota}}\subset I^p_{\widetilde{G}}\vert_{\Sh_K}$ consisting of the trivializations
\[
 \widetilde{\eta}:\widetilde{H}\otimes\underline{\Adele}_f^p\xrightarrow{\simeq}\widetilde{\bm{H}}_{\Adele_f^p},
\]
which carry $\iota_0\otimes 1$ to $\bm{\iota}_{\Adele_f^p}$.
\end{prp}
\begin{proof}
  We will only need the fact that $I^p_G$ maps naturally to $I^p_{\bm{\iota}}$, which we have already seen. The remainder of the proof is left to the reader.
\end{proof}

\begin{lem}\label{regular:lem:maxemb}
There exists a self-dual quadratic lattice $(\widetilde{L},\widetilde{Q})$ over $\Int_{(p)}$, and an embedding
\[
   (L,Q)\into (\widetilde{L},\widetilde{Q})
\]
of quadratic lattices carrying $L$ onto a direct summand of $\widetilde{L}$ such that $\Lambda=L^{\perp}$ is positive definite. If $\disc(L)$ is a cyclic abelian group, we can choose $\widetilde{L}$ such that $\rank\widetilde{L}\leq \rank L+1$.
\end{lem}
\begin{proof}
 For any $a\in\Int_{(p)}$, let $\langle a\rangle$ be the rank $1$-quadratic $\Int_{(p)}$-module $\Int_{(p)}$ equipped with the quadratic form with value $a$ on $1$.

 Let $v\in L$ be such that $\nu_p(Q(v))$ is minimal. Then one easily checks that $\nu_p([w_1,w_2]_Q)\geq\nu_p(Q(v))$, for all $w_1,w_2\in L$. In particular, for any $w\in L$, the projection $w-\frac{2[w,v]_Q}{[v,v]_Q}v$ onto $\langle v\rangle^{\perp}\subset L$ is well-defined, and we obtain an orthogonal decomposition
 \[
 L=\langle v\rangle\oplus\langle v\rangle^{\perp}.
 \]
 Applying this iteratively, we find that $(L,Q)$ can be diagonalized, so that it is isometric to a lattice of the form
 \[
  (\oplus_{i=1}^n\langle a_i\rangle)\bigoplus(\oplus_{j=1}^2\langle -b_j\rangle),
 \]
 where $a_i$ and $b_j$ positive integers for varying $i$ and $j$.

For any $b\in\Int_{(p)}$, the quadratic space $\langle b\rangle$ embeds as a direct summand spanned by the element $e+\frac{1}{2}b\cdot f$ within the hyperbolic plane $U=\Int_{(p)}e\oplus\Int_{(p)}f$ with $e^2=f^2=0$ and $[e,f]=1$. Also, given any positive element $a\in\Int_{(p)}^{>0}$, we can find a self-dual positive definite quadratic space $E_+(a)$ over $\Int_{(p)}$ such that $\langle a\rangle$ embeds as a direct summand of $E_+(a)$. In other words, we need a binary quadratic form over $\Int_{(p)}$, which primitively represents $a$ and whose discriminant is a negative element of $\Int_{(p)}^\times$. Take the form to be $ax^2+xy+cy^2$ with $c=pc'$, for $c'\in\Int^{>0}_{(p)}$ satisfying: $1-4ac'<0$. The discriminant is then $1-4pac'$, which is negative and a unit in $\Int_{(p)}$.

 We now find that $L$ embeds isometrically as a direct summand of $\widetilde{L}=(\oplus_{i=1}^nE_+(a_i))\oplus H^{\oplus 2}$.

 Suppose now that $\disc(L)$ is cyclic. This implies that at most one element in the set $\{a_i:1\leq i\leq n\}\cup\{b_1,b_2\}$ is a non-unit. If every element is a unit, then we can take $\widetilde{L}=L$. Otherwise, without loss of generality, we can assume that exactly one of $a_1$ or $b_1$ is a non-unit. In the former case, we set $\widetilde{L}=E_+(a_1)\bigoplus(\oplus_{i=2}^n\langle a_i\rangle)\bigoplus(\oplus_{j=1}^2\langle -b_j\rangle)$. In the latter, we can take $\widetilde{L}=(\oplus_{i=1}^n\langle a_i\rangle)\oplus\langle -b_2\rangle\oplus U.$
\end{proof}

\subsection{}\label{regular:subsec:fixemb}
We will now fix an embedding $(L,Q)\into (\widetilde{L},\widetilde{Q})$ as in (\ref{regular:lem:maxemb}). In the case where $\disc(L)$ is cyclic, we will assume in addition that $\widetilde{L}$ has been chosen so that $\rank\Lambda\leq 1$.

Let $\Ss_{\widetilde{K}}$ (resp $\Ss_{\widetilde{K}_0}$) be the integral canonical model for $\Sh_{\widetilde{K}}$ (resp. $\Sh_{\widetilde{K}_0}$) over $\Int_{(p)}$ (cf.~\ref{spin:subsec:intcan}). Assume that $\widetilde{K}$ is small enough. Then, over $\Ss_{\widetilde{K}}$, we have the Kuga-Satake abelian scheme $(\widetilde{A}^\KS,\widetilde{\lambda}^\KS,\bigl[\widetilde{\eta}^\KS\bigr])$. Since $\Ss_{\widetilde{K}}$ is smooth, and in particular normal, both the $\widetilde{G}(\Adele_f^p)$-torsor $I^p_{\widetilde{G}}$ and the canonical $K^p$-level structure $[\widetilde{\eta}_{\widetilde{G}}]\in H^0\bigl(\Sh_{\widetilde{K}},I^p_{\widetilde{G}}/\widetilde{K}^p\bigr)$ (cf.~\ref{spin:subsec:levelstructure}) extend over $\Ss_K$: We will denote these extensions by the same symbols.

The $G_0(\Adele_f^p)$-torsor $I^p_{\widetilde{G}}/\Gm(\Adele_f^p)$ has a canonical descent $I^p_{\widetilde{G}_0}$ over $\Ss_{\widetilde{K}_0}$: it parameterizes certain orientation preserving isometries
\[
 \widetilde{\eta}_0:\widetilde{L}\otimes\underline{\Adele}_f^p\xrightarrow{\simeq}\widetilde{\bm{L}}_{\Adele_f^p}.
\]
Again, we have a canonical section $[\widetilde{\eta}_{\widetilde{G}_0}]$ of $I^p_{\widetilde{G}_0}/\widetilde{K}^p_0$ over $\Ss_{\widetilde{K}_0}$.

\subsection{}
For any $\Ss_{\widetilde{K}_0}$-scheme $T$, write $\widetilde{L}(T)$ for the group of special endomorphisms defined as in (\ref{lifts:subsec:ortho}); in particular, if $T$ is actually a scheme over $\Ss_{\widetilde{K}}$, we will have $\widetilde{L}(T)=L(\widetilde{A}^{\KS}_T)$.

Let $\iota_0:\Lambda\into\widetilde{L}$ be the natural embedding. Suppose that we are given an isometric map $\iota:\Lambda\to L(T)$; since $\Lambda$ is positive definite any such map has to be an embedding. Let $I^p_{\iota,0}\subset I^p_{\widetilde{G}_0}\rvert_T$ be the sub-sheaf of isomorphisms $\widetilde{\eta}_0$ that carry $\iota_0$ to $\iota$; then, by Witt's extension theorem, $I^p_{\iota,0}$ is a torsor over $T$ under $G_0(\Adele_f^p)$.

If the map $T\to\Ss_{\widetilde{K}_0}$ arises from a lift $T\to\Ss_{\widetilde{K}}$, then one can also define a sub-sheaf $I^p_{\iota}\subset I^p_{\widetilde{G}}\rvert_T$: This is the pre-image of $I^p_{\iota,0}$ under the natural quotient map $I^p_{\widetilde{G}}\rvert_T\to I^p_{\widetilde{G}_0}\rvert_T$, and is thus a $G(\Adele_f^p)$-torsor over $T$.

\begin{defn}\label{regular:defn:lambdastructure}
For any $\Ss_{\widetilde{K}_0}$-scheme $T$, a \defnword{$\Lambda$-structure} for $T$ is an isometric map $\iota:\Lambda\to L(T)$. Given an $\Ss_{\widetilde{K}}$-scheme $T$ and a $\Lambda$-structure $\iota$ for $T$, a \defnword{$K^p$-level structure} on $(T,\iota)$ is a section $[\eta_{\iota}]$ of $I^p_{\iota}/K^p$ over $T$ mapping to $[\widetilde{\eta}_{\widetilde{G}}]$ under the obvious map
\[
  I^p_{\iota}/K^p\rightarrow \bigl(I^p_{\widetilde{G}}/\widetilde{K}^p\bigr)\rvert_T.
\]
Completely analogously, supposing only that $T$ is a $\Ss_{\widetilde{K}_0}$-scheme one can also define the notion of a $K^p_0$-level structure on $(T,\iota)$ as a section $[\eta_{\iota,0}]$ of $I^p_{\iota,0}/K^p_0$ mapping to $[\widetilde{\eta}_{\widetilde{G}_0}]$ in $H^0(T,I^p_{\widetilde{G}_0}/\widetilde{K}^p_0)$.
\end{defn}

\subsection{}\label{regular:subsec:zkplambda}
Let $Z_{K^p}(\Lambda)$ be the functor on $\Ss_{\widetilde{K}}$-schemes whose value on any $\Ss_{\widetilde{K}}$-scheme $T$ is given by:
\[
  Z_{K^p}(\Lambda)(T)=\left\{
  (\iota,[\eta_{\iota}]):\;\text{$\iota$ a $\Lambda$-structure for $T$; $[\eta_{\iota}]$ a $K^p$-level structure for $(T,\iota)$}
  \right\}.
\]
Similarly, one defines a functor $Z_{K^p_0}(\Lambda)$ on $\Ss_{\widetilde{K}_0}$-schemes employing $K^p_0$-level structures.

\begin{prp}\label{regular:prp:representable}
$Z_{K^p}(\Lambda)$ (resp. $Z_{K_0^p}(\Lambda)$) is represented by a scheme finite and unramified over $\Ss_{\widetilde{K}}$ (resp. over $\Ss_{\widetilde{K}_0}$). Moreover, the natural map $Z_{K^p}(\Lambda)\to Z_{K_0^p}(\Lambda)$ is finite \'etale.
\end{prp}
\begin{proof}
  It is clear from the definitions that the map $Z_{K^p}(\Lambda)\to Z_{K_0^p}(\Lambda)$ is finite \'etale, so it is enough to prove the remaining assertions for $Z_{K^p}(\Lambda)$. To show representability, we first note that $\underline{\End}(\widetilde{A}^\KS)_{(p)}$ is representable over $\Ss_{\widetilde{K}}$ by an inductive limit of schemes that are locally of finite type. Indeed, if we fix a representative $\widetilde{A}$ in the prime-to-$p$ isogeny class of $\widetilde{A}^\KS$, then the endomorphism scheme $\underline{\End}(\widetilde{A})$ is known to representable by a scheme locally of finite type over $\Ss_{\widetilde{K}}$, and we can identify $\underline{\End}(\widetilde{A}^\KS)_{(p)}$ with the inductive limit of the system $\{\underline{\End}(\widetilde{A})\}_{m\in\Int\backslash p\Int}$, where, for $m_1\mid m_2$, the transition map from the copy of $\underline{\End}(\widetilde{A})$ in the $m_1^{\text{th}}$ position to that in the $m_2^{\text{th}}$-position is given by multiplication by $m_2/m_1$. Since the property of being a special endomorphism is a closed condition on the base, we see that $Z_{K^p}(\Lambda)$ is also represented by an ind-scheme. To show that it is in fact represented by a scheme, it is enough to show that it is finite and unramified over $\Ss_{\widetilde{K}}$.

  The unramifiedness is a consequence of the fact that endomorphisms of abelian schemes lift uniquely (if at all) over nilpotent thickenings. The Ner\'onian property of abelian schemes over discrete valuation rings combines with the valuative criterion to show that $Z_{K^p}(\Lambda)$ is proper.

  Thus, it only remains to show quasi-finiteness. For this, take any geometric point $s\to\Ss_{\widetilde{K}}$. We view $[\widetilde{\eta}_{\widetilde{G}_0,s}]$ as a $\widetilde{K}_0^p$-orbit of isometries
  \[
   \widetilde{\eta}_{0,s}:\widetilde{L}\otimes\Adele_f^p\xrightarrow{\simeq}\widetilde{\bm{L}}_{\Adele_f^p,s}.
  \]
  If $\bm{\iota}$ is a $\Lambda$-structure for $s$ such that $(s,\bm{\iota})$ admits a $K^p_0$-level structure, then, for any $v\in\Lambda$, the pre-image of the $\Adele_f^p$-realization of $\bm{\iota}(v)$ under $\widetilde{\eta}_{0,s}$ must lie within the set $\widetilde{L}\cap K^p_0\cdot\iota_0(v)$. Since $\widetilde{L}$ is a discrete sub-group of $\widetilde{L}_{\Adele_f^p,s}$, this set is finite. So we see that the possible $\Lambda$-structures for $s$ (admitting a $K^p_0$-level structure) are finite in number. From this, it follows easily that the fiber of $Z_{K^p}(\Lambda)$ over $s$ is finite.
\end{proof}

The following result is presumably well-known, but we include it for lack of reference.
\begin{lem}\label{spin:lem:endderham}
Let $W=W(k)$ and let $E/W_{\Rat}$ be a finite extension of ramification index $e\leq p-1$. Let $A$ be an abelian scheme over $\Reg{E}$ with special fiber $A_0$. Then $\End(A_{E})\otimes W$ embeds in $\End_W(H^1_{\cris}(A_0/W))$ as a direct summand. In particular, $\End(A_E)\otimes k$ maps injectively into $\End_k(H^1_{\dR}(A_0/k))$.
\end{lem}
\begin{proof}
  Without loss of generality, we can assume that $k$ is algebraically closed. There is a canonical comparison isomorphism:
  \begin{align}\label{spin:eqn:derham}
   H^1_{\cris}(A_0/W)\otimes_W\Reg{E}\xrightarrow{\simeq}H^1_{\dR}(A/\Reg{E}).
  \end{align}
  Note that this is where we need the hypothesis $e\leq p-1$, which ensures that the map $\Reg{E}\to k$ is a divided power thickening.

  Set $M=\End_W(H^1_{\cris}(A_0/W))$: We have a direct summand $M^0\subset M$ consisting of endomorphisms that preserve the Hodge filtration $F^1H^1_{\dR}(A/\Reg{E})$. Here, we use the canonical isomorphism (\ref{spin:eqn:derham}) to view $M$ as a group of endomorphisms of $H^1_{\dR}(A/\Reg{E})$. The conjugation action of the semi-linear Frobenius on $H^1_{\cris}(A_0/W)$ induces an isomorphism $F:\sigma^*M\left[\frac{1}{p}\right]\xrightarrow{\simeq}M\left[\frac{1}{p}\right]$ such that $F(\sigma^*M^0)\subset M$. The last condition holds because the image of $F^1H^1_{\dR}(A/\Reg{E})$ in $H^1_{\dR}(A_0/k)$ has the following property: Its pull-back via $\sigma$ is precisely the kernel of the map $F:\on{Fr}^*H^1_{\dR}(A_0/k)\to H^1_{\dR}(A_0/k)$.

  One can deduce from classical Dieudonn\'e theory and Grothendieck-Messing theory that there is a canonical isomorphisms of $\Int_p$-modules:
  \[
   \End(A[p^\infty])\xrightarrow{\simeq}(M^0)^{F=1}.
  \]
  Moreover, by the Ner\'onian property, $\End(A_{E})=\End(A)$. Also, $\End(A)\otimes\Int_p$ is a direct summand of $\End(A[p^\infty])$, since any element of $\End(A)$ that kills the $p$-torsion $A[p]$ has to be divisible by $p$.

  So, to finish the proof, we have to show that $(M^0)^{F=1}\otimes W$ maps onto a direct summand of $M$. For this, set
\[
 M'=\{m\in M^0:F^i(m)\in M^0\text{, for all $i\in\Int_{\geq 0}$}\}.
\]
  We need to explain what we mean by $F^i(m)$. One defines this inductively: We set $F(m)=F(\sigma^*m)$ and $F^i(m)=F(\sigma^*F^{i-1}(m))$, where we are using the assumption that $F^{i-1}(m)\in M^0$ in each inductive step.

  Now, $F$ restricts to a (necessarily injective) map $\sigma^*M'\to M'$. Moreover, $M'$ is a direct summand of $M$, and $(M^0)^{F=1}\otimes W$ clearly maps into $M'$. It follows from the Dieudonn\'e-Manin classification of $F$-crystals over $W$~\cite{manin} that there exists a largest $F$-stable direct summand $M^{\et}\subset M'$ such that $F$ induces an isomorphism $\sigma^*M^{\et}\xrightarrow{\simeq}M^{\et}$. Moreover, it is known~\cite{katz:p-adic}*{4.1.1} that the map
  \[
   (M^{\et})^{F=1}\otimes_{\Int_p}W\to M^{\et}
  \]
  is an isomorphism. It follows \emph{a fortiori} that the map $(M^0)^{F=1}\otimes W\to M$ identifies its source with the direct summand $M^{\et}\subset M$.
\end{proof}

\subsection{}\label{regular:subsec:zkpprps}
The map $\Sh_{K_0}\to\Sh_{\widetilde{K}_0}$ canonically lifts to a map $\Sh_{K_0}\to Z_{K_0^p}(\Lambda)$. Indeed, it is enough to show that $\Sh_K\to\Sh_{\widetilde{K}}$ lifts canonically to a map $\Sh_K\to Z_{K^p}(\Lambda)$. This follows from (\ref{special:prp:lambdaemb}) and (\ref{special:prp:levelstructures}).

Let $\bm{\iota}$ be the tautological $\Lambda$-structure over $Z_{K^p_0}(\Lambda)$ and let $\bm{\Lambda}_{\dR}=\Lambda\otimes\Reg{Z_{K^p_0}(\Lambda)}$. The map $\bm{\iota}$ induces a map of vector bundles:
\[
\bm{\iota}_{\dR}:\bm{\Lambda}_{\dR}\to\widetilde{\bm{L}}_{\dR}\rvert_{Z_{K^p_0}(\Lambda)} .
\]
\begin{lem}\label{regular:lem:vectbundle}
\mbox{}
\begin{enumerate}
\item\label{vectbundle:prim}There exists an open sub-scheme $Z^{\on{pr}}_{K^p_0}(\Lambda)\subset Z_{K^p_0}(\Lambda)$ such that $T\to Z_{K^p_0}(\Lambda)$ factors through $Z^{\on{pr}}_{K^p_0}(\Lambda)$ if and only if the restriction of $\coker\bm{\iota}_{\dR}$ to $T$ is a vector bundle of rank $n+2$.
\item\label{vectbundle:primchar0} $Z^{\on{pr}}_{K^p_0}(\Lambda)$ contains the generic fiber $Z_{K^p_0}(\Lambda)_{\Rat}$.
\item\label{vectbundle:flat}Let $W=W(\overline{\Field}_p)$, and suppose that we have $\widetilde{s}:\Spec W\to Z_{K^p_0}(\Lambda)$ such that the restriction to $\Spec W_{\Rat}$ factors through $\Sh_{K_0}$. Then $\widetilde{s}$ factors through $Z^{\on{pr}}_{K^p_0}(\Lambda)$.
\item\label{vectbundle:smoothext}Suppose that $T$ is smooth over $\Int_{(p)}$ and that $f:T\to Z_{K^p_0}(\Lambda)$ is such that $f\vert_{T_{\Rat}}$ factors through $\Sh_{K_0}$. Then $f$ factors through $Z^{\on{pr}}_{K^p_0}(\Lambda)$.
\item\label{vectbundle:maximal}If $L$ has square-free discriminant, then $Z^{\on{pr}}_{K^p_0}(\Lambda)=Z_{K^p_0}(\Lambda)$.
\end{enumerate}
\end{lem}\begin{proof}
  We begin with (\ref{vectbundle:prim}): Choose a basis element for the top exterior power $\wedge^d\Lambda$, identifying it with $\Int_{(p)}$. Then the map $\bm{\iota}_{\dR}$ induces a map
  \[
  \Reg{Z_{K^p_0}(\Lambda)}=\wedge^d\bm{\Lambda}_{\dR}\xrightarrow{\wedge^d\bm{\iota}_{\dR}}\wedge^d\widetilde{\bm{L}}_{\dR},
  \]
  and hence a section $e\in H^0(Z_{K^p_0}(\Lambda),\wedge^d\widetilde{\bm{L}}_{\dR})$. One easily checks now that $Z^{\on{pr}}_{K^p_0}(\Lambda)$ is the open locus where this section does not vanish.

  To prove (\ref{vectbundle:primchar0}), it suffices to show that $Z^{\on{pr}}_{K^p_0}(\Lambda)$ contains every point $s\in Z_{K^p_0}(\Lambda)(\Comp)$. But the map $\bm{\Lambda}_{\dR,s}\to\widetilde{\bm{L}}_{\dR,s}$ arises form an isometric embedding of $\Int_{(p)}$-Hodge structures $\bm{\iota}_s:\Lambda\into\widetilde{\bm{L}}_{B,s}$. In particular, it has to be an embedding of $\Comp$-vector spaces.

  Define $Z^{\on{pr}}_{K^p}(\Lambda)$ to be the pre-image in $Z_{K^p}(\Lambda)$ of $Z^{\on{pr}}_{K^p_0}(\Lambda)$. It suffices to prove the remaining assertions for $\Ss_{\widetilde{K}}$-schemes with $Z_{K^p_0}^{\on{pr}}(\Lambda)$ replaced everywhere by $Z^{\on{pr}}_{K^p}(\Lambda)$.

  We will now consider (\ref{vectbundle:flat}). Choose an algebraic closure $\overline{E}/W(k)_{\Rat}$. By (\ref{spin:lem:endderham}), it suffices to show that the \'etale realization of $\bm{\iota}(\Lambda)$ generates a direct summand of $\widetilde{\bm{L}}_{p,\widetilde{s}_{\overline{E}}}$. But, over $\Sh_{K_0}$, the sub-space generated by this realization is globally a direct summand of $\widetilde{\bm{L}}_p$. Indeed, its inclusion in the latter is induced by the map of $G_0$-representations $\Lambda\into\widetilde{L}$.

  (\ref{vectbundle:smoothext}) now follows: Indeed, (\ref{vectbundle:prim}) and (\ref{vectbundle:flat}) show that $U\coloneqq f^{-1}\bigl(Z^{\on{pr}}_{K^p}(\Lambda)\bigr)$ is an open sub-scheme of $T$ containing $T_{\Rat}$, and through which all the $W(\overline{\Field}_p)$-valued points of $T$ factor. Since $T$ is smooth over $\Int_{(p)}$, this implies that all the $\overline{\Field}_p$-points of $T$ factor through $U$, and so $U$ must be all of $T$.

  Now assume that $L$ has square-free discriminant; then $\on{disc}(\Lambda)$ is either trivial or isomorphic to $\Int/p\Int$. In the first case, $Z_{K_0^p}(\Lambda)$ is smooth, and so we are done by~\eqref{vectbundle:smoothext}. In the second case, by our assumptions, $\Lambda$ is of rank $1$, and is generated by an element $v$ satisfying $\ord_p(Q(v))=1$. In particular, given any $\overline{\Field}_p$-point $x_0$ of $Z_{K^p}(\Lambda)$, the crystalline realization of the associated special endomorphism $f_0\in L(\widetilde{A}^\KS_{x_0})$ must necessarily generate a direct summand of $\bm{L}_{\cris,x_0}$. This shows~\eqref{vectbundle:maximal}.
\end{proof}

The following result is directly inspired by~\cite{ogus:ss}*{Theorem 2.9}.
\begin{prp}\label{regular:prp:smooth}
Let $Z^{\on{sm}}$ be the complement in $Z_{K^p}(\Lambda)$ of the non-smooth locus in $Z_{K^p}(\Lambda)_{\Field_p}$. Let $\overline{\eta}$ be a generic point of $Z^{\on{sm}}_{\Field_p}$ with algebraically closed residue field. Suppose that $n>r=\rank(\Lambda)$. Then the abelian variety $\widetilde{A}^{\KS}_{\overline{\eta}}$ is ordinary and the map:
\[
 \iota:\Lambda\to L\bigl(\widetilde{A}^{\KS}_{\overline{\eta}}\bigr)
\]
is an isomorphism.
\end{prp}
\begin{proof}
It follows from~\eqref{lifts:cor:smooth} that $\Lambda\otimes k(\overline{\eta})$ maps injectively into $\gr^0_F\bm{\widetilde{L}}_{\dR,\overline{\eta}}$. Therefore, over a sufficiently small \'etale neighborhood $Z_0$ of $\overline{\eta}$ in $Z^{\on{sm}}_{\Field_p}$, the de Rham realization $\bm{\Lambda}_{\dR}$ of $\iota(\Lambda)$ will map isomorphically onto a direct summand $\overline{\bm{\Lambda}}_{\dR}$ of $\gr^0_F\bm{\widetilde{L}}_{\dR}\vert_{Z_0}$. 

Now, consider the Frobenius map $\on{Fr}^*\bm{H}_{\dR}\vert_{Z_0}\to\bm{H}_{\dR}\vert_{Z_0}$: Its image is a Lagrangian sub-space $\bm{C}\subset\bm{H}_{\dR}\vert_{Z_0}$, and the locus where $\widetilde{A}^{\KS}\vert_{Z_0}$ is ordinary is precisely the open sub-space where $\bm{C}+F^1\bm{H}_{\dR}\vert_{Z_0}=\bm{H}_{\dR}\vert_{Z_0}$. By the discussion in~\eqref{spin:subsec:canonical_filt}, after localizing on $Z_0$ if required, the annihilator of $\bm{C}$ is a parallel isotropic line $\bm{N}\subset\widetilde{\bm{L}}_{\dR}\vert_{Z_0}$. So we can also describe the non-ordinary locus as the closed sub-space where we have:
\[
\bm{F}^1\widetilde{L}_{\dR}\vert_{Z_0}\subset F^0\widetilde{\bm{L}}_{\dR}\vert_{Z_0}\cap\bm{\Lambda}_{\dR}^{\perp}\cap\bm{N}^{\perp}.
\]

Now, by~\eqref{spin:lem:conj_hodge}, we can assume that $\bm{N}+\bm{\Lambda}_{\dR}$ is a horizontal direct summand of rank $r+1$ in $\widetilde{\bm{L}}_{\dR}\vert_{Z_0}$. Then, arguing as in the proof of~\cite{ogus:ss}*{Theorem 2.9}, we find that the locus where $\bm{F}^1\bm{L}_{\dR}\vert_{Z_0}$ is contained in this direct summand has dimension at most $r$. Since $n>r$, we can throw this locus out and assume that the summand does not contain $\bm{F}^1\bm{L}_{\dR}\vert_{Z_0}$.

Equivalently, we can assume that $\bm{N}$ is not contained in $\bm{F}^1\bm{L}_{\dR}\vert_{Z_0}+\bm{\Lambda}_{\dR}$; or that $\bm{F}^0\bm{L}_{\dR}\vert_{Z_0}\cap\bm{\Lambda}_{\dR}^{\perp}$ is not contained in $\bm{N}^{\perp}$. Again, arguing as in~\cite{ogus:ss}*{Theorem 2.9}, we now find that the non-ordinary locus is smooth of dimension $n-1$. This proves that $\widetilde{A}^{\KS}_{\overline{\eta}}$ must be ordinary.

Suppose that $\iota$ is not an isomorphism. Then we can find a non-zero element $f\in L\bigl(\widetilde{A}^{\KS}_{\overline{\eta}}\bigr)$ such that $\iota(\Lambda)+\langle f\rangle\subset L\bigl(\widetilde{A}^{\KS}_{\overline{\eta}}\bigr)$ is a direct summand of rank $r+1$.  

Now, by shrinking $Z_0$ if necessary, we can assume that $f$ extends to an element in $L\bigl(\widetilde{A}^{\KS}_{Z_0}\bigr)$. Let $\bm{f}_{\dR}$ be the de Rham realization of $f$, and let $\bm{\Lambda}_{\dR}\subset\widetilde{\bm{L}}_{\dR}\vert_{Z_0}$ be that of $\iota(\Lambda)$. 

Since $\widetilde{A}^{\KS}_{\overline{\eta}}$ is ordinary, the map
\begin{align}\label{regular:eqn:injective}
 L\bigl(\widetilde{A}^{\KS}_{\overline{\eta}}\bigr)\otimes k(\overline{\eta})\to\widetilde{\bm{L}}_{\dR,\overline{\eta}}
\end{align}
is injective. Therefore, after further localizing on $Z_0$, we can assume that the sub-sheaf:
\[
 \bm{\Lambda}_{\dR}+\langle\bm{f}_{\dR}\rangle\subset\bm{\widetilde{L}}_{\dR}\vert_{Z_0}
\]
is a local direct summand of rank $r+1$.

Using~\eqref{lifts:cor:smooth} and~\cite{ogus:ss}*{Remark 2.8}, we find that the locus in $Z_0$ where $F^1\widetilde{\bm{L}}_{\dR}$ is contained in $\bm{\Lambda}_{\dR}+\langle\bm{f}_{\dR}\rangle$ is a closed sub-space of dimension at most $1(r+1-1)=r$. Since $n>r$ by hypothesis, and since $\dim Z_0=n$, this implies that we can shrink $Z_0$ further and assume that $F^1\widetilde{\bm{L}}_{\dR}\vert_{Z_0}$ is not contained in $\bm{\Lambda}_{\dR}+\langle\bm{f}_{\dR}\rangle$. Since $F^1\widetilde{\bm{L}}_{\dR}\vert_{Z_0}$ and $\langle\bm{f}_{\dR}\rangle$ are both local direct summands in $\widetilde{\bm{L}}_{\dR}\vert_{Z_0}$ of rank $1$, it follows that $\langle\bm{f}_{\dR}\rangle$ is not contained in $F^1\widetilde{\bm{L}}_{\dR}\vert_{Z_0}+\bm{\Lambda}_{\dR}$. 

Therefore, if $\overline{\bm{\Lambda}}_{\dR}$ is the image of $\bm{\Lambda}_{\dR}$ in $\gr^0_F\widetilde{\bm{L}}_{\dR}\vert_{Z_0}$, we see that the map
\[
 \langle\bm{f}_{\dR}\rangle\to \frac{\gr^0_F\widetilde{\bm{L}}_{\dR}\vert_{Z_0}}{\overline{\bm{\Lambda}}_{\dR}}
\]
is non-zero. But, on the other hand, $\bm{f}_{\dR}$ is horizontal for the Gauss-Manin connection, which contradicts the fact~\eqref{lifts:cor:smooth} that the Kodaira-Spencer map:
\[
 \frac{\gr^0_F\widetilde{\bm{L}}_{\dR}\vert_{Z_0}}{\overline{\bm{\Lambda}}_{\dR}}\to\gr^{-1}_F\widetilde{\bm{L}}_{\dR}\otimes\Omega^1_{Z_0/\Field_p}
\]
is an isomorphism.
\end{proof}

\begin{corollary}\label{regular:cor:smooth}
Set $\Lambda'=\Lambda\perp\langle m\rangle$, for some $m\in\Int_{(p)}^{>0}$. The map $Z_{K^p}(\Lambda')\to\Ss_{\widetilde{K}}$ factors through $Z_{K^p}(\Lambda)$ in the obvious way. Suppose that $n>\rank(\Lambda)$; then the scheme
\[
 Z_{K^p}(\Lambda')\times_{Z_{K^p}(\Lambda)}Z^{\on{sm}}
\]
is flat over $\Int_{(p)}$.
\end{corollary}
\begin{proof}
Let $Z'$ be the fiber product under consideration: It is, \'etale locally on $Z'$, an effective Cartier divisor on $Z^{\on{sm}}$. Since $Z^{\on{sm}}$ is smooth over $\Int_{(p)}$, $Z'$ fails to be flat precisely when its image in $Z^{\on{sm}}$ contains an irreducible component $Z_0\subset Z^{\on{sm}}_{\Field_p}$. Suppose this is the case, and let $\overline{\eta}$ be a geometric generic point of $Z_0$. Then we find that the map $\Lambda\to L(\widetilde{A}^{\KS}_{\overline{\eta}})$ extends to an isometry on $\Lambda'$. But, by~\eqref{regular:prp:smooth}, this is impossible. 
\end{proof}

\subsection{}
Recall from (\ref{spin:cor:derhamrealization}) that we have a canonical $\widetilde{G}$-torsor $\widetilde{\mathcal{P}}_{\dR}$ over $\Ss_{\widetilde{K}}$ consisting of $\widetilde{G}$-structure preserving trivializations of $\widetilde{\bm{H}}_{\dR}$. Let $\bm{\iota}$ be the tautological $\Lambda$-structure over $Z(\Lambda)$, and let $\bm{\iota}_{\dR}:\bm{\Lambda}_{\dR}\into\widetilde{\bm{L}}_{\dR}$ be the de Rham realization of $\bm{\iota}$. Define $\mathcal{P}_{\dR,\Lambda}\subset\widetilde{\mathcal{P}}_{\dR,Z_{K^p}(\Lambda)}$ to be the $G$-equivariant sub-functor such that, for any $Z_{K^p}(\Lambda)$-scheme $T$, we have:
\[
 \mathcal{P}_{\dR,\Lambda}(T)=\{\xi\in\widetilde{\mathcal{P}}_{\dR}(T):\; \xi\circ(\iota_0\otimes 1)=\bm{\iota}_{\dR}\}.
\]

\begin{prp}\label{regular:prp:locmodel}
\mbox{}
\begin{enumerate}[itemsep=0.12in]
  \item\label{plamgtorsor}The restriction of $\mathcal{P}_{\dR,\Lambda}$ over $Z^{\on{pr}}_{K^p}(\Lambda)$ is a $G$-torsor.
  \item\label{p2smooth}The map $p_2:\mathcal{P}_{\dR,\Lambda}\to\on{M}^{\loc}_G$, given, for any $\Int_{(p)}$-scheme $T$, by:
  \begin{align*}
   \mathcal{P}_{\dR,\Lambda}(T)&\rightarrow\on{M}^{\loc}_G(T)\\
   (x,\xi)&\mapsto\xi^{-1}\bigl(F^1\widetilde{\bm{L}}_{\dR,x}\bigr).
   \end{align*}
   is $G$-equivariant and smooth of relative dimension $\dim G_{\Rat}$. Here, $(x,\xi)\in\mathcal{P}_{\dR,\Lambda}(T)$ lies over a point $x\in Z_{K^p}(\Lambda)(T)$.
  \item\label{etalesection}Around any point $x\in Z^{\on{pr}}_{K^p}(\Lambda)$ there exists an \'etale neighborhood $U\to Z^{\on{pr}}_{K^p}(\Lambda)$ and a section $s:U\to\mathcal{P}_{\dR,\Lambda}$ of the $G$-torsor $\mathcal{P}_{\dR,\Lambda}$ such that the induced map $p_2\circ s: U\to \on{M}^{\loc}_G$ is \'etale.
\end{enumerate}
\end{prp}
\begin{proof}
  We first note that the basic ideas for the proof can already be found in \cite{rapzink}*{\S 3} and \cite{pappas:jag}*{Thm. 2.2}.

  (\ref{plamgtorsor}) is an easy consequence of the definition of $Z^{\on{pr}}_{K^p}(\Lambda)$ and (\ref{lattice:lem:torsor}). (\ref{p2smooth}) is essentially a consequence of (\ref{lifts:prp:nillifts}) and the formal lifting criterion for smoothness of a finitely presented morphism. Here are the details: It is enough to check that $p_2$ is smooth over the closed points of $\on{M}^{\loc}_G$. Suppose therefore that we are given a surjection of $\Int/p^n\Int$-algebras $\Rg\to\overline{\Rg}$ with square zero kernel. We need to show that the map:
  \begin{align}\label{locmodel:eqn1}
   \varphi:\mathcal{P}_{\dR,\Lambda}(\Rg)\to\mathcal{P}_{\dR,\Lambda}(\overline{\Rg})\times_{\on{M}^{\loc}_G(\overline{\Rg})}\on{M}^{\loc}_G(\Rg)
  \end{align}
  is surjective. So suppose that we have a pair $(\overline{x},\overline{\xi})\in\mathcal{P}_{\dR,\Lambda}(\overline{\Rg})$. We obtain an isotropic line
  \[
  F^1L_{\overline{\Rg}}=\overline{\xi}^{-1}\bigl(F^1\widetilde{\bm{L}}_{\dR,\overline{x}}\bigr)\subset L_{\overline{\Rg}}.
  \]

  Let $\widetilde{\bm{H}}_{\Rg}$ is the evaluation of the crystal $\widetilde{\bm{H}}_{\cris}$ on $\Spec\overline{\Rg}\into\Spec\Rg$. As usual, if $\widetilde{\bm{L}}_{\Rg}$ is the evaluation of the crystal $\widetilde{\bm{L}}_{\cris}$ along the same thickening, we obtain an embedding $\widetilde{\bm{L}}_{\Rg}\subset\End(\widetilde{\bm{H}}_{\Rg})$. We also have the crystalline realization $\bm{\iota}_{\Rg}:\Lambda\otimes\Rg\into\widetilde{\bm{L}}_{\Rg}$ of $\bm{\iota}$; write $\bm{\iota}_{\overline{\Rg}}$ for its change of scalars along $\Rg\to\overline{\Rg}$.

  There is a canonical isomorphism $\widetilde{\bm{H}}_{\dR,\overline{x}}\xrightarrow{\simeq}\widetilde{\bm{H}}_{\Rg}\otimes_{\Rg}\overline{\Rg}$. Composing this with the trivialization $\overline{\xi}$ gives us a $\widetilde{G}$-structure preserving isomorphism $\xi_{\overline{\Rg}}:\widetilde{H}\otimes\overline{\Rg}\xrightarrow{\simeq}\widetilde{\bm{H}}_{\overline{\Rg}}$ carrying $\iota_0\otimes 1$ to $\bm{\iota}_{\overline{\Rg}}$.

  Suppose now that we are also given $F^1L_{\Rg}\in\on{M}^{\loc}_G(\Rg)$ lifting $F^1L_{\overline{\Rg}}$. Let $\mathcal{P}(\Rg\to\overline{\Rg})$ be the set of $\widetilde{G}$-structure preserving isomorphisms $\xi_{\Rg}:\widetilde{H}\otimes\Rg\xrightarrow{\simeq}\widetilde{\bm{H}}_{\Rg}$ carrying $\iota_0\otimes 1$ to $\bm{\iota}_{\Rg}$, and lifting $\xi_{\overline{\Rg}}$.

  Observe now that this set is non-empty: We can lift $\overline{x}$ to any $\Rg$-valued point $x'\in\Ss_{\widetilde{K}}(\Rg)$. This gives a canonical identification $\widetilde{\bm{H}}_{\dR,x'}\xrightarrow{\simeq}\widetilde{\bm{H}}_{\overline{\Rg}}$. Since $\widetilde{\mathcal{P}}_{\dR}$ is a $\widetilde{G}$-torsor over $\Ss_{\widetilde{K}}$, there now exists a $\widetilde{G}$-structure preserving trivialization $\xi'_{\Rg}:\widetilde{H}\otimes\Rg\xrightarrow{\simeq}\widetilde{\bm{H}}_{\Rg}$ lifting $\xi_{\overline{\Rg}}$. If $\xi'_{\Rg}$ carries $\iota_0\otimes 1$ to $\bm{\iota}_{\Rg}$, we are done. Otherwise, it follows from (\ref{lattice:lem:torsor}) that we can compose it with an element of $\widetilde{G}(\Rg)$ to ensure that it lies in $\mathcal{P}(\Rg\to\overline{\Rg})$. In fact, \emph{loc. cit.} implies that $\mathcal{P}(\Rg\to\overline{\Rg})$ is a non-empty torsor under $\ker(G(\Rg)\to G(\overline{\Rg}))$.

  We claim that $\mathcal{P}(\Rg\to\overline{\Rg})$ in canonical bijection with the fiber $\varphi^{-1}((\overline{x},\overline{\xi}),F^1L_{\Rg})$. This will show that $\varphi$ is smooth. If we set $\Rg=k[\epsilon]$ and $\overline{\Rg}=k$, we find that the fiber of $\varphi$ over any $k$-valued point is a torsor under $\ker(G(k[\epsilon])\to G(k))=\Lie G\otimes k$. So our claim would also show that $\varphi$ has relative dimension $\dim G_{\Rat}$.

  Let us prove the claim: Given such any $\xi_{\Rg}\in \mathcal{P}(\Rg\to\overline{\Rg})$, $\xi_{\Rg}(F^1L_{\Rg})\subset\widetilde{\bm{L}}_{\Rg}$ is an isotropic line lifting $F^1\widetilde{\bm{L}}_{\dR,\overline{x}}$ and isotropic to the image of $\bm{\iota}_{\Rg}$. By (\ref{lifts:prp:nillifts}), this determines a lift $x\in Z^{\on{pr}}_{K^p}(\Lambda)(\Rg)$ of $\overline{x}$. Furthermore, there exists a canonical isomorphism $\widetilde{\bm{H}}_{\Rg}\xrightarrow{\simeq}\widetilde{\bm{H}}_{\dR,x}$ carrying $F^1\widetilde{\bm{L}}_{\Rg}$ onto $F^1\widetilde{\bm{L}}_{\dR,x}$. The composition $\xi$ of this isomorphism with $\xi_{\Rg}$ gives us a lift $(x,\xi)\in \varphi^{-1}((\overline{x},\overline{\xi}),F^1L_{\Rg})$. Using the bijectivity of the correspondence in (\ref{lifts:prp:nillifts}), one can check that the assignment defined in this fashion is in fact a bijection from $\mathcal{P}(\Rg\to\overline{\Rg})$ to $\varphi^{-1}((\overline{x},\overline{\xi}),F^1L_{\Rg})$.

  For (\ref{etalesection}), it suffices to prove the result in a neighborhood of a closed point $x\in Z^{\on{pr}}_{K^p}(\Lambda)(\overline{\Field}_p)$. Fix any section $s_0:\Spec\overline{\Field}_p\to\mathcal{P}_{\dR,\Lambda}$ over $x$, and let $\xi_0:\widetilde{H}\otimes\overline{\Field}_p\xrightarrow{\simeq}\widetilde{\bm{H}}_{\dR,x}$ be the corresponding $\widetilde{G}$-structure preserving isomorphism. The point $y=p_2(s_0(x))\in\on{M}^{\loc}_G(\overline{\Field}_p)$ corresponds to the isotropic line $\xi_0^{-1}(F^1\widetilde{\bm{L}}_{\dR,x})\subset L\otimes\overline{\Field}_p$.

  Let $T=\Spec R$ (resp. $T'=\Spec R'$) be the Henselization of $Z^{\on{pr}}_{K^p}(\Lambda)$ at $x$ (resp. $\on{M}^{\loc}_G$ at $y$). It is enough to show that there exists a section $s:T\to\mathcal{P}_{\dR,\Lambda}$ lifting $s_0$ such that the induced map $p_2\circ s:T\to T'$ is an isomorphism. 
  
  Let $\mx_R\subset R$ (resp. $\mx_{R'}\subset R'$) be the maximal ideal, and set $R_1=R/(\mx_R^2+(p))$ (resp. $R'_1=R'/(\mx_{R'}^2+(p))$). Also, set $T_1=\Spec R_1$ and $T'_1=\Spec R'_1$. 
  
  By \cite{rapzink}*{3.33}, it is enough to find a section $s$ as above such that the induced map $T_1\to T'_1$ is an isomorphism.\footnote{In the language of \cite{rapzink}*{\S 3} such a section is `rigid of the first order'.} 
  
By Hensel's lemma, it now suffices to find a section $s_1:T_1\to\mathcal{P}_{\dR,\Lambda}$ lifting $s_0$ such that the induced map $p_2\circ s_1:T_1\to T'_1$ is an isomorphism.

  Composing the isomorphism $\xi_0$ with the obvious identification $\widetilde{\bm{H}}_{\dR,x}\otimes_kR_1=\widetilde{\bm{H}}_{\dR,R_1}$ gives us a $\widetilde{G}$-structure preserving isomorphism
  \[
   \xi_1:\widetilde{H}\otimes R_1\xrightarrow{\simeq}\widetilde{\bm{H}}_{\dR,R_1},
  \]
  which corresponds to a section $s_1:T_1\to\mathcal{P}_{\dR,\Lambda}$ lifting $s_0$. It now follows from (\ref{lifts:prp:nillifts}) that the induced map $q\circ s_1:T_1\to T'_1$ is an isomorphism.
\end{proof}

\begin{rem}
  The proposition gives us a \emph{local model diagram} in the terminology of \cites{rapzink,deligne-pappas,pappas:jag}.
\end{rem}

\begin{corollary}\label{regular:cor:locprps}
$Z^{\on{pr}}_{K^p}(\Lambda)$, and hence $Z^{\on{pr}}_{K^p_0}(\Lambda)$, is lci and flat of relative dimension $n$ over $\Int_{(p)}$. Moreover, $Z^{\on{pr}}_{K^p}(\Lambda)_{\Field_p}$ is reduced if $n\geq t$, where $t$ is the dimension of the radical of $L_{\Field_p}$. It is normal if $n\geq t+1$, and is smooth if $t=0$. In particular, if $n\geq t$, then $Z^{\on{pr}}_{K^p}(\Lambda)$ is normal.
\end{corollary}
\begin{proof}
Follows from (\ref{lattice:lem:mloc}) and (\ref{regular:prp:locmodel}).
\end{proof}

\begin{corollary}\label{regular:cor:healthy1}
Suppose that $L$ is maximal with $t\leq 1$. Then $Z_{K^p}(\Lambda)$ (and hence $Z_{K^p_0}(\Lambda)$) is regular and locally healthy.
\end{corollary}
\begin{proof}
  This is clear from (\ref{lattice:prp:mlochealthy}) and (\ref{regular:prp:locmodel}).
\end{proof}

\subsection{}\label{regular:subsec:teq2}
Assume now that $L$ is maximal with $t=2$. As in (\ref{lattice:subsec:mref}), we will fix a quadratic extension $F/\Rat$ in which $p$ is inert, but we will place the additional constraint that $F$ be \emph{real}. 

Fix a self-dual $\Reg{F,(p)}$-lattice $L^{\diamond}\subset F\otimes L$ containing $\Reg{F,(p)}\otimes L$. By \emph{loc. cit.}, this gives us a proper $G$-equivariant map $\on{M}^{\on{ref}}_G\to\on{M}^{\loc}_G$, whose source is regular and locally healthy.
\begin{prp}\label{regular:prp:healthy2}
There exists an algebraic space $Z^{\on{ref}}_{K^p}(\Lambda)$ over $\Int_{(p)}$ and a proper morphism $Z^{\on{ref}}_{K^p}(\Lambda)\to Z^{\on{pr}}_{K^p}(\Lambda)$, determined uniquely up to unique isomorphism, with the following properties:
\begin{enumerate}
  \item~\label{healthyreg:tis2}There exists a $G$-equivariant isomorphism:
      \[
       \mathcal{P}_{\dR,\Lambda}\times_{Z^{\on{pr}}_{K^p}(\Lambda)}Z^{\on{ref}}_{K^p}(\Lambda)\xrightarrow{\simeq}\mathcal{P}_{\dR,\Lambda}\times_{\on{M}^{\loc}_G}\on{M}^{\on{ref}}_G.
      \]
      Here, the $G$-action on the left hand side is via its action on $\mathcal{P}_{\dR,\Lambda}$, and the action on the right is the diagonal action. 
  \item~\label{healthyreg:tis1locmod}Every point of $Z^{\on{ref}}_{K^p}(\Lambda)$ has an \'etale neighborhood $U\to Z^{\on{ref}}_{K^p}(\Lambda)$ equipped with a section $s:U\to\mathcal{P}_{\dR,\Lambda}\times_{Z^{\on{pr}}_{K^p}(\Lambda)}Z^{\on{ref}}_{K^p}(\Lambda)$ such that the composition
      \[
       U\xrightarrow{s}\mathcal{P}_{\dR,\Lambda}\times_{Z^{\on{pr}}_{K^p}(\Lambda)}Z^{\on{ref}}_{K^p}(\Lambda)\xrightarrow{\simeq}\mathcal{P}_{\dR,\Lambda}\times_{\on{M}^{\loc}_G}\on{M}^{\on{ref}}_G\rightarrow\on{M}^{\on{ref}}_G
      \]
      is \'etale. 
\end{enumerate}
In particular, $Z^{\on{ref}}_{K^p}(\Lambda)$ is regular and locally healthy, and the map $Z^{\on{ref}}_{K^p}(\Lambda)\to Z^{\on{pr}}_{K^p}(\Lambda)$ is an isomorphism over the regular locus of the target.
\end{prp}
\begin{proof}
  This is a consequence of (\ref{regular:prp:locmodel}) and \cite{pappas:jag}*{Prop. 2.4}. The statement of the cited result does not apply directly in our setting, but it is easily seen that its \emph{proof} does.
\end{proof}

\begin{rem}\label{regular:rem:linmod}
$Z^{\on{ref}}_{K^p}(\Lambda)\to Z^{\on{pr}}_{K^p}(\Lambda)$ is a \defnword{linear modification} in the sense of \cite{pappas:jag}*{2.6}. Although our construction does not fit strictly within the framework of \emph{loc. cit.}, it is inspired by obvious analogy.
\end{rem}

\subsection{}\label{regular:subsec:artin_invariant}
We will continue with the assumption that $t=2$. As explained in the introduction, in this situation, $Z^{\on{pr}}_{K^p}(\Lambda)$ is strictly contained in $Z_{K^p}(\Lambda)$. In other words, there exist points in $Z_{K^p}(\Lambda)(\overline{\Field}_p)$ lying over $s_0\in\Ss_{\widetilde{K}}(\overline{\Field}_p)$ such that the image of the associated map $\iota_0:\Lambda\to L(\widetilde{A}^{\KS}_{s_0})$ does not generate a saturated sub-space of $\widetilde{\bm{L}}_{\cris,s_0}$. 

To understand this phenomenon, it will be helpful to have a better handle on the structure of $\widetilde{\bm{L}}_{\cris,s_0}$. In the language of~\cite{ogus:ss}*{\S 3}, this is (up to twist) a \defnword{K3 crystal} of rank $\widetilde{n}+2$: That is, it is strongly divisible in the sense of~\eqref{spin:subsec:strongdiv}, and it carries an $\bm{F}_{s_0}$-invariant self-dual quadratic form. 

The structure of such objects is well-understood. For simplicity, set $\widetilde{\bm{L}}=\widetilde{\bm{L}}_{\cris,s_0}$, and $\widetilde{L}_p(s_0)=\widetilde{\bm{L}}^{\bm{F}_{s_0}=1}$. One associates with $\widetilde{\bm{L}}$ a Hodge polygon and a Newton polygon; cf.~(1.2) and (1.3) of~\cite{katz:slope}. Essentially, the Hodge polygon encodes information about the Hodge filtration on $\widetilde{\bm{L}}_{\dR,s_0}$ and the Newton polygon encodes information about the slopes of $\bm{F}_{s_0}$. Both polygons are convex by construction, and by a theorem of Mazur~\cite{katz:slope}*{Theorem 1.4.1}, they have the same end-points, and the Newton polygon lies above the Hodge polygon.

In our situation, the Hodge polygon is very simple: It begins at $(0,0)$ and ends at $(\widetilde{n}+2,0)$, and the slope $0$ segment has length $\widetilde{n}$. This reflects the fact that $F^1\widetilde{\bm{L}}_{\dR,s_0}$ has rank $1$ and that $\gr^F_0\widetilde{\bm{L}}_{\dR,s_0}$ has rank $\widetilde{n}$.

There are now two possibilities for the Newton polygon, with different flavors:
\begin{itemize}
\item The Newton polygon is non-constant: In this case, it admits a break at $x$-co-ordinate $h$, where $h$ is an integer between $1$ and $\left\lfloor\frac{\widetilde{n}+2}{2}\right\rfloor$. We say then that $\widetilde{\bm{L}}$ has \defnword{finite height} $h$.  

By a theorem of Katz~\cite{katz:slope}*{1.6.1}, $\widetilde{\bm{L}}$ admits a Newton-Hodge decomposition:
\[
 \widetilde{\bm{L}} = \widetilde{\bm{L}}_{-h}\oplus \widetilde{\bm{L}}_0 \oplus \widetilde{\bm{L}}_h,
\]
where each of the summands is $\bm{F}_{s_0}$-stable after inverting $p$. The summand $\widetilde{\bm{L}}_{-h}$ has rank $h$ and the quadratic form on $\widetilde{\bm{L}}$ induces a perfect pairing between $\widetilde{\bm{L}}_{-h}$ and $\widetilde{\bm{L}}_h$. The summand $\widetilde{\bm{L}}_0$ is a unit root $F$-crystal (that is, it is generated as a $W$-module by its $\bm{F}_{s_0}$-invariant elements), orthogonal to $\widetilde{\bm{L}}_{-h}\oplus\widetilde{\bm{L}}_h$. 

We have $\widetilde{L}_p(s_0) = \widetilde{\bm{L}}_0^{\bm{F}_s=1}$, and $W\otimes_{\Int_p}\widetilde{L}_p(s_0)=\widetilde{\bm{L}}_0$ is a direct summand of $\widetilde{\bm{L}}$. In particular, any point of $Z_{K^p}(\Lambda)$ lying above $s_0$ will be in $Z^{\on{pr}}_{K^p}(\Lambda)$. 

\item The Newton polygon has constant slope $0$: In this case, we say that $\widetilde{\bm{L}}$ is \emph{supersingular}. By a result of Ogus~\cite{ogus:ss}*{Theorem 3.4}, $\widetilde{\bm{L}}$ admits an orthogonal decomposition:
\[
 \widetilde{\bm{L}} = \widetilde{\bm{L}}_1 \perp \widetilde{\bm{L}}_0,
\]
which is again $\bm{F}_{s_0}$-stable after inverting $p$. In this decomposition, $\widetilde{\bm{L}}_0$ is a unit root $F$-crystal, which is self-dual with respect to the quadratic form.. Moreover, $\widetilde{L}_p(s_0)$ is a free module of rank $\widetilde{n}+2$ over $\Int_p$, and it inherits an orthogonal decomposition:
\[
\widetilde{L}_p(s_0) = \widetilde{L}_1 \perp \widetilde{L}_0,
\]
where $\widetilde{L}_i = \widetilde{\bm{L}}_i^{\bm{F}_{s_0}=1}$, for $i=0,1$. The quadratic form on $\widetilde{L}_1$ is $p$-times a self-dual form. 

We can say more: the sub-space $W\otimes_{\Int_p}\widetilde{L}_p(s_0)\subset\widetilde{\bm{L}}$ has index $p^{\sigma}$, where $2\sigma = \rank\widetilde{L}_1$, and the discriminant of $\widetilde{L}_p(s_0)$ has $p$-adic valuation $\sigma$. The integer $\sigma$, which lies between $1$ and $\left\lfloor\frac{\widetilde{n}+2}{2}\right\rfloor$, is the \emph{Artin invariant} of $\widetilde{\bm{L}}$. 

Now, since $\dual{\Lambda}/\Lambda$ is isomorphic to $\Int/p\Int\oplus\Int/p\Int$, we can find an orthogonal decomposition $\Lambda = \Lambda_1\perp\Lambda_0$, where $\Lambda_0$ is self-dual, $\Lambda_1$ has rank $2$, and the restriction of the quadratic form to $\Lambda_1$ is divisible by $p$. If $2\sigma\geq \widetilde{n}$, then one can check that there can never be an isometric embedding $\Lambda_1\into\widetilde{L}_p(s_0)$ whose image generates a direct summand of $\widetilde{\bm{L}}$. Therefore, any point of $Z_{K^p}(\Lambda)$ lying above $s_0$ will be outside of $Z^{\on{pr}}_{K^p}(\Lambda)$.
\end{itemize}

\subsection{}\label{regular:subsec:shimura_curve}
Consider the special case where $\widetilde{n}=3$, $d=2$, and $\widetilde{L}$ admits a maximal isotropic sub-space of rank $2$: In this situation, $\Ss_{\widetilde{K}}$ is a moduli space of polarized abelian surfaces (the Kuga-Satake abelian scheme $\widetilde{A}^{\KS}$ is isomorphic to a power of the universal abelian surface), and $Z_{K^p}(\Lambda)$ is a union of integral models for compact Shimura curves with non-maximal parahoric level structure. The Artin invariant $\sigma$ is either $1$ or $2$: It is $1$ when $\widetilde{A}^{\KS}_{s_0}$ is isomorphic to a product of supersingular elliptic curves (the \emph{superspecial} case), and it is $2$ if $A^{\KS}_{s_0}$ is supersingular, but not superspecial, which should be the `generic' situation on $Z_{K^p}(\Lambda)$. The above discussion shows that the open locus $Z^{\on{pr}}_{K^p}(\Lambda)$ will not contain any of the non-superspecial points, and so misses out on what should be the most interesting part of the geometry of $Z_{K^p}(\Lambda)$. Similar caveats apply in higher dimensions. We intend to return to the general question of the structure of $Z_{K^p}(\Lambda)$ in future work.

\section{Integral canonical models II}\label{sec:regular}

The notation will be as in the previous section. Recall from (\ref{regular:subsec:zkpprps}) that we have canonical maps $\Sh_{K_0}\to Z_{K^p_0}(\Lambda)_{\Rat}$ and $\Sh_K\to Z_{K^p}(\Lambda)_{\Rat}$.
\begin{lem}\label{regular:lem:openclosed}
Both these maps are isomorphisms onto open and closed sub-schemes of the targets.
\end{lem}
\begin{proof}
By construction, it is enough to prove this for $\Sh_K\to Z_{K^p}(\Lambda)_{\Rat}$. Since both varieties in question are smooth of dimension $n$ and are unramified over $\Sh_{\widetilde{K}}$, it is enough to show that the map is injective on $\Comp$-valued points.

Pick a point $x\in\Sh_K(\Comp)$, and let $(\mb{h},g)\in X\times G(\Adele_f^p)$ be a lift of $x$. We can describe the corresponding point of $Z_{K^p}(\Lambda)(\Comp)$ as follows: First, we can view $(\mb{h},g)$ as a point in $\widetilde{X}\times\widetilde{G}(\Adele_f^p)$; write $\tilde{x}$ for its image in $\Sh_{\widetilde{K}}(\Comp)$. Then $\mb{h}$ induces a Hodge structure $\widetilde{H}_{\mb{h}}$ on $\widetilde{H}$, which preserves $\widetilde{L}\subset\widetilde{H}^{\otimes(1,1)}$, and is such that the induced Hodge structure on $\Lambda\subset\widetilde{L}$ is trivial. There is a $\widetilde{G}$-structure preserving isomorphism of Hodge structures $j_x:\widetilde{H}_{\mb{h}}\xrightarrow{\simeq}\bm{\widetilde{H}}_{B,\tilde{x}}$ and the composition:
\begin{align}\label{openclosed:eqn:1}
 \iota_x:\Lambda\into\widetilde{L}_{\mb{h}}\cap(\widetilde{L}_{\mb{h}}\otimes\Comp)^{0,0}\xrightarrow{j_x}\widetilde{\bm{L}}_{B,\tilde{x}}\cap(\widetilde{\bm{L}}_{B,\tilde{x}}\otimes\Comp)^{0,0}=L(\widetilde{A}^\KS_{\tilde{x}}).
\end{align}
is the $\Lambda$-structure attached to $x$. Further, the $K^p$-level structure on $(\tilde{x},\iota_x)$ attached to $x$ is the $K^p$-orbit $[\widetilde{\eta}_x]$ of the isomorphism:
\begin{align}\label{openclosed:eqn:2}
 \widetilde{\eta}_x:\widetilde{H}\otimes\Adele_f^p\xrightarrow[\simeq]{g}\widetilde{H}\otimes\Adele_f^p\xrightarrow[\simeq]{j_x\otimes 1}\bm{\widetilde{H}}_{B,\tilde{x}}\otimes\Adele_f^p\xrightarrow{\simeq}\bm{\widetilde{H}}_{\Adele_f^p,\tilde{x}}.
\end{align}
Note that the triple $(\tilde{x},\iota_x,[\widetilde{\eta}_x])$ determines the image of $x$ in $Z_{K^p}(\Lambda)(\Comp)$.

Suppose that $x'\in\Sh_K(\Comp)$ lifts to $(\mb{h}',g')\in X\times G(\Adele_f^p)$ and is such that $(\tilde{x},\iota_{x},[\widetilde{\eta}_{x}])=(\tilde{x}',\iota_{x'},[\widetilde{\eta}_{x'}])$. Then, using (\ref{openclosed:eqn:1}) we find that $\mb{h}'=\gamma\cdot \mb{h}$, where $\gamma=j_{x'}^{-1}\circ j_x\in G(\Int_{(p)})$. Also, using (\ref{openclosed:eqn:2}), we see that $\gamma g$ and $g'$ are in the same $K^p$-orbit in $G(\Adele_f^p)$. This shows $x'=x$ and finishes the proof of the lemma.
\end{proof}

\begin{defn}\label{regular:defn:models}
Let $\Ss^{\on{pr}}_K$ be the Zariski closure of $\Sh_K$ in $Z^{\on{pr}}_{K^p}(\Lambda)$. If $L$ is maximal with $t\leq 1$, set $\Ss_K\coloneqq\Ss^{\on{pr}}_K$. If $L$ is maximal and $t=2$, fix $F$ and $L^{\diamond}$ as in (\ref{regular:subsec:teq2}) and let $\Ss^{\on{ref}}_K$ be the proper resolution of $\Ss^{\on{pr}}_K$ obtained by taking the Zariski closure of $\Sh_K$ in $Z^{\on{ref}}_{K^p}(\Lambda)$.
\end{defn}
 
\subsection{}\label{regular:subsec:prodefn}
We will now assume that $L$ is maximal with $t\leq 1$.

We set:
\[
  \Ss_{K_p}=\varprojlim_{K^p\subset\widetilde{K}^p}\Ss_{K_pK^p}.
\]
In this inverse system, for $K^p_1\subset K^p_2$, the map $\Ss_{K_pK^p_1}\to\Ss_{K_pK^p_2}$ is induced by the obvious map $Z_{K^p_1}(\Lambda)\to Z_{K^p_2}(\Lambda)$. \comment{In the case $t=2$, we also need to observe that this latter map lifts canonically to a map $Z^{\on{ref}}_{K^p_1}(\Lambda)\to Z^{\on{ref}}_{K^p_2}(\Lambda)$. In fact, it follows from (\ref{regular:prp:healthy2}) that there is a canonical isomorphism
\[
 Z^{\on{ref}}_{K^p_1}(\Lambda)\xrightarrow{\simeq}Z^{\on{ref}}_{K^p_2}(\Lambda)\times_{Z_{K^p_2}(\Lambda)}Z_{K^p_1}(\Lambda).
\]
}
Then $\Ss_{K_p}$ is actually the projective limit of quasi-projective $\Int_{(p)}$-schemes along finite \'etale morphisms, and so is itself a scheme over $\Int_{(p)}$.

\begin{thm}\label{regular:thm:main}
\mbox{}
\begin{enumerate}
\item $\Ss_{K_p}$ (resp. $\Ss_{K_{0,p}}$) is an integral canonical model for the pro-Shimura variety $\Sh_{K_p}$ (resp. $\Sh_{K_{0,p}}$). The map $\Ss_{K_p}\to\Ss_{K_{0,p}}$ is a pro-\'etale cover.
\item For any neat compact open $K\subset G(\Adele_f^p)$ with $p$-primary part $K_p$, there exists a projective regular $\Int_{(p)}$-scheme $\overline{\Ss}_K$ and a fiberwise dense open immersion $\Ss_K\into\overline{\Ss}_K$. Moreover, $\overline{\Ss}_K$ has reduced special fiber.
\end{enumerate}
\end{thm}
\begin{proof}
By (\ref{regular:lem:openclosed}), (\ref{regular:cor:healthy1}) and (\ref{regular:prp:healthy2}), $\Ss_{K_p}$ is regular and locally healthy. We need to show that it has the extension property. We already know that $\Ss_{\widetilde{K}_p}$ has the extension property by \cite{kis3}*{2.3.8}. So, using~\cite{falchai}*{I.2.7}, we find that $\Ss_{K_p}$ also has the extension property. If $t\leq 1$, we are now done.

Now, let $\Ss_{K_{0,p}}$ be the Zariski closure of $\Sh_{K_{0,p}}$ in
\[
  \varprojlim_{K^p_0\subset G_0(\Adele^p_f)}Z_{K^p_0}(\Lambda).
\]
Then there is a pro-finite, pro-\'etale map $\Ss_{K_p}\to\Ss_{K_{0,p}}$.

To show that $\Ss_{K_{0,p}}$ has the extension property, we can use the argument from the proof of \cite{moonen}*{3.21.4}: Given a regular, locally healthy $\Int_{(p)}$-scheme $S_0$ and a map $f_0:S_{0,\Rat}\to\Sh_{K_{0,p}}$, it follows from the Nagata-Zariski purity theorem that there exists a pro-finite pro-\'etale cover $S$ of $S_0$ and a map $f:S_{\Rat}\to\Sh_{K_p}$ lifting $f_0$. Now, $S$ is also regular and locally healthy, so, by the extension property of $\Ss_{K_p}$, $f$ extends to a map $S\to\Ss_{K_p}$ that descends to a map $S_0\to\Ss_{K_{0,p}}$.

The existence of the regular projective compactification $\overline{\Ss}_K$ is a special case of Theorem 1 of \cite{mp:toroidal}, which says the following: Suppose that we have a symplectic embedding $(G_{\Rat},X)\into (\GSp(U),\mathcal{X}(U))$ of Shimura data, and suppose that $U_{(p)}\subset U$ is a symplectic lattice such that $G(\Int_p)$ is the stabilizer of $U_{(p)}\otimes\Int_p$ in $G(\Rat_p)$. Fix $\mathcal{K}(U)\subset\GSp(U)(\Adele_f)$ such that $K$ maps to $\mathcal{K}(U)$. Let $\mathcal{S}_{\mathcal{K}(U)}$ be the associated Mumford integral model over $\Int_{(p)}$ for the Siegel Shimura variety $\Sh_{\mathcal{K}(U)}$.\footnote{This is defined as a moduli space of polarized abelian schemes as in (\ref{spin:subsec:primetop}).} Let $\Ss$ be the unique normal integral model for $\Sh_K$ over $\Int_{(p)}$ such that $\Sh_K\to\Sh_{\mathcal{K}(U)}$ extends to a finite map $\Ss\to\mathcal{S}_{\mathcal{K}(U)}$: it does not depend on the choice of $\mathcal{K}(U)$. Then $\Ss$ admits a fiberwise dense open immersion $\Ss\into\overline{\Ss}$, where $\overline{\Ss}$ is projective over $\Int_{(p)}$ and with singularities no worse than that of $\Ss$. In particular, $\overline{\Ss}$ is regular (resp. has reduced special fiber) whenever $\Ss$ is regular (resp. has reduced special fiber).

The required hypothesis holds for the model $\Ss_K$ in our situation: We take $U_{(p)}=H$. 

To finish, it is now enough to show that $\Ss_K$ has reduced special fiber. This follows from (\ref{regular:cor:locprps}).
\end{proof}

\begin{rem}\label{regular:rem:smoothpr}
In particular, $\Ss_{K_p}$ and $\Ss_{K_{0,p}}$ are uniquely determined by their generic fibers: they depend only on $L$ and not on the choice of self-dual quadratic space $\widetilde{L}$ containing $L$.
\end{rem}

\begin{corollary}\label{regular:cor:conncomp}
Suppose that $n\geq t$; then, for any neat compact open sub-group $K\subset G(\Adele_f^p)$ with $p$-primary part $K_p$, there is a canonical bijection between the connected components of $\Sh_{K,\overline{\Rat}}$ (resp. $\Sh_{K_0,\overline{\Rat}}$) and $\Ss_{K,\overline{\Field}_p}$ (resp. $\Ss_{K_0,\overline{\Field}_p}$).
\end{corollary}
\begin{proof}
  Since $\Ss_{K_0}$ is a finite \'etale quotient of $\Ss_K$, the result for the latter implies that for the former. For $\Ss_K$, it is enough to check that the components of $\overline{\Ss}_{K,\overline{\Rat}}$ and $\overline{\Ss}_{K,\overline{\Field}_p}$ are in bijection.

  But, in general, given any flat, projective map $X\to\Spec\Int_{(p)}$ with reduced special fiber, there is a canonical bijection between the components of $X_{\overline{\Rat}}$ and $X_{\overline{\Field}_p}$. Indeed, using Stein factorization, we reduce to the case where $X$ is finite flat over $\Int_{(p)}$. Having reduced special fiber implies that $X$ is actually finite \'etale over $\Int_{(p)}$, and so the result is immediate.
\end{proof}

For applications, we will need a different kind of canonical model that makes sense in some non-maximal cases.
\begin{defn}\label{regular:defn:extension}
A pro-scheme $X$ over $\Int_{(p)}$ satisfies the \defnword{smooth extension property} if, for any regular, formally smooth $\Int_{(p)}$-scheme $S$, any map $S\otimes\Rat\to X$ extends to a map $S\to X$.
\end{defn}

Suppose that $L$ is non-maximal with cyclic discriminant. We still have the pro-Shimura variety $\Sh_{K_p}$ over $\Rat$, where $K_p=G(\Int_p)$, with $G$ the smooth group scheme attached to $L$.
\begin{defn}\label{regular:defn:intcanonical}
In analogy with (\ref{spin:defn:intcanonical}), we will define a \defnword{smooth integral canonical model} of $\Sh_{K_p}$ to be a regular, formally smooth model $\Ss^{\on{sm}}_{K_p}$ (resp. $\Ss^{\on{sm}}_{K_{0,p}}$) which has the smooth extension property.
\end{defn}

\begin{prp}\label{regular:prp:smoothcan}
$\Sh_{K_p}$ (resp. $\Sh_{K_{0,p}}$) admits a smooth canonical model $\Ss^{\on{sm}}_{K_p}$ (resp. $\Ss^{\on{sm}}_{K_{0,p}}$) such that the map $\Ss^{\on{sm}}_{K_p}\to\Ss^{\on{sm}}_{K_{0,p}}$ is pro-\'etale.
\end{prp}
\begin{proof}
  By (\ref{lifts:prp:principal})(\ref{principal:regular}) we find that, for any level sub-group $K^p\subset G(\Adele_f^p)$ contained in $\widetilde{K}^p$, the $\Ss_{\widetilde{K}_p\widetilde{K}^p}$-scheme $Z^{\on{pr}}_{K^p}(\Lambda)$ is smooth at all of its $\overline{\Field}_p$-valued points. In particular, $\Ss^{\on{sm}}_{K_pK^p}\coloneqq\Ss^{\on{pr}}_{K_pK^p}$ must be smooth over $\Int_{(p)}$. Now set:
\[
 \Ss^{\on{sm}}_{K_p}\coloneqq \varprojlim_{K^p\subset\widetilde{K}^p}\Ss^{\on{sm}}_{K_pK^p}.
\]

It remains to show that $\Ss^{\on{sm}}_{K_p}$ has the smooth extension property. Since $\Ss_{\widetilde{K}_p}$ has the extension property, by \cite{falchai}*{I.2.7}, the inverse limit
\[
 \varprojlim_{K^p\subset G(\Adele^p_f)}Z_{K^p}(\Lambda)
\]
also has the extension property.

It is therefore sufficient to make the following observation, which is immediate from (\ref{regular:lem:vectbundle})(\ref{vectbundle:smoothext}): Suppose that we are given a map $x:S\to Z_{K^p}(\Lambda)$, with $S$ smooth over $\Int_{(p)}$. Suppose also that the restriction of $x$ to $S_{\Rat}$ factors through $\Sh_{K_pK^p}$. Then $x$ must necessarily factor through $Z^{\on{pr}}_{K^p}(\Lambda)$.

The construction of $\Ss^{\on{sm}}_{K_{0,p}}$ from $\Ss^{\on{sm}}_{K_{0,p}}$ proceeds as in the maximal case.
\end{proof}

\subsection{}\label{regular:subsec:lcrislderham}
Fix $K^p\subset G(\Adele_f^p)$ sufficiently small with image $K^p_0\subset G_0(\Adele_f^p)$. Set $K=K_pK^p$ and $K_0=K_{0,p}K^p_0$ and write $\Ss=\Ss_{K}$ (resp. $\Ss_0=\Ss_{K_0}$) if $L$ is maximal with $t\leq 1$, $\Ss=\Ss^{\on{ref}}_K$ (resp. $\Ss^{\on{ref}}_{K_0}$) if $L$ is maximal with $t=2$; and $\Ss=\Ss^{\on{sm}}_{K}$ (resp. $\Ss_0=\Ss^{\on{sm}}_{K_{0}}$) for non-maximal $L$ with cyclic discriminant. All these schemes are regular and locally healthy.

Set $\bm{L}_{\dR}=\bm{\Lambda}_{\dR}^{\perp}\subset\widetilde{\bm{L}}_{\dR}\rvert_{\Ss}$ and $\bm{L}_{\cris}=\bm{\Lambda}_{\cris}^{\perp}\subset\widetilde{\bm{L}}_{\cris}$. Here, $\bm{\Lambda}_{\cris}$ is the sub-crystal of $\widetilde{\bm{L}}_{\cris}\rvert_{\Ss}$ generated by the crystalline realization of $\bm{\iota}(\Lambda)$.

Suppose that we are given a point $s\in\Ss(k)$. The evaluation of $\bm{L}_{\cris}$ along $\Spec k\into\Spf W(k)$ is a direct summand $\bm{L}_{\cris,s}=\bm{\Lambda}_{\cris,s}^{\perp}\subset\widetilde{\bm{L}}_{\cris,s}$. Moreover, the $F$-isocrystal structure on $\widetilde{\bm{L}}_{\cris}[p^{-1}]$ induces one on $\bm{L}_{\cris,s}[p^{-1}]$.

Set $W=W(k)$, let $E/W_{\Rat}$ be a finite extension, and suppose that $\tilde{s}\in\Ss(\Reg{E})$ is a lift of $s$. Let $\overline{E}/E$ be an algebraic closure and let $\widetilde{s}_{\overline{E}}$ be the geometric generic fiber of $\widetilde{s}$. From the definition of the sheaves, (\ref{spin:prp:kisinpadichodge})(\ref{intphodge:comp}), and the functoriality of the comparison isomorphism, we obtain:
\begin{prp}\label{regular:prp:cris}
There is a canonical isometric comparison isomorphism compatible with additional structures:
\begin{align*}
       \bm{L}_{p,\widetilde{s}_{\overline{E}}}\otimes_{\Int_p}\Bcris&\xrightarrow{\simeq}\bm{L}_{\cris,s}\otimes_{W}\Bcris.
\end{align*}
\end{prp}
\qed

\subsection{}\label{regular:subsec:intkugasatake}
Consider the `intrinsic' Kuga-Satake abelian scheme $A^{\KS}_{\Sh_K}$ over $\Sh_K$: We claim that it extends uniquely to an abelian scheme $A^{\KS}$ over $\Ss$. Indeed, since $\Ss$ is locally healthy and hence healthy, it suffices to show the following:

Let $x\to\Ss$ be a co-dimension $1$-point such that $\Reg{\Ss_K,x}$ is a discrete valuation ring of mixed characteristic $(0,p)$. Set $T=\Spec\Reg{\Ss,x}$. Then $A^{\KS}_{T_{\Rat}}$ has good reduction over $T$.

To show this, we first observe that, for $\ell\neq p$, the $\ell$-adic cohomology $\bm{H}_{\ell}$ of $A^{\KS}_{\Sh_K}$ is related to $\widetilde{\bm{H}}_{\ell}$ via the formula (cf. proof of~(\ref{special:lem:ks2ks})):
\[
 \bm{H}_{\ell}\otimes_{C(L)}C(\widetilde{L})\xrightarrow{\simeq}\widetilde{\bm{H}}_{\ell}\rvert_{\Sh_K}.
\]
Since $\widetilde{A}^{\KS}_{T_{\Rat}}$ extends to an abelian scheme over $T$, the restriction of $\widetilde{\bm{H}}_{\ell}$ over $T_{\Rat}$ is unramified, which implies that the restriction of $\bm{H}_{\ell}$ to $T_{\Rat}$ is also unramified. Our claim now follows from the usual criterion for good reduction of abelian varieties.

\subsection{}\label{regular:subsec:kstoks}
By \cite{falchai}*{I.2.7}, the isomorphism from (\ref{special:lem:ks2ks}) extends to an isomorphism of abelian schemes over $\Ss$:
\begin{align}\label{regular:eqn:kstoks}
i:\widetilde{A}^{\KS}_{\Ss}\xrightarrow{\simeq}\SHom_{C}(\widetilde{C},A^\KS).
\end{align}

If $\bm{H}_{\dR}$ is the degree-$1$ de Rham cohomology of $A^\KS$ over $\Ss$, the de Rham realization of $i$ gives us a canonical isomorphism of vector bundles with integrable connection:
\begin{align}\label{regular:eqn:ksderham}
\bm{i}_{\dR}:\bm{H}_{\dR}\otimes_C\widetilde{C}\xrightarrow{\simeq}\widetilde{\bm{H}}_{\dR}\rvert_{\Ss}.
\end{align}
Similarly, if $\bm{H}_{\cris}$ is the $F$-crystal over $(\Ss_{\Field_p}/\Int_p)_{\cris}$ obtained from the degree-$1$ crystalline cohomology of $A^{\KS}_{\Field_p}$, then we have a canonical $\widetilde{C}$-equivariant, $\Int/2\Int$-graded isomorphism of $F$-crystals:
\begin{align}\label{regular:eqn:kscris}
\bm{i}_{\cris}:\bm{H}_{\cris}\otimes_C\widetilde{C}\xrightarrow{\simeq}\widetilde{\bm{H}}_{\cris}\rvert_{(\Ss_{\Field_p}/\Int_p)_{\on{cris}}}.
\end{align}

It follows from the argument in (\ref{lattice:lem:grpoints}) that we have:
\begin{align}\label{regular:eqn:cfromtildec}
 C=\{z^++z^{-}\in\widetilde{C}^+\oplus\widetilde{C}^{-}=\widetilde{C}:\;v\cdot(z^++z^-)=(z^+-z^-)\cdot v\text{, for all $v\in\Lambda$}\}.
\end{align}

We can view $\widetilde{\bm{H}}_{\cris}^{\otimes(1,1)}$ as the internal endomorphism object $\underline{\End}(\widetilde{\bm{H}}_{\cris})$ within the category of crystals over $\Ss_{\Field_p}$. As such, within it we have the sub-object $\underline{\End}_{\widetilde{C}}(\widetilde{\bm{H}}_{\cris})$ of $\widetilde{C}$-equivariant endomorphisms. Similarly, within $\bm{H}^{\otimes(1,1)}_{\cris}$ we have the sub-object $\underline{\End}_C(\bm{H}_{\cris})$ of $C$-equivariant endomorphisms. Using (\ref{regular:eqn:cfromtildec}) and (\ref{regular:eqn:kscris}), we find that we can exhibit $\underline{\End}_C(\bm{H}_{\cris})$ as the sub-crystal of $\underline{\End}_{\widetilde{C}}(\widetilde{\bm{H}}_{\cris})\rvert_{\Ss_{\Field_p}}$ that consists of endomorphisms that anti-commute with $\bm{\Lambda}_{\cris}$. By its definition, $\bm{L}_{\cris}$ anti-commutes with $\bm{\Lambda}_{\cris}$ within $\underline{\End}_{\widetilde{C}}(\widetilde{\bm{H}}_{\cris})\rvert_{\Ss_{\Field_p}}$. Therefore, we obtain a commutative diagram of embeddings of crystals:
\begin{align}\label{regular:eqn:lembeddings}
\begin{diagram}
  \bm{L}_{\cris}&&\rInto&&\bm{H}_{\cris}^{\otimes(1,1)}\\
  \dTo_{\simeq}&&&&\dInto\\
  \bm{\Lambda}_{\cris}^{\perp}&\rInto&\bm{\widetilde{L}}_{\cris}\rvert_{\Ss_{\Field_p}}&\rInto&\widetilde{\bm{H}}_{\cris}^{\otimes(1,1)}\rvert_{\Ss_{\Field_p}}.
\end{diagram}
\end{align}

\subsection{}
Given (\ref{regular:eqn:lembeddings}), the notions of $\ell$-specialness and $p$-specialness for an endomorphism of $A^\KS$ carry over verbatim from Section~\ref{sec:special}. For any $\Ss$-scheme $T$, denote the space of $\ell$-special endomorphisms of $A^\KS_T$ by $L_{\ell}(A^\KS_T)$. Let $L(A^\KS_T)$ be the space of special endomorphisms, where `special' means $\ell$-special for every $\ell$. As in (\ref{special:prp:lambdaemb}), we obtain:
\begin{prp}\label{regular:prp:specialgood}
For any scheme $T\to\Ss$, there exists a canonical isometry
\[
 L(A^\KS_T)\xrightarrow{\simeq}\bm{\iota}(\Lambda)^{\perp}\subset L(\widetilde{A}^\KS_T)
\]
compatible with all cohomological realizations.
\end{prp}
\qed

\begin{rem}\label{regular:rem:ortho}
Just as in (\ref{lifts:subsec:ortho}), the sheaf of endomorphism algebras $\underline{\End}\bigl(A^\KS\bigr)_{(p)}$ canonically descends to a sheaf $\bm{E}$ over $\Ss_0$. Suppose now that $T_0$ is a scheme over $\Ss_0$. then one can canonically attach to $T_0$ a group of `special endomorphisms' $L(T_0)\subset\bm{E}(T_0)$, such that, if $T_0$ is in fact an $\Ss$-scheme, then $L(T_0)=L(A^\KS_{T_0})$.
\end{rem}

\subsection{}
Given that we have the notion of a special endomorphism of $A^\KS$, for any $m\in\Int^{>0}_{(p)}$, just as in~\eqref{regular:subsec:zkplambda}, we can define the finite unramified schemes $Z_{K^p}(m)\to\Ss$ and $Z_{K_0^p}(m)\to\Ss_0$, which parameterize special endomorphisms of $A^\KS$ of degree $m$. From the construction of the schemes $\Ss$,~\eqref{regular:prp:specialgood}, and~\eqref{lifts:cor:locus}, we find that the deformation theory of a special endomorphism of $A^\KS$ is locally governed by a single equation. In particular, \'etale locally on $\Ss$, $Z_{K^p}(m)$ is actually an effective Cartier divisor. 

\begin{prp}\label{regular:prp:flatness_divisor}
Suppose that $n>2$. Then the schemes $Z_{K^p}(m)$ and $Z_{K_0^p}(m)$ are flat over $\Int_{(p)}$
\end{prp}
\begin{proof}
It is enough to show that $Z_{K^p}(m)$ is flat over $\Int_{(p)}$. Since, \'etale locally on $\Ss$, $Z_{K^p}(m)$ is an effective Cartier divisor, if it were not $\Int_{(p)}$-flat, its image in $\Ss$ would contain an entire component of $\Ss_{\Field_p}$. 

On the other hand, it is clear from the description of the local properties of $\Ss$ via local models that the smooth locus of $\Ss_{\Field_p}$ is a dense open in $\Ss_{\Field_p}$. So it suffices to show that the restriction of $Z_{K^p}(m)$ to the complement of the non-smooth locus in $\Ss_{\Field_p}$ is flat over $\Int_{(p)}$. This follows directly from~\eqref{regular:cor:smooth}
\end{proof}

\end{document}